\documentclass[usegeometry,11pt,english,bibliography=totoc]{extarticle} 

\usepackage{preamble} 
\usepackage{cyrillic}

\begin{document}
\thispagestyle{firststyle}


\maketitle


\begin{abstract}
\begin{spacing}{0.1}
{\small
\noindent \textsc{Abstract.} Let $M$ be a globally hyperbolic manifold with complete spacelike Cauchy hypersurface $\Sigma \subset M$. Building on past and recent works of Bär and Strohmaier, we extend their Fredholm result of the Atiyah-Singer Dirac operator on compact Lorentzian spaces to the case, where $M$ is diffeomorphic to a product of $\Sigma$ with a compact time intervall and the hypersurface is a Galois covering with respect to a group $\Upgamma$. We follow the first approach of both authors in this extended setting, where a well-posedness result of the Cauchy problem for the Dirac operator on non-compact manifolds is needed in preparation. 
After employing von Neumann algebras and further ingredients for Galois coverings, the well-posedness result is specified for the setting of interest, which leads to $\Upgamma$-Fredholmness of the Dirac operator under APS boundary conditions. 

}
\end{spacing}
\end{abstract}


{\small\normalfont\tableofcontents
}

\addtocontents{toc}{\vspace{-3ex}}			
\section{Introduction}
\pagenumbering{arabic}
\justifying
\begin{spacing}{1}
This paper deals with the initial value problem of the Dirac equation on curved spacetime from a global point of view, where the results for the case of smooth inital data 
is extended here to those with Sobolev regularity. This contribution will prove, that under those (spatial) compactly supported initial data the resulting spinor fields are of finite energy, i.e. they are continuous in time and of some Sobolev regularity in space, as it was stated but not concretely proven in \cite{AndBaer}. The proof of the well-posedness for the linear wave equation on globally hyperbolic spacetimes from \cite{BaerWafo} heavily influenced our approach here. Besides this the focus is on an extension of the Lorentzian index theorem for the Dirac operator on globally hyperbolic manifolds with compact Cauchy hypersurfaces to those with specific non-compact hypersurfaces, i.e. a spatial $L^2$-index theorem. The proof for compact hypersurfaces is explained in \cite{BaerStroh}, where this extension has been proposed. 
The well-posedness becomes an important part in proving the Fredholmness of Dirac operators, which is a priori not clear, since Dirac operators on Lorentzian spaces are not elliptic. \\
\\
So far several other contributions to index theory on globally hyperbolic manifolds were achieved, based on the pioneering work of Bär and Strohmaier in \cite{BaerStroh}: next to twisted $\mathsf{Spin}^c$-bundles and more general boundary conditions (\cite{BaerStroh},\cite{BaerHan}) a first generalisation of well-posedness and Fredholmness for non-compact spatial hypersurfaces has been shown by Bravermann \cite{Bravind} for Dirac operators of strongly Callias type, where a self-adjoint bundle map is added to the Dirac operator. This map anticommutes with Clifford multiplication and satisfies growth condiditions, which allow to control the behaviour of sections and operator properties away from a sufficiently large compact set in the interior of the manifold. Furthermore certain product structures were assumed in the complement of this compact set. Recently Bär and Strohmaier published another extension to non-compact hypersurfaces in the form of a local index theorem (see \cite{BaerStroh2}); moreover they showed, that for compact Cauchy boundary Fredholmness is still given even if the Dirac operator on the globally hyperbolic manifold or the associated elliptic Dirac operator on the boundary is not assumed to be self-adjoint. Their proof is based on Feynman parametrices, which were already used in \cite{BaerStroh} as alternative strategy. This treatment furthermore allows to consider any Lorentzian Dirac-type operator and twisting bundles with non-positive definite inner product. Besides non-compactness in spatial direction one could also consider non-compactness in temporal direction: Shen and Wrochna \cite{shenwroch} showed an index theorem for asymptotically static spacetimes, i.e. the Lorentzian metric decays to a product type metric at past and future timelike infinity. We refer to this recent work and \cite{BaerStroh} for former and other contributions. \\ 
\\
Our general setting is as follows: let $(M,\met)$ be a $(n+1)$-dimensional globally hyperbolic spin manifold. The analysis is focused on the (Atiyah-Singer) Dirac operator $\Dirac$. Any spin bundle $\spinb(M)$, associated to $TM$, is a bundle of modules over a Clifford bundle, which comes with several wanted features; see \cite{LawMi} for details. Any Cauchy hypersurface $\Sigma$ is non-compact, but a complete submanifold. The family $\SET{\Sigma_t}_{t\in \timef(M)}$ corresponds to a family of Cauchy hypersurfaces in $M$, which are level sets of the Cauchy temporal function $\timef$. The space of finite $s$-energy spinors $FE^s_\scomp(M,\timef,\Dirac)$ corresponds to spatial compactly supported continuous sections of a bundle of $s$-Sobolev spaces on the time domain $\timef(M)$, having compact support on each hypersurface and their image under $\Dirac$ is locally $L^2$ in time and Sobolev in space. Sections of $FE^s_\scomp(M,\kernel{\Dirac})$ are those in $FE^s_\scomp(M,\timef,\Dirac)$, which are in the kernel of $\Dirac$. For explicit defintions see \cref{finensec} and Definition \ref{finensolkern}. 
Our first main result claims the well-posedness of the Cauchy problem by saying, that $\Dirac$ and the restriction operator $\rest{t}$ for any $t \in \timef(M)$ are isomoporphisms between sections of finite energy spinors in $M$ and their initial values on $\Sigma$:
\begin{theo}\label{maintheo}
Let $(M,\met)$ be a $(n+1)$-dimensional globally hyperbolic spin manifold, $\Dirac$ the Atiyah-Singer Dirac operator and $\Sigma$ a complete, but possibly non-compact Cauchy hypersurface; for all $s\in \R$ and any $t\in \timef(M)$ the mappings
\begin{eqnarray*}
\mathsf{res}_t \oplus \Dirac &:& FE^s_\scomp(M,\timef,\Dirac) \,\,\rightarrow\,\, H^s_\comp(\spinb(M)\vert_{\Sigma_t})\oplus L^2_{\loc,\scomp}(\timef(M),H^s_\loc(\spinb(M)_{\Sigma_\bullet}))\\
\text{and}\,\,\mathsf{res}_t  &:& FE^s_\scomp(M,\kernel{\Dirac}) \,\,\rightarrow\,\, H^s_\comp(\spinb(M)\vert_{\Sigma_t})
\end{eqnarray*}
are isomorphisms of topological vector spaces.
\end{theo} 
The well-posedness is stated from an operational point of view, which is favoured in this paper, because of its useful interpretation in proving Fredholmness. From the viewpoint of PDE analysis the first line can be rephrased as follows: given an initial value $u_0 \in H^s_\comp(\spinb(M)\vert_{\Sigma_{t_0}})$ at any initial time $t_0 \in \timef(M)$ and inhomogeneity $f \in L^2_{\loc,\scomp}(\timef(M),H^s_\loc(\spinb(M)\vert_{\Sigma_\bullet})$, the Cauchy problem
\begin{equation*}
\Dirac u = f \quad\text{with}\quad u\vert_{\Sigma_{t_0}}=\rest{t_0}(u)=u_0
\end{equation*}
has a unique solution $u \in FE^s_\scomp(M,\timef,\Dirac)$; the second line can be interpreted analogously with $f=0$ and $u \in FE^s_\scomp(M,\kernel{\Dirac})$. This result generalizes Theorem 2.1 in \cite{BaerStroh} to non-compact, complete hypersurfaces as well as Theorem 4 in \cite{AndBaer} to arbitrary Sobolev regularity. The calculations and proofs are done in detail for the case, where the number of spatial dimensions $n$ is odd and thus the spinor bundle is $\Z_2$-graded, see Theorem \ref{inivpwell}, Corollary \ref{homivpwell} and Theorem \ref{cauchyinvdneghom}. This allows to distinguish between spinors of positive and negative chirality, inducing a splitting of the spin bundle into $\spinb^{\pm}(M)$ and the Dirac operator into two Dirac operators $D_{\pm}$, mapping from spinor fields with positive chirality to those of negative chirality and vice versa. The main result follows from the well-posedness with respect to $D_{\pm}$, where the details are only shown for $D_{+}$. From this a similar result for $\Dirac$ with $n$ even (no $\Z_2$-grading) can be obtained.\\
\\
The second main result follows by specifying the non-compactness of the foliating hypersurface: suppose, that $M$ is given as above, but $\Sigma$ is a Galois covering, such that the quotient $\quotspace{\Sigma}{\Upgamma}$ with a discrete as well freely acting group $\Upgamma$ is a smooth, but closed manifold. Restricting to manifolds, which are temporaly compact, i.e. $\timef(M)$ is a closed intervall in $\R$, yields, that $M$ has two spacelike boundary hypersurfaces. Fredholmness of $D_{\pm}$ under (anti-) Atiyah-Patodi-Singer boundary conditions (APS and aAPS) on these Cauchy boundaries can be stated and proven in the framework of Hilbert $\Upgamma$-modules and $\Upgamma$-morphisms, where $\Dirac$ as well as $D_{\pm}$ are commuting with the left action representation of $\Upgamma$ and thus can be interpreted as lifts of their corresponding operators on the closed manifold. 
The explicit results and necessary definitions are given in chapter \ref{chap:atiyah} as Theorem \ref{indexDaAPS} and Theorem \ref{indexDaAPSneg}, which we sum up as follows:
\begin{theo}\label{maintheoII}
Let $M$ be a $(n+1)$-dimensional temporal compact, globally hyperbolic spatial $\Upgamma$-manifold with spin structure, $\spinb^{\pm}(M)\rightarrow M$ the $\Upgamma$-spin bundles of positive/negative chirality; the $\Upgamma$-invariant Dirac operators $\Dirac\vert_{\mathrm{(a)APS}}:FE^0_{\Upgamma,\textrm{(a)APS}}(M,\timef,\spinb(M))\rightarrow L^2_\Upgamma(\spinb(M))$ under (anti-) APS boundary conditions on the Cauchy $\Upgamma$-hypersurfaces $\Sigma_1=\Sigma_{t_1}$ and $\Sigma_2=\Sigma_{t_2}$ with closed base are $\Upgamma$-Fredholm with $\Upgamma$-indices
\begin{equation*}
\Index_{\Upgamma}(\Dirac\vert_{\mathrm{(a)APS}})=\left\lbrace\begin{matrix}
\Index_{\Upgamma}(D_{\mathrm{(a)APS}}) & & n\,\,\text{even} \\
&\text{for}& \\
\Index_{\Upgamma}(D_{+}\vert_{\mathrm{(a)APS}}) + \Index_{\Upgamma}(D_{-}\vert_{\mathrm{(a)APS}}) && n\,\, \text{odd}
\end{matrix}\right. \,\,.
\end{equation*}
\end{theo}
The $\Upgamma$-indices $\Index_{\Upgamma}(D_{\pm,\mathrm{(a)APS}})$ are going to be related to the indices of certain spectral parts of a Dirac-wave evolution operator, which enables a connection to the well-posedness result. If $n$ is even, $\Upgamma$-Fredholmness will be concluded from the $\Upgamma$-Fredholmness of $D=D_{+}$ under (anti-) APS boundary conditions. We follow and modify the first proof strategy of \cite{BaerStroh}, where a well-posedness result of the Cauchy problem in the $\Upgamma$-setting leads to an associated Dirac-wave evolution operator. The techniques for proving Theorem \ref{maintheo} and a detailed analysis of the wave evolution operator in the general non-compact case will become helpful in order to get these statements for Cauchy hypersurfaces, which are Galois coverings. The decomposition of spectral subspaces with respect to the elliptic hypersurface Dirac operator leads to a decomposition of the evolution operator into four matrix entries. 
The off-diagonal elements in the $\Upgamma$-setting imply $\Upgamma$-Fredholmness of the diagonal elements, as they provide initial parametrices without ellipticity. This $\Upgamma$-Fredholmness carries over to the full Dirac operator under APS and aAPS boundary conditions, which also shows, that the $\Upgamma$-indices coincide with the $\Upgamma$-indices of the diagonal matrix entries of the evolution operator. \\
\\
The treatment of these two problems is organized as follows: chapter \ref{chap:lorentz} is meant for fixing important notations and to recapitulate the geometric setup of globally hyperbolic manifolds and function spaces on manifolds in general, where a closer look on Sobolev spaces on Riemannian manifolds has been taken. At the end of this introductory part the spaces of relevance for the proof of the first main result are introduced in a similar manner as in \cite{BaerWafo} and \cite{BaerStroh}. Chapter \ref{chap:dirac} deals with Dirac operators in general and then especially with their representation along Cauchy hypersurfaces. Both chapters contain some necessary preparatory work. Chapter \ref{chap:cauchy} contains three parts, where the first two are focused on the Dirac equation with respect to $D_{+}$; the last section is dedicated to recapitulate the same approach for $D_{-}$ and finally $\Dirac$. As important intermediate step an energy estimate and its consequences are derived in the first section. Afterwards we follow the same strategy as in \cite{BaerStroh}: 
in chapter \ref{chap:feynman} the well-posedness of the homogeneous Dirac equation implies also a wave evolution operator $Q$, mapping compactly supported Sobolev sections over one hypersurface to a compactly supported Sobolev-sections over another hypersurface. This operator is again well defined as Fourier integral operator between non-compact hypersurfaces. The appendix of this paper contains a closer look on this kind of operator and their application in hyperbolic initial value problems. Chapter \ref{chap:galois} gives an introduction into the geometry and (functional) analysis of Galois coverings and clarify, how to implement these special kind of non-compactness in the framework of global hyperbolicity. 
The last chapter \ref{chap:atiyah} brings together all ingredients, 
converted to the $\Upgamma$-setup, and the results follow with a similar, but modified argumentation. The analysis of the Fredholmness of $D_{\pm}$ under either APS or aAPS boundary conditions is done separately, which implies Theorem \ref{maintheoII} for odd and even dimensional Cauchy boundary. \\
\\
In a coming paper a concrete index formula in the $\Upgamma$-setting will be derived for this and several modifications, e.g. more general spin bundles or more general boundary conditions. Well-posedness and $\Upgamma$-Fredholmness are expected to hold with a minor change of the arguments.\\
\\
\textit{Acknowledgements.} The author would like to thank Boris Vertman and Daniel Grieser for their supervising and topical support in the process of working out these results. Moreover I would like to thank Alexander Strohmaier and Christian Bär, who encouraged to extend their work, allowing the author to contribute to the increasing research field of index theory on Lorentzian spaces and its applications. 
\end{spacing}


\addtocontents{toc}{\vspace{-3ex}}
\section{Geometric setup \& function spaces}\label{chap:lorentz}
This section gives a basic background of globally hyperbolic manifolds. Details of Lorentzian geometry in general can be found in several references, see e.g. \cite{ONeill}. After this introduction definitions and properties of several functional analytic spaces on a manifold will be repeated. Afterwards the focus is on several function spaces, especially constructed for globally hyperbolic spacetimes. In order to investigate Sobolev regularity on spacelike non-compact Cauchy hypersurfaces, a repetition of some basics about Sobolev spaces on non-compact Riemannian manifolds are prepared.

\subsection{Globally hyperbolic manifolds}
\justifying
Let $(M,\met)$ be a $(n+1)$-dimensional time-oriented Lorentzian manifold with pseudo-Riemannian metric $\met$; the signature chosen throughout this paper is $(-,+,+,\dots,+)$. The letter $n \in \N$ counts the number of spatial dimensions. Later on we further assume, that $M$ is orientable, connected and a Lorentzian spin manifold.
A subset $\Sigma \subset M$ is called \textit{Cauchy hypersurface}, if $\Sigma$ is a smooth, embedded hypersurface and every inextendable timelike curve in $M$ meets $\Sigma$ exactly once. If $M$ admits several Cauchy hypersurfaces, then all of them are homeomorphic to each other. 
A time-oriented Lorentzian manifold is called \textit{globally hyperbolic}, if and only if it contains a Cauchy hypersurface; see \cite{Geroch70} Theorem 11. Geroch as well as later on Bernal and S\'{a}nchez proved several results in order to classify these kinds of Lorentzian spaces and their properties: if a Cauchy hypersurface $\Sigma$ is fixed, one can find a function on $M$, which level sets are foliating the spacetime, where $\Sigma$ is one of them:
\begin{theo}[\cite{BerSan2006} Theorem 1.2]\label{theo21-1}
Suppose $(M,\met)$ is a globally hyperbolic manifold and $\Sigma$ a spatial Cauchy hypersurface; there exists a function $\timef \in C^\infty(M,\R)$, such that
\begin{itemize}
\item[a)] $\Sigma_{t_0}=\Sigma$ for a fixed chosen $t_0 \in \R$ and
\item[b)] $\Sigma_t:=\timef^{-1}(t)$ is a Cauchy hypersurface $\forall\, t \in \R\setminus\SET{t_0}$, if nonempty.
\end{itemize}
\end{theo}
Geroch's topological splitting theorem says, that any globally hyperbolic manifold $M$ is homeomorphic to $\R\times \Sigma$. The following result shows, that the manifold is even more isometrically related to this product manifold:
\begin{theo}[\textit{Geroch's splitting theorem}, \cite{BerSan2005} Thm.1.1 \& \cite{Geroch70} Property 7]\label{theo22-1}
$(M,\met)$ is isometric to the product manifold $\R\times\Sigma$ with Lorentzian metric 
\begin{equation*}
\met = - N^2 \differ \timef^{\otimes 2} + \met_\timef \quad , 
\end{equation*} 
where $\timef$ is a surjective smooth function on $M$, $N\in C^\infty(M,\Rpos)$ and $\met_\timef$ is a smooth one-parameter family of smooth Riemannian metrics on $\Sigma$, satisfying
\begin{itemize}
\item[(a)] $\mathrm{grad}(\timef)$ is a past-directed timelike gradient on $M$,
\item[(b)] each hypersurface $\Sigma_t$ for fixed $t$ is a spacelike Cauchy hypersurface with Riemannian metric $\met_\timef$, where $\Sigma_{t_0}:= \Sigma$, and
\item[(c)] $\Span{}{\mathrm{grad}\vert_p(\timef)}$ is orthogonal to $T_p \Sigma_t$ with respect to $\met_T\vert_p$ at each $p \in \R\times\Sigma$.
\end{itemize}
\end{theo}
$\timef$ is referred as \textit{Cauchy temporal function}. The time domain for a certain globally hyperbolic manifold is denoted by $\timef(M)$. The existence of such a function ensures, that each level set can be interpreted as a slice $\SET{t}\times\Sigma$, which from now is meant by $\Sigma_t$. The result can be extended to non-spacelike, non-smooth or achronal, but at least non-acausal Cauchy hypersurfaces; see \cite{BerSan2006} for more details. Theorem \ref{theo22-1} furthermore suggests, that along $\mathrm{grad}(\timef)$ the spacetime is foliated by these level sets, wherefore the function $N$ is called \textit{lapse function (of the foliation)}. 
If $\timef(M)$ does not contain any critical points of $\timef$, then each level set is regular and the regular level set theorem ensures, that $\Sigma_t$ is a closed embedded submanifold of codimension one. Thus the embedding 
$\inclus_t:\Sigma_t\,\hookrightarrow\,M$ is a proper map. In the following we will use $\differ t$ and $\partial_t$ instead of $\differ \timef$ and $\mathrm{grad}(\timef)$ to stress the time differentials/derivatives as coordinate (co-) vector with respect to a hypersurface $\Sigma_t$ for $t\in \timef(M)$. \\
\\
An alternative defintion of global hyperbolic manifolds is given in terms of causal sets: for any $p \in M$ define
\begin{equation*}
\Jlight{\pm}(p):=\SET{q \in M \,\vert\, \exists\,\ \text{causal past (\textminus)/future (+) directed curve}\,\, \upgamma : p \,\rightsquigarrow\, q} \quad ;
\end{equation*}
for any subset $A \subset M$. Put $\Jlight{\pm}(A):=\bigcup_{p \in A}\Jlight{\pm}(p)$ and $\Jlight{}(A):=\Jlight{+}(A)\cup\Jlight{-}(A)$, where the latter one is known as \textit{causal domain} 
or \textit{light cone} of $A$. In comparison the \textit{domain of dependence} (also known as \textit{Cauchy developement}) is defined as $\domdep{}(A)=\domdep{+}(A)\cup \domdep{-}(A)$ for $A \subset M$, such that $A$ is achronal, i.e. no two contained points can be connected by a timelike curve. Here
\begin{equation*}
\domdep{\pm}(A):=\SET{p \in M\,\vert\,\text{every past (+)/future (\textminus)\,\,inextendible causal curve through}\,\,p\,\,\text{meets}\,\,A}\,.
\end{equation*}
In particular $A \subset \domdep{\pm}(A)\subset\Jlight{\pm}(A)$. Both concepts are used to rephrase global hyperbolicity in terms of causal sets: 
\begin{theo}[\cite{BerSan2007} Theorem 3.2]\label{theo2-2}
Given a time-oriented Lorentzian manifold $(M,\met)$; the following claims are equivalent to each other:
\begin{enumerate}
\item $(M,\met)$ is globally hyperbolic.
\item $(M,\met)$ satisfies
\begin{itemize}
\item[(a)] $\Jlight{+}(p)\cap\Jlight{-}(q)$ is compact for all $p,q \in M$,
\item[(b)] $(M,\met)$ is causal, i.e. has no causal loops, and
\item[(c)] $(M,\met)$ is strongly causal, i.e. given a neighborhood $\mathcal{U}_p$ for any $p \in M$ there exists a smaller neighborhood $\mathcal{V}_p \subset \mathcal{U}_p$, containing $p$, such that any causal future-directed or past-directed curve on $M$ with endpoints in $\mathcal{V}_p$ is entirely contained in $\mathcal{U}_p$. 
\end{itemize}
\end{enumerate}
\end{theo}
The causal domain will appear in the well-posedness of the Dirac equation in a similar way as for the wave equation in \cite{BaerWafo}. A subset $A \subset M$ is called \textit{spatially compact}, if $A$ is a closed subset and there exists $K \subset M$ compact, such that $A \subset \Jlight{}(K)$. The intersection of a spatially compact subset with any Cauchy hypersurface is compact. In contrast to this definition one calls the whole manifold $M$ spatially compact, if every Cauchy hypersurface of $M$ is compact.
A notion of timelike compactness can be introduced as well: a closed subset $A \subset M$ is \textit{future/past compact}, if $A\cap \Jlight{\pm}(K)$ is compact for every $K \subset M$ compact; it is called \textit{temporal compact}, if $A$ is both future and past compact. In contrast we call the whole manifold $M$ temporal compact, if $\timef(M)$ is a closed intervall. This is equivalent by saying, that there exists a Cauchy hypersurface $\Sigma$, such that $A \subset \Jlight{\pm}(\Sigma)$. See \cite{Sanders2013} or \cite{BGP07} for more details.\\
\\
For each slice $\Sigma_t$ denote by $\SET{e_i(t)}_{i=1}^n$ a Riemann-orthonormal tangent frame with respect to $\met_t$ and $e_0(t)$ is chosen to be perpendicular to each slice. Then $\SET{e_i(t)}_{i=0}^n$ is a Lorentz-orthonormal tangent frame with respect to $\met$, such that $e_0(t)$ is timelike. We follow the convention from \cite{BaerStroh} and choose this vector to be past-directed. According to Theorem \ref{theo22-1} (c) $e_0(t)$ is parallel to $-\partial_t$. Orthonormality implies for each $t \in \timef(M)$
\begin{equation*}
\met(e_0(t),e_i(t))=0 \,\,\,\forall\, i \in \SET{1,...,n}\quad \text{and} \quad \met(e_0(t),e_0(t))=-N^2 
\end{equation*}
and thus $e_0(t)=-\frac{1}{N}\partial_t$. Because of its often appearance this vector will be denoted by $\upnu$ from now on.
The Levi-Civita covariant derivative on $TM$ splits into a parallel and an orthonormal part to each slice: 
let $X \in \mathfrak{X}(M)$, such that $X_p \in T_pM$ for $p \in \Sigma_t$ and fixed $t$, then the covariant derivative operator $\nabla_X$ on $M$ takes the form
\begin{equation}\label{covsplit}
\nabla_X Y = \nabla_X^{\Sigma_t} Y - \met(\nabla_X \upnu, Y) \upnu = \nabla_X^{\Sigma_t} Y + \met(\wein(X), Y) \upnu \quad ,
\end{equation}
where $\wein$ is the Weingarten map, $\nabla^{\Sigma_t}$ the induced Levi-Civita covariant derivative on the slice $\Sigma_t$ with respect to the related tangent bundle and $Y \in \mathfrak{X}(M)$. With this in mind one is able to prove the following expressions:
\begin{lem}
For $p \in \Sigma_t$ let $X,Y \in T_p\Sigma_t$ and $\upnu$ as above, then the following expressions are fullfilled near the point $p$:
\begin{itemize}
\item[(i)] $\commu{1}{}{\upnu}{X}=0$, 
\item[(ii)] $\upnu$ is autoparallel with respect to $\nabla$, i.e. $\nabla_{\upnu}\upnu = 0$, and
\item[(iii)] $\met(\wein(X),Y)=\frac{1}{2 N}\partial_t \met_t(X,Y)$  .
\end{itemize} 
\end{lem}
\begin{proof}\renewcommand{\qedsymbol}{}  
(i) and (ii) are taken from \cite{BaerGauMor}. For (iii) extend the vectors $X,Y$ to coordinate vector fields on $\Sigma_t$ and lift them to $M$. In Koszul's definition of Levi-Civita connection the Weingarten map is expressed as derivatives on the metric and commutators, which vanish near a point:
\begin{eqnarray*}
-2\met(\wein(X),Y) &=& 
X \met(\upnu, Y) + \upnu \met(X,Y) -  Y \met(\upnu, X) = \upnu \met(X,Y) =-\frac{1}{N} \partial_t \met_t(X,Y) \, .\quad\quad\,\square
\end{eqnarray*}
\end{proof}

The volume form on $M$ and its induced volume forms on each slice are given as follows: in local coordinates $\SET{x^i}_{i=1}^n$ for $\Sigma$
\begin{equation}\label{int0}
\dvol{}:=\sqrt{g_t}N\differ t \wedge \bigwedge_{i=1}^n \differ x^{i} \,\, \text{and} \,\, \dvol{\Sigma_t}:=\inclus_t^{*}\iota_{-\upnu}\dvol{}=\sqrt{g_t} \bigwedge_{i=1}^n \differ x^{i}\quad\text{for}\,\,t\,\,\text{fixed,}
\end{equation}
where $g_t = \det\left(\met_t\right)$ and $\iota_{-\upnu}$ is the interior product by a future timelike vector. 
For the energy estimates an expression for the time-derivative of 
\begin{equation}\label{int1}
I_{f}(t)=\int_{\Sigma_t} f_t \dvol{\Sigma_t}
\end{equation}
is needed, where $f_t$ is any $t$-differentiable function on $\Sigma_t$. In order to compute the derivative with respect to $t$, one applies the same proof for the Lorentzian first variation of area formula by treating spacelike level sets as pushforward of $\Sigma$ via variation maps under compact support. See \cite{maxprin16} and \cite{BBC08} for details.
\begin{lem}\label{lem2-2}
Let $M$ be a time-oriented Lorentzian manifold of dimension $(n+1)$, $\tau \in \timef(M)$ and $\inclus_{\tau} \,:\, \Sigma_{\tau}\,\rightarrow\,M$ a spacelike immersion of a hypersurface with mean curvature $H_{\tau}$. 
The variation under compact support of the integral \cref{int1} at $t=\tau$ for a time-differentiable function $f_t$ on $\Sigma_t$ is
\begin{equation*}
\left. \frac{\differ}{\differ t} I_f(t)\right\vert_{\tau} = \int_{\Sigma_\tau} \phi(p) \left(n H_\tau(p)-\left(\upnu f_{t}\right)\vert_{\tau,p}\right)  \dvol{\Sigma_\tau}(p) \quad ,
\end{equation*}
where $\phi \in C^\infty_c(\Sigma_\tau)$.
\end{lem}

\subsection{Sections and operators on manifolds}
\justifying
This subsection refers to standard literature, e.g. \cite{shubin11},\cite{hoerm1}, \cite{hoerm3} and in special to \cite{BaerWafo}. \\
\\
Let $E\,\rightarrow \, M$ be any $\C-$(anti-) linear smooth vector bundle over a pseudo-Riemannian \\manifold $M$ with metric $\met$. The space of smooth sections\bnote{f1c} is as usually denoted by $C^\infty(M,E)$. We equip the tangent bundle $TM$ and via $\met$ the cotangent bundle $T^{*}M$ with the Levi-Civita connection and the vector bundle comes with a related Koszul connection $\Nabla{E}{}$. If the situation is clear from the context, we won't mention the vector bundle or the specific connection and write $\nabla$. The corresponding connection for the $l$'th covariant derivative $\nabla^l$ is induced by the connection on $(T^{*}M)^{\otimes l}\otimes E$ via $l$-fold application of $\nabla$. Let $K \subset M$ compact, $m \in \N_0$ and $\norm{\cdot}{(T^{*}M)^{\otimes l}\otimes E}$ a norm on the vector bundle $(T^{*}M)^{\otimes l}\otimes E$, for $f \in C^\infty(M,E)$ one defines seminorms
\begin{equation}\label{normveccomp}
\norm{f}{K,m}:= \max_{l \in \SET{0,...,m}} \max_{x \in K} \left\lbrace \norm{(\nabla^l f)(x)}{(T^{*}M)^{\otimes l}\otimes E} \right\rbrace \quad .
\end{equation}
The definition is independent of a concrete choice of different norms and connections on the vector bundle. 
If the support is contained in a closed subset $A \subset M$, one defines 
\begin{equation*}
C^\infty_A(M,E):=\SET{u \in C^\infty(M,E)\,\vert\, \supp{u} \subset A} \quad ,
\end{equation*}
The union over all compact subsets of $M$ defines $C^\infty_{\comp}(M,E)$ as \textit{space of compactly supported sections} of $E$.
This is a locally convex topological vector space with continuous inclusion mapping $C^\infty_K \hookrightarrow C^\infty_{\comp}$ for $K \Subset M$. In a similar way the \textit{space of spatially compactly supported sections} on $E$ is defined via
\begin{equation*}
C^\infty_{\scomp}(M,E):= \bigcup_{\substack{A \subset M\\ A\,\text{spatially}\\ \text{compact}}} C^\infty_A(M,E) \quad .
\end{equation*} 
Linear maps between these two spaces and other locally convex topological vector spaces are continuous, if their restriction to any subspace $C^\infty_A(M,E)$ is continuous, where either $A$ is compact or spatially compact.\\
\\
Elements in $C^{-\infty}(M,E)$, i.e. distributional sections, are introduced and used here as follows: let $E^{*}\,\rightarrow\,M$ denote the dual vector bundle of $E$. 
The common way to define distributions is as $\C$-(anti-) linear functionals on test functions $\phi \in C^\infty_\comp(M,E^{*})$, where the evaluation is denoted by $\dpair{1}{C^\infty(M,E^{*})}{u}{\phi}$. If $u \in L^1_\loc(M,E)$, i.e. a locally integrable section on $E$, a distributional pairing with $u$ can be considered as dual pairing with a regular distribution:
\begin{equation*}
\dpair{1}{C^\infty(M;E^{*})}{u}{\phi}:= \int_M \dpair{1}{E}{\phi}{u}(x)\,\dvol{}(x)
\end{equation*}
with $\dpair{1}{E}{\phi}{u}(x)$ as dual pairing $E^{*}_x \times E_x \,\rightarrow \, \C$ on the fibers at each $x \in M$. We define the space $C^{-\infty}_\comp(M,E)$ of compactly supported distributions as the dual space of smooth functions $(C^\infty(M,E^\ast))'$. If the real- or complex vector bundle $E$ comes with a bundle metric\bnote{f3} $\idscal{1}{E}{\cdot}{\cdot}\,,$ a $L^2$-scalar product
\begin{equation*}
\idscal{1}{L^2(M,E)}{u}{v}:=\int_M \idscal{1}{E}{u}{v}(x) \,\dvol{}(x) 
\end{equation*}
can be introduced for $u, v \in C^\infty_\comp(M,E)$. If the bundle metric is Riemannian or Hermitian, stressed by writing $\dscal{1}{E}{\cdot}{\cdot}$ instead of $\idscal{1}{E}{\cdot}{\cdot}$ , the $L^2$-scalar product, denoted by $\idscal{1}{L^2(M,E)}{\cdot}{\cdot}\, ,$ is positive definite and a norm is induced by $\norm{u}{L^2(M,E)}^2:=\idscal{1}{L^2(M,E)}{u}{u}$. In the pseudo-Riemannian case the corresponding bundle metrics and thus scalar products do not need to be positive definite, wherefore we stress this difference with other notations. \textit{Square integrable sections} 
of $E$ are defined as the completion of $C^\infty_\comp(M,E)$ with respect to this $L^2$-norm:
\begin{equation*}
L^2(M,E):= \overline{C^\infty_\comp(M,E)}^{\norm{\cdot}{L^2(M,E)}} \quad.
\end{equation*}
Let $P:C^\infty_\comp(M,E) \,\rightarrow \,C^\infty_\comp(M,E)$ be a linear operator. Additional structures given by a dual pairing or a bundle metric allow to define the \textit{formally adjoint operator} $P^{\dagger}$ as map $C^{-\infty}(M,E^{*}) \,\rightarrow \,C^{-\infty}(M,E^{*})$, such that $\dpair{1}{C^\infty_\comp(M,E^{*})}{P^\dagger u}{\phi}:=\dpair{1}{C^\infty_\comp(M,E^{*})}{u}{P\phi}$ for $\phi$ as test section. For general distributional sections this becomes the definition, how to apply an operator on distributions. 
The formally adjoint operator with respect to the $L^2$-scalar product is denoted by $P^{*}$ and fullfills $\idscal{1}{L^2(M,E)}{Pu}{v}=\idscal{1}{L^2(M,E)}{u}{P^{*}v}$ with $\supp{u}\cap\supp{\phi}$ compact. An operator is called \textit{formally self-adjoint} or \textit{formally anti/skew-self-adjoint}, if $P^{\dagger}=P$ or $P^\dagger=-P$ respectively and analogous for $P^{*}$. Since the focus is going to be on function spaces, which are related to Hilbert spaces, and the fact, that differential operators are in general unbounded, one needs to distinguish between \textit{self-adjoint} and \textit{essentially self-adjoint} operators. 
For this operators $P$ with domain $\dom{}{P}$ in a densely defined subset of a Hilbert space are considered, mapping into the same or another Hilbert space.
A revision of this terms is given e.g. in \cite{shubinspec}. 
\\
\\
The standard, non-geometric introduction of partial-differential operators is based on local coordinates: given two smooth vector bundles $E,F$ on $M$; a linear map $P:C^\infty(M,E)\,\rightarrow\,C^\infty(M,F)$ is a \textit{linear partial differential operator (with smooth coefficients)}, if it is local and a linear differential operator for any local coordinates and trivializing frames. We refer to $\Diff{m}{}(M,\Hom(E,F))$ as the \textit{set of linear partial differential operators of order m} on $M$ with smooth coefficients from $E$ to $F$. 
An important quantity for any differential operator is the \text{principal symbol}, which is defined by oscillatory testing: for any $\phi \in C^\infty(M)$, such that $\xi=\differ \phi$, the quantity is introduced by
\begin{equation}\label{prinsymbdiff}
\bm{\sigma}_m(P)(x,\xi)u=\lim_{\lambda \rightarrow \infty} \left. \left(\frac{\Imag}{\lambda}\right)^m \expe{-\Imag \lambda \Phi}P\left(\expe{\Imag \lambda \phi}u\right)\right\vert_{x}
\end{equation}
for any section $u$ of $E$. 
\\
\\
Generalizing the Schwartz kernel representation of differential operators on the local level leads to the broader class of pseudo-differential operators ($\Psi$DO) 
on a manifold, which act between two vector bundles: a continuous map $P\,:\,C^\infty_\comp(M,E)\,\rightarrow\,C^\infty(M;F)$ is a \textit{pseudo-differential operator} of order $m$, if for every open set in $M$ with any local coordinates and trivializations the operator $P$ can be related to a $\rank(F)\times\rank(E)$ matrix of pseudo-differential operators of order $m$ on each open set.
The set of all pseudo-differential operators of order $m \in \R$ on $M$ between sections $E$ and $F$ is denoted by $\ydo{m}{}(M,\Hom(E,F))$. One has $\Diff{m'}{}(M,\Hom(E,F))\subset \ydo{m}{}(M,\Hom(E,F))$ for all $m' \in \N_0,\,\, m \in \R$ with $m' \leq m$. The algebra properties for $\ydo{*}{}(M,\Hom(E,F))$ are inherited by the algebra properties from their local description, if in addition the pseudo-differential operator is properly supported: 
An operator $P$, acting between sections of two manifolds $X$ and $Y$, having Schwartz kernel $K_P$, is said to be \textit{properly supported}, if both projections from $X\times Y$ to one factor, restricted to $\supp{K_P}\subset X\times Y $, are proper maps. As a consequence $P$ maps smooth/distributional sections with compact support back into these spaces respectively,
which is a more handy characterization, taken from \cite{shubin11}: an operator $P$ from sections on $X$ to sections on $Y$ is properly supported, if and only if
\begin{eqnarray}
\text{a)}&& \label{propsuppa}\forall\,K\Subset X\,\exists\, K_1 \Subset Y\,:\,\supp{u}\subset K \,\Rightarrow \, \supp{Pu}\subset K_1\quad,\label{propsuppa} \\
\text{b)}&& \label{propsuppb}\forall\,K\Subset Y\,\exists\, K_1 \Subset X\,:\,\supp{v}\subset K \,\Rightarrow \, \supp{P^\dagger u}\subset K_1\quad.\label{propsuppb}
\end{eqnarray}
From this one observes, that the composition of two properly supported operators is again properly supported. Fourier integral operators (FIO) as third important and generalizing class of operators are introduced in the \hyperref[chap:app2]{appendix}, where the notions of distributions will be defined slighly different.

\subsection{Sobolev spaces on Riemannian manifolds}
\justifying
On a Riemannian manifold\bnote{fc1} $\mathcal{M}$ there exists for any vector bundle $E \rightarrow \mathcal{M}$ with Riemannian/Hermitian inner product and connection $\Nabla{E}{}$ as well as for any real number $s\in\R$ a properly supported, classical elliptic $\Psi$DO of order $s$ with strictly positive principal symbol. A special choice, convenient in the definition of Sobolev spaces, is given by a power of the Laplace-type operator $\adNabla{E}{}\Nabla{E}{} + \id\,:\,  C^\infty(\mathcal{M},E)\,\rightarrow\,C^\infty(\mathcal{M},E)$. 
This operator and all $s$-powers
\begin{equation}\label{laps}
\Lambda^s:= \left(\adNabla{E}{}\Nabla{E}{} + \id{} \right)^\frac{s}{2}
\end{equation}
fullfill all stated features and is furthermore essentially self-adjoint 
on $L^2(\mathcal{M},E)$ with positive spectrum, if $\mathcal{M}$ is complete.\\
\\
Sobolev spaces on compact manifolds $\mathcal{M}$ without boundary are going to be introduced first, since they provide a bulding block for the non-compact case after a suitable reinterpretation. 
The $s-$\textit{Sobolev norm} of $u\in C^\infty(\mathcal{M},E)$ is defined as
\begin{equation}\label{sobnorm}
\Vert u \Vert_{H^s(\mathcal{M},E)}:=\Vert \Lambda^s u \Vert_{L^2(\mathcal{M},E)}
\end{equation}
and the norm completion of $C^\infty(\mathcal{M},E)$ defines the \textit{Sobolev space of order} $s$:
\begin{equation*}
H^s(\mathcal{M},E):= \overline{C^\infty(\mathcal{M},E)}^{\Vert\cdot\Vert_{H^s(\mathcal{M},E)}} \quad .
\end{equation*}
The definition depends on neither the choice of the metric nor of the connection, such that all Sobolev norms and spaces with different metrics and connections are equivalent. For $s=m \in \N_0$ the norm \cref{sobnorm} is equivalent to 
\begin{equation}\label{sobnorm2}
\Vert u \Vert_{\mathpzc{H}^m(\mathcal{M},E)}^2:=\sum_{j=0}^m \left\Vert \Big(\Nabla{E}{}\Big)^j u \right\Vert_{L^2(\mathcal{M},E)}^2 \quad,
\end{equation}
where we introduced another notation of this norm for distinguishing reasons. These spaces can be extended to non-compact manifolds under additional conditions and concepts, see e.g. \cite{hebey} for details. 
The starting point of interest are Sobolev sections of order $m$ with respect to the vector bundle $E\rightarrow \mathcal{M}$, which have a fixed support in $K\Subset \mathcal{M}$. They are introduced as norm closure of smooth functions with compact support in $K$ with respect to $\Vert \cdot \Vert_{\mathpzc{H}^m(\mathcal{M},E)}$:
\begin{equation*}
H^m_K(\mathcal{M},E)=\overline{C^\infty_K(\mathcal{M},E)}^{\Vert\cdot\Vert_{\mathpzc{H}^m(\mathcal{M},E)}} \quad.
\end{equation*}
They carry the topology of a Hilbert space. The compactly supported Sobolev sections then follow by taking the union over all compact subsets:
\begin{equation}\label{Hcomp}
H^m_{\comp}(\mathcal{M},E)=\bigcup_{\substack{K \subset \mathcal{M} \\ \text{compact}}} H^m_K(\mathcal{M},E) \quad,
\end{equation}
which makes them a locally convex vector space with strict inductive limit topology from Hilbert spaces. One notices $H^0_\comp(\mathcal{M},E)=L^2_\comp(\mathcal{M},E)$. Local Sobolev sections are regarded as those distributional sections $u$, such that $\phi u \in H^s_{\supp{\phi}}$ for all $\phi$, smooth and compactly supported:
\begin{equation}\label{Hloc}
H^m_{\loc}(\mathcal{M},E):=\SET{u \in C^{-\infty}(\mathcal{M},E)\,\vert\,\phi u \in H^m_{\comp}(\mathcal{M},E) \quad \forall\,\phi \in C^\infty_\comp(\mathcal{M})}\quad.
\end{equation}
The corresponding seminorm $\Vert\cdot\Vert_{H^m_\loc(\mathcal{M},E)}$ is induced by the norm in $H^m_{\supp{\phi}}$, which makes $H^m_\loc$ a Fr\'{e}chet space. Note, that $H^0_{\loc}(\mathcal{M},E)=L^2_\loc(\mathcal{M},E)$. All these spaces are provided with important properties e.g regularity of Sobolev functions after restriction to submanifolds, embeddings or localization, see e.g. \cite{shubin11}.\\
\\
In order to extend the latter three spaces to Sobolev orders $s \in \R$, another approach from \cite{BaerWafo} will be presented: Sobolev sections of $E$ with fixed compact support on a non-compact Riemannian manifold are reinterpreted as Sobolev sections on another Riemannian, but closed manifold. In order to do so, everything is extended to the \textit{double} of a suitable subset of $\mathcal{M}$. We follow the acquisition, made in \cite{BaerWafo}: let $K \subset \mathcal{M}$ be a compact subset of $\mathcal{M}$, containing the support of a function. Take another relatively compact subset $K_1\subset \mathcal{M}$, such that $K_1$ has a smooth boundary $\bound K_1$ and contains $K$ insinde $\mathring{K_1}$. Choose the boundary in such a way, that it is totally geodesic and that a certain curvature condition is fullfilled; see \cite{Mori} for details. Milnor's collar neighborhood theorem ensures the existence of a diffeomorphism, which deforms the metric near the boundary into a product metric. In order to construct the \textit{closed double} of $K_1$ take a copy of it, denoted by $K_2$, and glue both copies together along their common boundary $\bound K_1 = \bound K_2$, which results in a closed manifold:
\begin{equation*}
\widetilde{\mathcal{M}}:= K_1 \cup_{\bound K_1} K_2 := (K_1 \sqcup K_2)\setminus\bound K_1 \quad .
\end{equation*}
During this procedure everything on $K$ is untouched. Any smooth vector bundle $\widetilde{E}\,\rightarrow\,\tilde{\mathcal{M}}$ can be considered as extension of $E\vert_{K_1}$, if $\widetilde{E}\vert_{K_1}=E\vert_{K_1}$. All bundle metrics on $K\subset \widetilde{\mathcal{M}}$ can be extended to smooth bundle metrics on the whole closed double, which is assured by the additional assumptions on $K_1$. Since the Levi-Civita connection is fully determined by the restricted metric, it can be extended to a Levi-Civita connection with respect to the extended metric. Under this modification any smooth section of $E$ with compact support in $K$, can be considered as smooth section of $\widetilde{E}$ over a closed manifold by zero-extension: $C^\infty_K(\mathcal{M},E)\subset C^\infty(\widetilde{\mathcal{M}},\widetilde{E})$. The space $H^s_K(\mathcal{M},E)$ for real powers $s$ is then defined as the completion of smooth, compactly supported functions with respect to the norm $\Vert\cdot\Vert_{H^s(\widetilde{\mathcal{M}},\widetilde{E})}$:
\begin{equation*}
H^s_K(\mathcal{M},E):= \overline{C^\infty_K(\mathcal{M},E)}^{\Vert\cdot\Vert_{H^s(\widetilde{\mathcal{M}},\widetilde{E})}} \quad .
\end{equation*}
The spaces $H^s_\comp$ and $H^s_\loc$ are then defined as in \cref{Hcomp} und \cref{Hloc} for real orders $s$, which completes the extension to arbitrary real Sobolev orders. The advantage of this definition is, that Sobolev sections with compact support can be treated in the same way as Sobolev sections on a closed manifold with all beneficial properties. For $K\subset K'$ one has the inclusion $C^\infty_K(\mathcal{M},E)\subset C^\infty_{K'}(\mathcal{M},E)$, inducing $H^s_K(\mathcal{M},E)\subset H^s_{K'}(\mathcal{M},E)$ as continuous linear inclusion and this implies the continuous inclusion $H^s_K(\mathcal{M},E)\hookrightarrow H^s_\comp(\mathcal{M},E)$. Moreover the inclusion $H^s_\comp(\mathcal{M},E) \subset C^{-\infty}_\comp(\mathcal{M},E) \subset C^{-\infty}(\mathcal{M},E)$ is also continuous according to the embedding theorem of Sobolev.

\subsection{Special sections on globally hyperbolic manifolds}
\justifying
This last part is dedicated to some special sections, which come from the geometric setting of a globally hyperbolic manifold. Their definitions are related to those given in the proof of the well-posedness of the Cauchy problem for the wave equation in \cite{BaerWafo}. Let $M$ be a globally hyperbolic manifold with Cauchy temporal function $\timef:M\,\rightarrow\,\R$ and $E \rightarrow M$ a smooth vector bundle. Take a complete, spacelike Cauchy hypersurface without boundary, such that the family of slices $\SET{\Sigma_t}_{t \in \timef(M)}$ foliates $M$. Fix $s\in \R$ and consider the family $\SET{H^s_\loc(E\vert_{\Sigma_t})}_{t \in \timef(M)}$ as Fr\'{e}chet bundle over $\timef(M)\subset \R$. Each member in the family of slices differs from each other only by the metric $\met_t$, but not topologically. Since different metrics lead to equivalent Sobolev norms each Sobolev space $H^s_K(E\vert_{\Sigma_t})$ for $K\Subset \Sigma$ and consequently $H^s_\comp(E\vert_{\Sigma_t})$ and $H^s_\loc(E\vert_{\Sigma_t})$ are equivalent. We keep the extra $t$ to mark the different metrics and furthermore keep notational compatability with the referred literature. This bundle of topological vector spaces can be globally trivialized as follows: for each $t \in \timef(M)$ a section of this bundle becomes a section in $H^s_\loc(E\vert_{\Sigma_t})$. By parallel transport along the integral curves of the vector field $\mathrm{grad}(\timef)$ these sections can be moved to $H^s_\loc(E\vert_{\Sigma_{\tau}})$ for a fixed $\tau \in \timef(M)$. Support properties and the Sobolev regularity are preserved by this transport, such that this bundle of Fr\'{e}chet spaces becomes diffeomorphic to $\timef(M)\times H^s_\loc(E\vert_{\Sigma_{\tau}})$; see e.g. section 1.7 from \cite{BaerWafo}. Sections of this bundle, which are $l$-times continuously differentiable ($l \in \N_0$), are denoted  by $C^l(\timef(M),H^s_\loc(E\vert_{\Sigma_{\bullet}}))$. The elements can be considered as distributional sections on $M$ via the dual pairing
\begin{equation*}
\dpair{1}{C^{\infty}_\comp(M,E^{*})}{u}{\phi}:=\int_{\timef(M)} \dpair{1}{C^\infty_\comp(E\vert_{\Sigma_t})}{u(t)}{(N\phi)\vert_{\Sigma_t}} \differ t
\end{equation*}
for $u \in C^l(\timef(M),H^s_\loc(E\vert_{\Sigma_{\bullet}}))$ and $\phi \in C^\infty_\comp(M,E^{*})$. The lapse function $N\in C^\infty(M,\R_{>0})$ appears by the volume element $\dvol{}$ in \cref{int1}. For $s\geq 0$ those distributional sections are locally integrable since $H^s_\loc \subset L^2_\loc$ by the continuous embedding of Sobolev spaces. The dual pairing in the integral can be expressed as regular distributional action of the form
\begin{equation*}
\dpair{1}{C^\infty_\comp(E\vert_{\Sigma_t})}{u(t)}{(N\phi)\vert_{\Sigma_t}}=\int_{\Sigma_t}\dpair{1}{(E\vert_{\Sigma_t})^{*}}{u}{(N\phi)\vert_{\Sigma_t}} \dvol{\Sigma_t}
\end{equation*} 
and thus
\begin{equation*}
\dpair{1}{C^{\infty}_\comp(M,E^{*})}{u}{\phi}=\int_{\timef(M)}\left[ \int_{\Sigma_t}\dpair{1}{(E\vert_{\Sigma_t})^{*}}{u}{(N\phi)\vert_{\Sigma_t}} \dvol{\Sigma_t} \right] \differ t = \int_{M} \dpair{1}{E^{*}}{u}{\phi} \dvol{} \quad.
\end{equation*}
This observation confirms, that $C^l(\timef(M),H^s_\loc(E\vert_{\Sigma_{\bullet}}))$ is embedded into $C^{-\infty}(M,E)$. For any compact subintervall $I \subset \timef(M)$ and any spatially compact $K \subset M$ one defines
\begin{equation}\label{scompsuppglobdiff}
C^l_K(\timef(M),H^s_\loc(E\vert_{\Sigma_{\bullet}})):=\SET{u \in C^l(\timef(M),H^s_\loc(E\vert_{\Sigma_{\bullet}}))\,\vert\,\supp{u}\subset K}
\end{equation}
with the seminorm
\begin{equation}\label{snclk}
\Vert u \Vert_{I,K,l,s}:= \max_{k \in [0,l]\cap \N_0} \max_{t \in I} \Vert (\nabla_t)^{k} u\Vert_{H^s_\loc(E\vert_{\Sigma_t})} \quad.
\end{equation}
Varying over all compact subsets $I \subset \timef(M)$ shows for fixed $l$, $K$ and $s$, that \cref{scompsuppglobdiff} is a Fr\'{e}chet space. Taking the union over all spatially compact subset defines sections of this bundle, which have support in any spatially compact subset of $M$:
\begin{equation}\label{clsc}
C^l_{\scomp}(\timef(M),H^s_\loc(E\vert_{\Sigma_{\bullet}})):=\bigcup_{\substack{K \subset M \\ K\,\,\text{spatially}\\ \text{compact}}}C^l_K(\timef(M),H^s_\loc(E\vert_{\Sigma_{\bullet}})) \quad.
\end{equation} 
The inclusion $C^l_K(\timef(M),H^s_\loc(E\vert_{\Sigma_{\bullet}}))\hookrightarrow C^l_\scomp(\timef(M),H^s_\loc(E\vert_{\Sigma_{\bullet}}))$ is continuous and any linear map from $C^l_\scomp(\timef(M),H^s_\loc(E\vert_{\Sigma_{\bullet}}))$ to any locally convex topological vector space is continuous, if and only if the restriction of the mapping to $C^l_K(\timef(M),H^s_\loc(E\vert_{\Sigma_{\bullet}}))$ for any spatially compact subset $K$ is continuous. The case $l=0$ will be of special interest in this paper, wherefore it gets a definition on its own:
\begin{defi}
Let $E \rightarrow M$ be a vector bundle over a globally hyperbolic manifold $M$ with temporal function $\timef$ and foliating family of spatial Cauchy hypersurfaces $\SET{\Sigma_t}_{t \in \timef(M)}$; for any $s \in \R$ 
\begin{equation}\label{finensec}
FE^s_{\scomp}(M,\timef,E):=C^0_{\scomp}(\timef(M),H^s_\loc(E\vert_{\Sigma_{\bullet}}))
\end{equation}
is the \textit{space of finite} $s$-\textit{energy sections}.
\end{defi}

\begin{rems}
\begin{itemize}
\item[]
\item[(a)] This definition occurs in the proof of the index theorem on globally hyperbolic manifolds for compact Cauchy hypersurfaces and their modification for non-compact Cauchy hypersurfaces and other sections than spinors in \cite{BaerWafo}. Only continuous sections on the hypersurfaces are needed, since the Dirac equation is of first order and only one initial condition is necessary. For other differential operators of higher order the definition of finite energy sections has to be modified in order to regard all initial values, e.g. 
\begin{equation*}
FE^s_{\scomp}(M,\timef,E):=C^0_{\scomp}(\timef(M),H^s_\loc(E\vert_{\Sigma_{\bullet}}))\cap C^1_{\scomp}(\timef(M),H^{s-1}_\loc(E\vert_{\Sigma_{\bullet}}))
\end{equation*} 
in \cite{BaerWafo} for the wave operator on globally hyperbolic manifolds with two initial values on an initial hypersurface. 

\item[(b)] This sections in general depend on the Cauchy temporal time function. But a certain subset of finite $s$-energy sections do not have this dependence, which becomes clear in chapter \ref{chap:cauchy}. 
\end{itemize}
\end{rems}

In \cref{clsc} the differentiability with respect to the time coordinate can be weakened to local square integrability:
\begin{defi}
Let $E \rightarrow M$ be a vector bundle over a globally hyperbolic manifold $M$ with temporal function $\timef$ and foliating family of spatial Cauchy hypersurfaces $\SET{\Sigma_t}_{t \in \timef(M)}$ and $K \subset M$ spatially compact; sections of $L^2_{\loc,K}(\timef(M),H^s_\loc(E\vert_{\Sigma_{\bullet}}))$ consists of those $u$, such that
\begin{itemize}
\item[1)] $\supp{u} \subset K\cap\Sigma_t$ for almost all $t\in \timef(M)$ ;
\item[2)] $t\,\mapsto\,\dpair{1}{C^\infty_\comp((E\vert_{\Sigma_t})^{*})}{u}{\phi\vert_{\Sigma_t}}$ is measurable for any $\phi\in C^\infty_{\comp}(M,E)$ and
\item[3)] $t\,\mapsto\,\Vert u \Vert_{H^s_\loc(E\vert_{\Sigma_t})}$ is in $L^2_\loc(\timef(M))$
\end{itemize}
for any $s \in \R$.
\end{defi}
By a similar argument as for $C^l_K(\timef(M),H^s_\loc(E\vert_{\Sigma_{\bullet}}))$ one can prove, that the embedding
\begin{equation*}
L^2_{\loc,K}(\timef(M),H^s_\loc(E\vert_{\Sigma_{\bullet}}))\hookrightarrow C^{-\infty}(M,E)
\end{equation*}
is continuous. In order to topologize this space one introduces the seminorms for $I$ any compact subintervall in $\timef(M)$:
\begin{equation*}
\Vert u \Vert^2_{I,K,s}:=\int_{I} \Vert u \Vert^2_{H^s_\loc(E\vert_{\Sigma_t})} \differ t \quad .
\end{equation*}
This turns $L^2_{\loc,K}(\timef(M),H^s_\loc(E\vert_{\Sigma_{\bullet}}))$ into a Fr\'{e}chet bundle. In a similar way
\begin{equation*}
L^2_{\loc,\scomp}(\timef(M),H^s_\loc(E\vert_{\Sigma_{\bullet}}))=\bigcup_{\substack{K \subset M \\ K\,\,\text{spatially} \\ \text{compact}}} L^2_{\loc,K}(\timef(M),H^s_\loc(E\vert_{\Sigma_{\bullet}})) \quad.
\end{equation*}
is introduced as space of locally square-integrable sections of the Fr\'{e}chet bundle with spatially compact support. Lemma 2 from \cite{BaerWafo} shows, that $C^\infty_{\scomp}(M,E)$ is a dense subset of $L^2_{\loc,\scomp}(\timef(M),H^s_\loc(E\vert_{\Sigma_{\bullet}}))$.

\addtocontents{toc}{\vspace{-3ex}}
\section{The Dirac equation on globally hyperbolic manifolds}\label{chap:dirac}
This chapter starts with some general informations, before introducing special features for the case of interest. References like \cite{LawMi} and \cite{Gin} provide more details.

\subsection{General aspects}
Suppose first, that the number of spatial dimensions $n \in \N$ is odd. Let $(M,\met)$ be a $(n+1)$-dimensional globally hyperbolic Lorentzian spin manifold 
with Cauchy temporal function $\timef$, such that $M$ is foliated by a family of non-compact spatial Cauchy hypersurfaces $\SET{\Sigma_t}_{t\in \timef(M)}$ without boundary. Denote by $\spinb(M) \rightarrow M$ the Hermitian spinor bundle, associated to the tangent bundle $TM$ (i.e. $\spinb(M):=\spinb(TM)$) with indefinite inner product $\idscal{1}{\spinb(M)}{\cdot}{\cdot}$ . Clifford multiplication between vector fields is inherited from the underlying Clifford algebra structure: let $\mathbf{c}$ denote a pointwise acting vector space homomorphism from the Clifford algebra of $T_pM$ 
to $\group{End}_{\C}(\spinb_p(M))$, then Clifford multiplication satisfies
\begin{equation}\label{cliffmult}
\cliff{X}\cdot\cliff{Y}+\cliff{Y}\cdot\cliff{X}=-2\met_p(X,Y)\id{\spinb_p(M)}
\end{equation} 
for any $X,Y \in T_pM$ and $p \in M$. For two elements $u,v \in \spinb_p(M)$, called \textit{spinors}, the action of $\cliff{X}$ on such an element is formally self-adjoint with respect to the inner product: 
\begin{equation}\label{cliffmult2}
\idscal{1}{\spinb_p(M)}{\cliff{X}u}{v}=\idscal{1}{\spinb_p(M)}{u}{\cliff{X}v}\quad \forall\, X \in T_pM \quad .
\end{equation}
The consequence of both relations is shown in the expression
\begin{equation*}
\idscal{1}{\spinb(M)}{\cliff{X}u}{\cliff{X}v}=\idscal{1}{\spinb(M)}{u}{\cliff{X}\cdot\cliff{X}v}=-\met(X,X)\idscal{1}{\spinb(M)}{u}{v} \quad.
\end{equation*}
Thus, if $X$ is timelike and normalized with respect to $\met$, Clifford multiplication acts as an isometry. Let $e_0,e_1,...,e_n$ be a Lorentz-orthonormal tangent frame. The \textit{spinorial volume form} is $\omega_{\spinb(M)}= \Imag^{(n+1)/2}\cliff{e_0}\cdot\cliff{e_1}\cdot\dots\cdot\cliff{e_n}$ and satisfies $\omega_{\spinb(M)}^2=\id{\spinb(M)}$ as well as $\cliff{X}\omega_{\spinb(M)}=-\omega_{\spinb(M)}\cliff{X}$ for all $X \in \mathfrak{X}(M)$, since $(n+1)$ is even. These properties induce a \textit{chirality decomposition} as an eigenspace decomposition of $\omega_{\spinb(M)}$:
\begin{eqnarray*}
\spinb(M)&=&\spinb^{+}(M)\oplus\spinb^{-}(M) \quad \text{with} \\
\spinb^{\pm}(M)&:=& \SET{u \in C^\infty(\spinb(M))\,\vert\, \omega_{\spinb(M)}u=\pm u}=(1\pm \omega_{\spinb(M)})\spinb(M)\quad .
\end{eqnarray*} 
Sections of $\spinb^{\pm}(M)$ are called \textit{spinor fields of positive} and \textit{negative chirality} respectively and are eigenspinors of the spinorial volume form. The spinor bundle becomes $\Z_2$-graded.\\
\\
A connection on the spinor bundle can be introduced as connection on the underlying principal bundle, see e.g. \cite{LawMi}: the Hermitian spinor bundle $\spinb(M)$ comes with a spinorial Levi-Civita connection $\Nabla{\spinb(M)}{X}$, which satisfies certain compatibilities with the inner product and the Clifford multiplication: 
\begin{align}
a)\quad&X\idscal{1}{\spinb(M)}{u}{v}=\idscal{1}{\spinb(M)}{\Nabla{\spinb(M)}{X}u}{v}+\idscal{1}{\spinb(M)}{u}{\Nabla{\spinb(M)}{X}v} \label{compspin1}\\
b)\quad&\Nabla{\spinb(M)}{X}(\cliff{Y}u)=\cliff{\Nabla{}{X}Y}u+\cliff{Y}\Nabla{\spinb(M)}{X}u \label{compspin2}
\end{align}
for all $X,Y \in \mathfrak{X}(M)$ and $u,v \in C^\infty(\spinb(M))$. This makes $\spinb(M)$ a Dirac bundle with compatible inner product. Condition \cref{compspin2} implies, that the connection on the bundle $\group{End}_{\C}(\spinb(M))$ satisfies $\nabla_{X} \mathbf{c}=0$ and the spinorial volume form is globally parallel with respect to the same connection: $\nabla_{X} \omega_{\spinb(M)}=0$. Consequently this connection preserves the eigenspace decomposition. With all this background the \textit{Dirac operator} is defined as operator $\Dirac\,:\,C^\infty(\spinb(M))\,\rightarrow\,C^\infty(\spinb(M))$, which is locally given by
\begin{equation}\label{diracM}
\Dirac u := \sum_{j=0}^n \varepsilon_{j}\cliff{e_j} \Nabla{\spinb(M)}{e_j} 
\end{equation}
for a Lorentz-orthonormal tangent frame $e_0,e_1,...,e_n$ and $\varepsilon_j=\met(e_j,e_j)=\pm 1$. One sees, that $\Dirac \in \Diff{1}{}(\End(\spinb(M)))$ and its principal symbol can be calculated with \cref{prinsymbdiff}, showing in particular non-ellipticity of $\Dirac$:
\begin{equation*}
\bm{\sigma}_1(\Dirac)(x,\xi)u=\lim_{\lambda \rightarrow \infty} \left. \left(\frac{\Imag}{\lambda}\right) \expe{-\Imag \lambda \Phi}\Dirac\left(\expe{\Imag \lambda \phi}u\right)\right\vert_{x}=-\sum_{j=0}^n \varepsilon_j (\xi^\sharp)^j \cliff{e_j}u(x)=-\cliff{\tilde{\xi}^\sharp}u(x)\, ,
\end{equation*}
where $\tilde{\xi}^{\sharp}$ is $\xi^\sharp$ with $(\tilde{\xi}^\sharp)^0=-(\xi^\sharp)^0$. The corresponding \textit{Dirac-Laplacian} $\Dirac^2$ is a differential operator of second order and its principal symbol can be calculated by the composition-to-multiplication-correspondence of the principal symbol map: 
\begin{equation*}
\bm{\sigma}_2(\Dirac^2)(x,\xi)=\bm{\sigma}_1(\Dirac)(x,\xi)\cdot\bm{\sigma}_1(\Dirac)(x,\xi)=-\met(\xi^\sharp,\xi^\sharp)\id{\spinb(M)}\quad.
\end{equation*}
In particular this, shows, that $\Dirac^2$ is a normally hyperbolic operator, which becomes important in section \ref{chap:feynman} and in the appendix.\\
\\
The $\Z_2$-grading induces a decomposition of the Dirac operator itself: since both subbundles $\spinb^{\pm}(M)$ have the same rank and $\Dirac \omega_{\spinb(M)}=-\omega_{\spinb(M)}\Dirac$, 
the Dirac operator is of the form
\begin{equation}\label{directsumrep}
\Dirac=\left(\begin{matrix} 0 & D_{-} \\ D_{+} & 0 \end{matrix}\right)\quad\text{with}\quad D_{\pm}\in \Diff{1}{}(M,\Hom(\spinb^{\pm}(M),\spinb^{\mp}(M)))
\end{equation}
We will first focus on $D_{+}$, which we will abbreviate with $D$. The operator $D_{-}$ will be abbreviated with $\tilde{D}$, if it becomes necessary for notational reasons. Both are non-elliptic and their Dirac-Laplacians $D_{\pm}D_{\mp}$ are normally hyperbolic.\\
\\
If the number of spatial dimensions $n$ is even, the situation is as follows: on odd dimensional manifolds the spin structure does not admit a chirality decomposition, consequently the Dirac operator is purely defined by \cref{diracM} and one does not need to distinguish between spinor fields of positive or negative chirality. This fact will simplify the restriction of the spin structure on the slices.  

\subsection{Dirac Operator on the Cauchy hypersurfaces}
So far the recapitulation was done without insisting the globally hyperbolic structure of $M$. This will be done now:

\begin{prop}\label{diracselfadprop}
Suppose $M$ is a temporal compact globally hyperbolic spin manifold with Cauchy temporal function $\timef$, such that $M$ is foliated by a family of non-compact spatial Cauchy hypersurfaces $\SET{\Sigma_t}_{t\in \timef(M)}$, and $\Dirac$ the Dirac operator of a Dirac bundle; this Dirac operator is formally anti-self-adjoint, if the intersection of supports of two differentiable spinor fields is compact, non-empty and contained in the interior of $M$; otherwise the defect is described by
\begin{equation}\label{diracselfad}
\begin{split}
\int_M \idscal{1}{\spinb(M)}{\Dirac u}{v}+&\idscal{1}{\spinb(M)}{u}{ \Dirac v} \dvol{} = \\
&\int_{\Sigma_{2}} \idscal{1}{\spinb(M)}{u}{\cliff{\nu}v} \dvol{\Sigma_t}-\int_{\Sigma_{1}} \idscal{1}{\spinb(M)}{u}{\cliff{\nu}v} \dvol{\Sigma_t} \, , 
\end{split}
\end{equation}
if furthermore the intersection of the supports meets at least one of the two boundary hypersurfaces.  
\end{prop}
\begin{proof}
If $M$ is $(n+1)$-dimensional and temporal compact, i.e. $\exists\,t_1,t_2 \in \R$, such that $\timef(M)=[t_1,t_2]$, and each non-compact hypersurface in the foliating family has no boundary, the boundary of $M$ is $\bound M =\Sigma_{1}\sqcup \Sigma_{2}$, where here and in the following $\Sigma_1:=\Sigma_{t_1}$ and $\Sigma_{2}:=\Sigma_{t_2}$. The sum
\begin{equation*}
\idscal{1}{\spinb(M)}{\Dirac u}{v}+\idscal{1}{\spinb(M)}{u}{ \Dirac v}
\end{equation*}
for $u,v \in C^1_\comp(\spinb(M))$ without specifying the compact support leads to a boundary contribution in terms of a divergence expression. 
A calculation like in \cite{LawMi} shows formula \cref{diracselfad}, where the divergence theorem on Lorentzian manifolds has been applied with a timelike unit normal vector $\mathfrak{n}$, where $\mathfrak{n}=\upnu$ is outward pointing on $\Sigma_{1}$. If the supports satisfy $\supp{u}\subset \mathring{M}$ and $\supp{v} \subset \mathring{M}$, the right hand side vanishes, thus $\Dirac$ is indeed formally anti-self-adjoint: $\Dirac^{\dagger}=-\Dirac$. If $(\supp{u}\cap\supp{v})\cap\Sigma_{i}\neq \emptyset$ the defect is described by a boundary term on $\Sigma_{i}$ for either $i=1$ or $i=2$ or both. 
\end{proof}
\begin{cor}\label{diracselfadchir}
The same preassumptions imply formally $D^\dagger=-\tilde{D}$ and $\tilde{D}^\dagger=-D$.
\end{cor}
\begin{proof}
Use the expression for $\Dirac$ with respect to the chirality decomposition in \cref{directsumrep}.
\end{proof}

We follow the notation of \cite{BaerStroh} and set $\upbeta:=\cliff{\upnu}$. \cref{cliffmult} implies $\upbeta^2=\id{\spinb(M)}$ and it induces a positive definite inner product 
\begin{equation}\label{restspinbundmet}
\dscal{1}{\spinb(\Sigma_t)}{\cdot}{\cdot}=\idscal{1}{\spinb(M)}{\upbeta\,\cdot}{\cdot}\quad, 
\end{equation}
where the inserted spinors either needs to be restricted to the hypersurface or lifted to the whole manifold. With respect to this inner product $\upbeta$ still acts as an isometry by the formal self-adjointness \cref{cliffmult2}: for $u,v$ spinor fields on a hypersurface $\Sigma_{t}$ and their lifts $\tilde{u}, \tilde{v}$ to $M$ 
\begin{equation}\label{isombeta}
\dscal{1}{\spinb(\Sigma_t)}{\upbeta u}{\upbeta v}=\idscal{1}{\spinb(M)}{\upbeta^2 \tilde{u}}{\upbeta \tilde{v}}=\idscal{1}{\spinb(M)}{\upbeta \tilde{u}}{\tilde{v}}=\dscal{1}{\spinb(\Sigma_t)}{u}{v} \quad .
\end{equation} 
In order to decompose the Dirac operator along a hypersurface in the foliating family of Cauchy hypersurfaces, one first needs to introduce a spin structure on each $\Sigma_t$ for $t \in \timef(M)$. This can be done by restricting the bundle $\spinb(M)$ to each $\Sigma_t$, see \cite{BaerGauMor} or \cite{Gin} for the details of the underyling principal bundle structure. For $n$ odd the restricted bundle is $\spinb^{\pm}(M)\vert_{\Sigma_t}=\spinb(\Sigma_t)$ for each $t \in \timef(M)$. This implies $\spinb(M)\vert_{\Sigma_t}=\spinb(\Sigma_t)\oplus \spinb(\Sigma_t)=:\spinb(\Sigma_t)^{\oplus 2}$. The Clifford multiplication for spinors on the hypersurfaces differ by the choice of the restricted eigensubbundle: one defines for a vector field $X$ on $\Sigma_t$ the corresponding homomorphism
\begin{equation*}
\Cliff{t}{X}:= \Imag \upbeta \cliff{X} \quad\text{for}\quad \spinb(\Sigma_t)=\spinb^{\pm}(M)\vert_{\Sigma_t}\quad .
\end{equation*} 
We will introduce the notation $\spinb^{\pm}(\Sigma_t)$ in order to stress, which spinor eigenbundle on $M$ is restricted to the hypersurface. 
Thus on $\spinb(M)\vert_{\Sigma_t}$ the Clifford multiplication becomes \\ $\Cliff{t}{X} \oplus (-\Cliff{t}{X})$. Note, that $\upbeta$ anticommutes with $\cliff{X}$ by \cref{cliffmult} for any tangent vector $X$.
A Clifford structure on the hypersurface is inherited from the one on $M$ and the above anti-commuting: 
\begin{equation}\label{clifftstruc}
-2\met_t(X,Y)\id{\spinb(\Sigma_t)} = \Cliff{t}{X}\Cliff{t}{Y}+\Cliff{t}{Y}\Cliff{t}{X}
\end{equation}
for vector fields $X,Y$ on $\Sigma_t$ for fixed $t$. While Clifford multiplication is formally self-adjoint on $M$, one can show, that $\Cliff{t}{\cdot}$ is formally skew-adjoint with respect to the positive inner product, introduced above: 
\begin{equation}\label{clifftskew}
\dscal{1}{\spinb(\Sigma_t)}{\Cliff{t}{X}u}{v}=-\dscal{1}{\spinb(\Sigma_t)}{u}{\Cliff{t}{X}v} \quad;
\end{equation}
see \cite{BaerStroh} for the calculation. \\
\\
The connection on $\spinb(M)\vert_{\Sigma_t}$ is the direct sum connection $\Nabla{\spinb(\Sigma_t)}{}\oplus\Nabla{\spinb(\Sigma_t)}{}$, where $\Nabla{\spinb(\Sigma_t)}{}$ is induced by the covariant derivative along a vector in $T_p\Sigma_t$ for a fixed $t$, described in \cref{covsplit}: suppose $u \in C^\infty(\spinb^{\pm}(M))$, in that case
\begin{equation}\label{spincovsplit}
\left.\Nabla{\spinb(M)}{X}u\right\vert_{\Sigma_t}=\left.\Nabla{\spinb(\Sigma_t)}{X}u\right\vert_{\Sigma_t}+\left.\frac{1}{2}\upbeta\cliff{\wein(X)}u\right\vert_{\Sigma_t}=\left.\Nabla{\spinb(\Sigma_t)}{X}u\right\vert_{\Sigma_t}\mp\left.\frac{\Imag}{2}\Cliff{t}{\wein(X)}u\right\vert_{\Sigma_t} \quad;
\end{equation}
see \cite{BaerGauMor} for the derivation. This notion of connection enables to calculate the Dirac operator along a fixed hypersurface: suppose $e_0=\upnu,e_1,...,e_n$ is a Lorentz-orthonormal tangent frame, then choose $e_1,...,e_n$ as Riemann-orthonormal frame for $\Sigma_t$. For a spinor field $u \in C^\infty(\spinb^{+}(M))$ one yields
\begin{equation*}
\left. \Dirac u \right\vert_{\Sigma_t}=-\left.\upbeta \left(\Nabla{\spinb(\Sigma_t)}{\upnu}+\Imag \mathcal{A}_t -\frac{n}{2}H_t \right)u\right\vert_{\Sigma_t} \quad.
\end{equation*}
The \textit{hypersurface-Dirac operator} on sections of $\spinb(M)\vert_{\Sigma_t}$ for an odd dimensional submanifold thus takes the form

\begin{equation}\label{Dirachyp}
\mathcal{A}_t=\left(\begin{matrix}
A_t & 0 \\
0 & -A_t
\end{matrix} \right)\quad \text{with} \quad A_t=\sum_{j=1}^n \Cliff{t}{e_j} \Nabla{\spinb(\Sigma_t)}{e_j} \quad.
\end{equation}
The corresponding Dirac operators acting on $\spinb^{\pm}(M)\vert_{\Sigma_t}$-sections along a hypersurface follow by the correct sign for the restricted Clifford multiplication. For sections $u$ with either positive or negative chirality, one has 
\begin{equation} \label{dirachyppos}
\left. D_{\pm} u \right\vert_{\Sigma_t} = -\upbeta \left.\left(\Nabla{\spinb(\Sigma_t)}{\upnu}\pm\Imag A_t  -\frac{n}{2}H_t \right)u\right\vert_{\Sigma_t}=-\upbeta \left.\left(\Nabla{\spinb(\Sigma_t)}{\upnu}+ B_{t,\pm} \right)u\right\vert_{\Sigma_t} \quad;
\end{equation}
The abbreviation $B_{t,\pm}=\pm\Imag A_t - \frac{n}{2} H_t$ indicates an operator of most first order acting tan-gential to the hypersurface.\\
\\
In comparison to the approach in \cite{BaerWafo} several special relations turn out to be useful in computing energy estimates. 
\begin{lem} \label{lemenest1} 
The following relations hold for a vector field $X$ and a smooth section $u$ of $\spinb(M)$ along a hypersurface $\Sigma_t$ for each $t\in \timef(M)$: 
\begin{itemize}
\item[(a)] $\Nabla{\spinb(\Sigma_t)}{X}(\upbeta u)=\upbeta \Nabla{\spinb(\Sigma_t)}{X}u$ and $A_t (\upbeta u)=-\upbeta A_t u$;
\end{itemize}
if moreover each $\Sigma_t$ is complete
\begin{itemize}
\item[(b)] $\Lambda_{t}^s\upbeta=\upbeta\Lambda_{t}^s$ for $s\in \R$, where $\Lambda_{t}^2=\id{}+\adNabla{\spinb(\Sigma_t)}{}\Nabla{\spinb(\Sigma_t)}{}$;
\item[(c)] $\dscal{1}{\spinb(\Sigma_t)}{B_{t,\pm} v}{w}+\dscal{1}{\spinb(\Sigma_t)}{v}{B_{t,\pm} w}=-n H_t\dscal{1}{\spinb(\Sigma_t)}{v}{w}$ for all $v,w\in C^\infty(\spinb^{\pm}(\Sigma_t))$, sharing the same chirality.
\end{itemize}
\end{lem}
\begin{proof}
\begin{itemize} \item[]
\item[(a)] The first commutativity is a consequence of the compatibility with Clifford multiplication \cref{compspin2} and \cref{spincovsplit}, implying the anti-commuting with $A_t$.
\item[(b)] Lichnerowicz formula for the hypersurface Dirac operator and result (a) lead to
\begin{equation*}
\adNabla{\spinb(\Sigma_t)}{}\Nabla{\spinb(\Sigma_t)}{}(\upbeta u)= \mathcal{A}_t^2 (\upbeta u)-\frac{\curvcon_{\Sigma_t}}{4}\upbeta u = \upbeta \mathcal{A}_t^2 u -\upbeta\frac{\curvcon_{\Sigma_t}}{4} u= \upbeta \adNabla{\spinb(\Sigma_t)}{}\Nabla{\spinb(\Sigma_t)}{}u
\end{equation*}
and thus $\Lambda_{t}^2(\upbeta u)= \upbeta \Lambda_{t}^2 u$; here $\curvcon_{\Sigma_t}$ denotes the Ricci scalar for $\Sigma_t$. This holds true for any positive even power $\Lambda_{t}^{2k}$, $k \in \N_0$, by applying the above result $k$ times and thus for any polynomial in $\Lambda_{t}^{2}$. It still holds for $\Lambda_{t}^s$ for any $s\in \R$, as $\Lambda_{t}^2$ is essentially self-adjoint on $L^2(\spinb^{\pm}(\Sigma_t))$ by the completeness of the hypersurfaces. Its spectrum is positive, wherefore the function $f(x)=x^{s/2}$ is continuous on the spectrum of $\Lambda_{t}^2$. Then $\Lambda_{t}^s=f(\Lambda_{t}^2)$ is defined by the limit of any sequence of polynomials, converging uniformly on the spectrum to $f$, on which the showed commutativity of $\upbeta$ and any polynomial in $ \Lambda_{t}^2$ can be applied.

\item[(c)] The left hand side of the equation gives for both chiralities
\begin{equation*}
\pm\left(\dscal{1}{\spinb(\Sigma_t)}{\Imag A_t v}{w}+\dscal{1}{\spinb(\Sigma_t)}{v}{\Imag A_t w}\right)-n H_t\dscal{1}{\spinb(\Sigma_t)}{v}{w} \quad.
\end{equation*}
$A_t$ is formally self-adjoint w.r.t to the induced inner product on $\Sigma_t$, since $\bound \Sigma_t =\emptyset$ for all $t \in \timef(M)$ by completeness of the Cauchy hypersurface; the proof can be taken from Proposition 5.3. in chapter II of \cite{LawMi}. One needs to be careful with the signs, since Clifford multiplication is now formally skew-adjoint by \cref{clifftskew}. The action of the covariant derivative on $\Cliff{t}{e_j}$ has two contributions, coming with $\cliff{\nabla_{e_j}\upnu}$ after applying on $\upbeta$ and $\cliff{\nabla_{e_j} e_j}$. By choosing the Riemann-orthonormal frame $e_1,...,e_n$ in such a way, that the Lorentz-orthonormal frame $\upnu,e_1,...,e_n$ becomes synchronous at a point, these Clifford multiplications won't contribute and the left boundary contribution as in the mentioned reference vanishes. Thus $\Imag A_t$ is formally anti-self-adjoint with respect to the same inner product and the term in the brackets vanishes.
\end{itemize}
\end{proof}

This result can be extended to a positive definite $L^2$-scalar product:
\begin{equation}\label{extendB}
\dscal{1}{L^2(\spinb(\Sigma_t))}{B_{t,\pm} v}{w}+\dscal{1}{L^2(\spinb(\Sigma_t))}{v}{B_{t,\pm} w}=-n\dscal{1}{L^2(\spinb(\Sigma_t))}{H_t v}{w} \quad.
\end{equation}
The proof is based on the calculation in (c), but now one uses the essential self-adjointness of the Riemannian Dirac operator $A_t$ on $L^2$-spaces, which is justified\bnote{f14}, if each hypersurface is either compact or complete. The operator $B_{t,\pm}$ is then defined as above, but one takes the extension of $A_t$ instead, which we will also denote by the same symbol.\\
\\
As there is no $\Z_2$-grading for an even number of spatial dimensions, the restricted spinor bundle is isomorphic to the spinor bundle with respect to a fixed chosen hypersurface: $\spinb(M)\vert_{\Sigma_t}=\spinb(\Sigma_t)$ for all $t\in \timef(M)$. This and the following results can be found in Proposition 1.4.1 in \cite{Gin}. The homomorphism for the Clifford multiplication on the hypersurfaces is related with the Clifford multiplication homomorphism on $M$ via $\Cliff{t}{X}:= \Imag \upbeta \cliff{X}$. The Clifford multiplication on the even dimensional hypersurface is in that case similar to the one with positive chirality for an odd number of spatial dimensions. The properties \cref{clifftstruc} and \cref{clifftskew} still hold true by this similarity. The Dirac operator $\Dirac$ along a fixed hypersurface can be derived in the same manner as it was done for spinor fields with positive chirality: 
\begin{equation}\label{dirachypodd}
\Dirac u \vert_{\Sigma_t}=-\upbeta\left.\left(\Nabla{\spinb(\Sigma_t)}{\upnu}+\Imag A_t - \frac{n}{2}H_t\right)u\right\vert_{\Sigma_t}
\end{equation}
for all $t \in \timef(M)$ fixed. Notice, that for odd dimensional hypersurfaces $A_t$ is replaced by $\mathcal{A}_t$ from \cref{Dirachyp}. One observes, that its restriction has the same form as $D_{+}$. This allows to reduce the analysis of the well-posedness of the corresponding Cauchy problem to the well-posedness properties of $D$ alone with the only difference, that there is no need to pay attention to the chiralities.


\addtocontents{toc}{\vspace{-3ex}}
\section{Well-posedness of the Cauchy problem for the Dirac equation}\label{chap:cauchy}
The aim of this section is to prove well-posedness of the Cauchy problem for the Dirac operator $D$ on a globally hyperbolic Lorentzian spin manifold, where each member in the foliating family $\{\Sigma_t\}_{t \in \timef(M)}$ is a non-compact, bue complete Cauchy hypersurface:
\begin{equation}\label{Diraceq}
D u = f\quad\text{with}\quad u\vert_{\Sigma_t}=g
\end{equation}
where $u$ is a suitable weak solution as section of $\spinb^{+}(M)$, $f$ a suitable section of $\spinb^{-}(M)$ and $g$ a section of $\spinb^{+}(\Sigma_t)$. To distinguish between solutions of the homogeneous and inhomogeneous Cauchy problem in the space of finite energy sections one introduces the following two subspaces.
\begin{defi}\label{finensolkern}
For any $s \in \R$ the set of \textit{finite} $s$-\textit{energy solutions of} $D$ is defined by
\begin{equation}\label{finensol}
FE^{s}_{\scomp}(M,\timef,D):=\SET{u \in FE^s_{\scomp}(M,\timef,\spinb^{+}(M))\,\vert\, Du \in L^2_{\loc,\scomp}(\timef(M),H^s_\loc(\spinb^{-}(\Sigma_{\bullet})))} \quad;
\end{equation}
the set of \textit{finite} $s$-\textit{energy kernel solutions of} $D$ is defined as
\begin{equation}\label{finenkern}
FE^{s}_{\scomp}(M,\timef,\kernel{D}):=FE^s_{\scomp}(M,\timef,\spinb^{+}(M))\cap \kernel{D} \quad.
\end{equation}
\end{defi}
The kernel solutions come with an interesting property on its own:
\begin{lem}\label{boots}
Suppose $u \in FE^{s}_{\scomp}(M,\timef,\kernel{D})$; for $s > \frac{n}{2}+2$, then $u \in C^1_{\scomp}(\spinb^{+}(M))$ .
\end{lem}
\begin{proof}
With \cref{dirachyppos} and $Du=0$ one has along each hypersurface
\begin{equation*}
\left. \Nabla{\spinb(M)}{\partial_t}u \right\vert_{\Sigma_t}=-N\left. \Nabla{\spinb(M)}{\upnu}u \right\vert_{\Sigma_t}=N\left(\Imag A_t - \frac{n}{2}H_t\right)u\vert_{\Sigma_t} \quad.
\end{equation*}
From its definition $u \in C^0_{\scomp}(\timef(M),H^s_\loc(\spinb^{+}(\Sigma_{\bullet})))$ has support inside a spatially compact subset in $M$. For $u\vert_{\Sigma_t} \in H^s_\loc(\spinb^{+}(\Sigma_t))$ at each $t \in \timef(M)$ and $\supp{u}\cap \Sigma_t$ compact by the spatial compactness of the support, the right hand side consists of differential operators at most order 1 along each $\Sigma_t$, thus one has $\left. \Nabla{\spinb(M)}{\partial_t}u \right\vert_{\Sigma_t} \in H^{s-1}_\loc(\spinb^{+}(\Sigma_t))$ and hence $u \in C^1_{\scomp}(\timef(M),H^{s-1}_\loc(\spinb^{+}(\Sigma_{\bullet})))$. The claim follows by the Sobolev embedding theorem for $s-1 > \frac{n}{2}+1$.
\end{proof}
For the Cauchy problem of interest choose an initial data $u\vert_{\Sigma_t}\in H^s_\loc(\spinb^{+}(\Sigma_t))$ and an inhomogeneity $f \in L^2_{\loc,\scomp}(\timef(M),H^s_\loc(\spinb^{-}(\Sigma_{\bullet})))$ in \cref{Diraceq} for any $s\in \R$ and $t \in \timef(M)$. As in the main reference \cite{BaerWafo} we start with the more stronger condition $f \in FE^{s-1}_\scomp(M,\timef,\spinb^{-}(M))$, which can be weakened later on, but does not affect the main proof. For the coming energy estimates a time reversal argument is going to be used, for which reason a closer look on the time reversed Cauchy problem needs to be taken: the \textit{time reversal map}
\begin{align*}
\mathcal{T}\,:\, &M  \rightarrow  M \\
&  (t,x)  \mapsto  (-t,x)
\end{align*}
is smooth and acts as involution, since $\mathcal{T}^2=\id{M}$, hence it is a diffeomorphism on $M$. Moreover this implies, that $\mathcal{T}$ is formally self-adjoint. We will quote its inverse with the same letter, as it is a self-inverse map. The pullback of a spinor field with respect to this diffeomorphism is well defined as spinor field with respect to a Clifford algebra, which is defined by the pullback metric $\mathcal{T}^{\ast}\met $: 
\begin{equation*}
(\mathcal{T}^{\ast}u)(t,x)=u(\mathcal{T}(t,x))=u(-t,x)
\end{equation*}
for a smooth spinor field $u$; more details concerning the structure of this scalar like transformation behaviour can be found in \cite{DP}. We use these facts in the proof of the following statement, which provides us with a time reversal argument in the upcoming analysis.
\begin{lem}\label{timeinvarinace}
Given a globally hyperbolic Lorentzian manifold $M$ with Cauchy temporal function $\timef$, a spinor bundle $\spinb(M)$, $K \subset M$ compact and $s \in \R$; the following are equivalent:
\begin{itemize}
\item[(a)] $u \in FE^s_\scomp(M,\timef,\spinb^{\pm}(M))$ solves the forward time Cauchy problem 
\begin{equation*}
D_{\pm}u=f \in FE^{s-1}_\scomp(M,\timef,\spinb^{\mp}(M)) \quad , \quad u\vert_{t}=: u_0 \in H^s_\loc(\spinb^{\pm}(\Sigma_t))
\end{equation*}
for the Dirac equation at fixed initial time $t \in \timef(M)$.
\item[(b)] $\mathcal{T}^\ast u \in FE^s_\scomp(\mathcal{\mathcal{T}}^{-1}(M),\mathcal{T}(\timef),(\mathcal{T}^{-1})^\ast\spinb^{\pm}(M))$ solves the backward time Cauchy problem 
\begin{equation*}
(\mathcal{T}^\ast\circ D_{\pm} \circ \mathcal{T}^\ast)u=\mathcal{T}^\ast f \quad , \quad (\mathcal{T}^\ast u)\vert_{t}= u_0 \in H^s_\loc(\spinb^{\pm}(\Sigma_t))
\end{equation*}
with $\mathcal{T}^\ast f \in FE^{s-1}_\scomp(\mathcal{\mathcal{T}}^{-1}(M),\mathcal{T}(\timef),(\mathcal{T}^{-1})^\ast\spinb^{\mp}(M))$ for the Dirac equation at fixed initial time $t \in \timef(M)$.
\end{itemize}
Moreover $(\mathcal{T}^\ast\circ D_{\pm} \circ \mathcal{T}^\ast)$ is the Dirac operator for reversed time orientation and takes the form
\begin{equation*}
\left. (\mathcal{T}^\ast \circ D_{\pm} \circ \mathcal{T}^\ast) v \right\vert_{\Sigma_t} = -\left.\widetilde{\upbeta}\left(\nabla_{\widetilde{\upnu}}\pm\Imag \widetilde{A}_t v-\frac{n}{2} \widetilde{H}_t\right) v\right\vert_{\Sigma_t} \quad,
\end{equation*}
where $\widetilde{\upnu}=\mathcal{T}_\ast \upnu$, $\widetilde{\upbeta}=\cliff{\widetilde{\upnu}}$, $\widetilde{H}_t$ is the mean curvature with respect to the normal vector $\widetilde{\upnu}$ and $\widetilde{A}_t$ the hypersurface Dirac operator, defined as in \cref{Dirachyp} with a Riemann-orthonormal tangent frame with respect to $\mathcal{T}^\ast \met_t$. 
\end{lem}
\begin{proof}
Both Dirac equations are formally the same, if one applies a pullback by $\mathcal{T}$ on both sides and uses the self-inverse property $\mathcal{T}^2=\id{M}$ between $D$ and the spinor $u$. The claim will be proven for smooth initial data and inhomogeneities. Because finite energy sections are embedded in the set of distributional sections, the claim follows from the proof of this reduction by duality: consider the dual pairing of any operator, applied on a spin-valued distribution $u \in C^{-\infty}(\spinb(M))$, with a smooth, compactly supported section $\phi \in C^\infty_\comp(\spinb^\ast(M))$, i.e. a smooth, compactly supported cospinor field. This is equivalent with the dual pairing of $u$ with the formal adjoint operator, now acting on $\phi$. Thus the forward and backward time Dirac equations for distributions are
\begin{equation*}  
\dpair{1}{\spinb(M)}{D_{\pm}u}{\phi}=-\dpair{1}{\spinb(M)}{u}{D_{\mp}\phi}\,\, \text{\&} \,\, \dpair{1}{\spinb(M)}{(\mathcal{T}^\ast\circ D_{\pm} \circ \mathcal{T}^\ast)u}{\phi}=-\dpair{1}{\spinb(M)}{u}{(\mathcal{T}^\ast\circ D_{\mp} \circ \mathcal{T}^\ast)u\phi}
\end{equation*}
with the formal self-adjointness of $\mathcal{T}$ and Corollary \ref{diracselfadchir}. Since the support of $u$ is contained in the future and past light cone, $\mathcal{T}$ only swaps these two cones, wherefore the support satisfies $\supp{\mathcal{T}^{\ast}u}\subset \Jlight{}(K)$. 
Suppose $u\vert_{\Sigma_t} \in C^\infty_\comp(\spinb^{\pm}(\Sigma_t))$ and $f \in C^\infty_\comp(\spinb^{\mp}(M))$. Theorem 4 in \cite{AndBaer} implies the existence of a unique section $u\in C^\infty(\spinb^{\pm}(M))$ with support $\supp{u}\subset \Jlight{}(K)$ for $K \subset M$ compact, solving $D_{\pm}u=f$ on $M$ with initial condition $u\vert_{\Sigma_t}$. $\mathcal{T}^{\ast}u$ and $\mathcal{T}^{\ast} f$ are defined and again smooth and the latter is compactly supported. The time reversal map doesn't change the initial value itself, because it is defined on a fixed time slice: $\left.\mathcal{T}^{\ast}u\right\vert_{\Sigma_\tau}=u\vert_{\Sigma_\tau}$ for any $\tau\in \timef(M)$ fixed. The same holds true for any initial value with Sobolev regularity, since only the metric is influenced by the time reversal, but different metrics leads to equivalent Sobolev norms. $v=\mathcal{T}^\ast u$ for a solution $u$ is defined and again smooth with $\supp{v}\subset \Jlight{}(K)$ by the same reasoning as above. The formal equivalence of the forward and backward time Dirac equation implies, that $v$ solves $\mathcal{T}^{*}D_{\pm}\mathcal{T}^{*}v=\mathcal{T}^{*}f$ if and only if $u$ solves $D_{\pm}u=f$. Thus one only needs to check, that $\mathcal{T}^{*}D\mathcal{T}^{*}$ along any hypersurface $\Sigma_t$ is also a Dirac operator, given as in \cref{dirachyppos}. The pulled back spin structure is determined by the pulled back metric, hence the pullback on any Clifford multiplication $\cliff{X}$ with respect to a vector field $X$ is the Clifford multiplication with respect to the pushforward $\mathcal{T}_{\ast}X$ at each point:
\begin{equation*}
\mathcal{T}^{\ast}\circ\cliff{X}=\cliff{\mathcal{T}_{\ast}X}\quad .
\end{equation*}
Applied to each component of \cref{dirachyppos}, the pullback of the Dirac operator along any spatial hypersurface is defined by the pullback-connection: given a Riemann-orthonormal tangent frame $\SET{e_j}_{j=1}^n$ with respect to $\met_t$ for each leaf, then $\SET{\mathcal{T}_\ast e_j}_{j=1}^n=\SET{\mathcal{T}^\ast e_j}_{j=1}^n$ is a Riemann-orthonormal tangent frame with respect to $\mathcal{T}^\ast \met_t$ for each leaf. Using all these ingredients, shows 
\begin{eqnarray*}
\mathcal{T}_\ast \upnu &=& - \frac{1}{N\circ\mathcal{T}}\mathcal{T}_\ast \partial_{t}=:\widetilde{\upnu}\\
\mathcal{T}^{\ast}\circ\left(\upbeta \nabla_{\upnu}\right)\circ\mathcal{T}^\ast v &=& \cliff{\mathcal{T}_{\ast}\upnu} \mathcal{T}^\ast \nabla_{\upnu} \left(\mathcal{T}^{\ast}v\right)=\cliff{\widetilde{\upnu}}\nabla_{\widetilde{\upnu}}\mathcal{T}^{\ast}\mathcal{T}^{\ast}v = \widetilde{\upbeta}\nabla_{\widetilde{\upnu}}v \\
\mathcal{T}^\ast \circ\left(\upbeta A_t \right)\circ\mathcal{T}^\ast v &=& \widetilde{\upbeta}\sum_{j=1}^n \cliff{\mathcal{T}_\ast e_j}\mathcal{T}^\ast \nabla_{e_j}\mathcal{T}^\ast v =  \widetilde{\upbeta}\sum_{j=1}^n \cliff{\mathcal{T}_\ast e_j} \nabla_{\mathcal{T}_\ast e_j} v = \widetilde{\upbeta} \widetilde{A}_t v \\
\mathcal{T}^\ast \circ \left(\upbeta H_t\right)\circ \mathcal{T}^\ast v &=& \widetilde{\upbeta}\sum_{j=1}^n\mathcal{T}^\ast \met_t\left(\wein(e_j),e_j\right) v 
=  \widetilde{\upbeta} \widetilde{H}_t v \\
\Rightarrow\,\,(\mathcal{T}^\ast \circ D_{\pm} \circ \mathcal{T}^\ast) v &=& -\widetilde{\upbeta}\left(\nabla_{\widetilde{\upnu}}+\Imag \widetilde{A}_t v-\frac{n}{2} \widetilde{H}_t\right) v \quad,
\end{eqnarray*}
where the tilded quantities are Clifford multiplication, the Weingarten map and the mean curvature with respect to the future poining vector $\widetilde{\upnu}$, being orthonormal to each hypersurface as well. All quantities along a hypersurface have been lifted to the manifold $M$ before (without extra notation), in order to compute the pullback. 
\end{proof}
A concrete calculation shows, that $\widetilde{A}_t^2=\mathcal{T}^\ast A_t^2 \mathcal{T}^\ast$ and thus similarly for $\mathcal{A}^2_t$ with reversed time orientation. The Lichnerowicz formula then shows, that the spinorial Laplacian in the time-reversed picture is given by the spinorial Laplacian in ordinary time orientation, composed with $\mathcal{T}^\ast$ before an after.
In an analogous way as in the proof of Lemma \ref{lemenest1} (b) one yields 
\begin{equation}\label{timelambda}
\Lambda^s_{\mathcal{T}(t)}=\mathcal{T}^\ast\Lambda^s_{t}\mathcal{T}^\ast \quad\forall\,s\in \mathbb{R}\,\, ,t\in \timef(M)\,,
\end{equation}
where $\Lambda^s_t$ is given by \cref{laps} with the induced spin connection on each hypersurface $\Sigma_t$; ${\mathcal{T}(t)}$ denotes $t$ under reversed time orientation.

\subsection{Energy estimate}
\justifying
Roughly speaking one defines the energy along an initial spatial hypersurface as sum of the norms of all values, which has to be fixed on this leaf. 
Suppose $M$ is spatially compact, then the Cauchy hypersurface $\Sigma$ is compact and one defines the \textit{s-energy along} $\Sigma$ of a sufficiently differentiable section $u$ of a vector bundle $E\rightarrow M$ as
\begin{equation*}
\mathcal{E}_s(u,\Sigma):= \norm{u\vert_\Sigma}{H^s(E\vert_\Sigma)}^2 \quad .
\end{equation*}   
For a non-compact hypersurface $\Sigma$ one needs to apply the presented doubling procedure in order to reduce to the above case. Let $(M,\met)$ be a globally hyperbolic Lorentzian spacetime and $E\rightarrow M$ a Riemannian or Hermitian vector bundle. Choose a connection, which is compatible with the bundle metric. Let $u$ be again a sufficiently differentiable section of this bundle, but has compact support $\supp{u} \subset \Jlight{}(K)$ with $K \Subset M$. Since the support is defined to be a closed set, one observes, that it is spatially compact by this assumption. As a consequence one obtains, that $\mathfrak{K}:=(\supp{u}\cap \Sigma) \subset \Sigma$ is compact for every Cauchy hypersurface. Without changing $\mathfrak{K}$ choose a compact subset $K_1$ as in the description of Sobolev spaces via the double and receive $\widetilde{\Sigma}$ as double of $K_1$, a corresponding extended vector bundle $\widetilde{E}$ of $E\vert_{K_1}$ and a zero-extended section $\tilde{u}$ of $u\vert_\Sigma$. This allows to consider the \textit{s-energy} along $\Sigma$ of a sufficiently differentiable section $u$ in a similar manner:
\begin{equation}\label{senergy}
\mathcal{E}_s(u,\Sigma):= \norm{\tilde{u}}{H^s(\widetilde{\Sigma},\widetilde{E})}^2
\end{equation}
for any $s \in \R$. The following statement is the pendant of Theorem 8 in \cite{BaerWafo} for the Dirac equation acting on spinor sections of positive chirality. The proof contains a similar argumentation, but since only one initial value is given and no constraint on the mean curvature is proposed, we had to show the equations in (b) and (c) of Lemma \ref{lemenest1} in addition.  
\begin{prop}\label{enesttheorem}
Let $I \subset \timef(M)$ be a closed intervall, $K\subset M$ compact and $s\in \R$; there exists a constant $C>0$, depending on $K$ and $s$, such that
\begin{equation}\label{enesttheoremform}
\mathcal{E}_s(u,\Sigma_{t_1}) \leq \mathcal{E}_s(u,\Sigma_{t_0})\expe{C(t_1-t_0)}+\int_{t_0}^{t_1}\expe{C(t_1-\tau)}\norm{Du\vert_{\Sigma_\tau}}{H^s_\loc(\spinb(\Sigma_\tau))}^2 \differ \tau
\end{equation}
applies for all $t_0,t_1 \in I$ with $t_0 < t_1$ and for all $u \in FE^{s+1}_{\scomp}(M,\timef,\spinb^{+}(M))$ with support $\supp{u}\subset \Jlight{}(K)$ and $Du \in FE^{s}_{\scomp}(M,\timef,\spinb^{-}(M))$. 
\end{prop}
\begin{proof}
W.l.o.g. assume $M$ to be spatially compact, i.e. every leaf $\Sigma_t$ is closed. 
Otherwise one applies the doubling procedure of each complete non-compact hypersurfaces and start with \cref{senergy} - the difference in the following calculation is only of notational nature. The Dirac operator is decomposed into a tangential and normal part with respect to the hypersurface $\Sigma_t$: $D=-\upbeta\left(\nabla_\upnu+B_t\right)$ with $B_t:=B_{t,+}$ as in Lemma \ref{lemenest1} (c). Rewriting this in terms of the covariant derivative $\nabla_{\partial_t}=-N_t \nabla_{\upnu}$ leads to
\begin{equation*}
\left.\nabla_{\partial_t} u\right\vert_{\Sigma_t} = \left. N_t\upbeta Du \right\vert_{\Sigma_t} + \left. N_t B_t u \right\vert_{\Sigma_t}\quad .
\end{equation*}
$B_t$ is a differential operator of order at most 1, acting in tangential direction; the preassumption $u \in FE^{s+1}_{\scomp}(M,\timef,\spinb^{+}(M))$ implies $u\vert_{\Sigma_t} \in H^{s+1}_{\loc}(\spinb^{+}(\Sigma_t))$ and thus $N_t B_t u\vert_{\Sigma_t} \in H^{s}_\loc(\spinb^{+}(\Sigma_t))$ implying $N_t B_t u \in FE^s_{\scomp}(M,\timef,\spinb^{+}(M))$. Since $Du \in FE^{s}_{\scomp}(M,\timef,\spinb^{-}(M))$ by preassumption the first part of the right hand side satisfies $\upbeta N_t Du \in FE^{s}_{\scomp}(M,\timef,\spinb^{+}(M))$. Thus the covariant derivative with respect to $\upnu$ along any hypersurface $\Sigma_t$ is a Sobolev section in $H^s_\loc(\spinb^{+}(\Sigma_t))$ and therefore $\nabla_{\partial_t} u \in C^0(\timef(M),H^s_\loc(\spinb^{+}(\Sigma_\bullet))$, implying $u \in C^1(\timef(M),H^s_\loc(\spinb^{+}(\Sigma_\bullet))$. This time-differentiability and the continuity of the norm shows, that the map $t \mapsto \mathcal{E}_s(u,\Sigma_t)$ is differentiable in temporal direction. Differentiation of $\mathcal{E}_s(u,\Sigma_t)$ with respect to $t$ and Lemma \ref{lem2-2} imply:
\begin{eqnarray*}
\frac{\differ}{\differ t}\mathcal{E}_s(u,\Sigma_{t}) &=&\int_{\Sigma_t} n H_t \dscal{1}{\spinb(\Sigma_t)}{\Lambda_t^s u}{\Lambda_t^s u}-\upnu \dscal{1}{\spinb(\Sigma_t)}{\Lambda_t^s u}{\Lambda_t^s u} \dvol{\Sigma_t} \\
&=& n\dscal{1}{L^2(\spinb(\Sigma_t))}{H_t \Lambda^s_t u}{\Lambda^s_t u}-\int_{\Sigma_t}\upnu \dscal{1}{\spinb(\Sigma_t)}{\Lambda_t^s u}{\Lambda_t^s u} \dvol{\Sigma_t} \,\, ,
\end{eqnarray*}
where $u$ is evaluated on the hypersurface $\Sigma_t$ and one has choosen $\phi=1$ in Lemma \ref{lem2-2}, since every hypersurface is an artificially closed submanifold. Choose the connection to be compatible with the bundle metric and one obtains
\begin{footnotesize}
\begin{eqnarray*}
\frac{\differ}{\differ t}\mathcal{E}_s(u,\Sigma_{t}) &=&  n\dscal{1}{L^2(\spinb(\Sigma_t))}{H_t \Lambda^s_t u}{\Lambda^s_t u}-2 \Re\mathfrak{e}\left\lbrace\dscal{1}{L^2(\spinb(\Sigma_t))}{\Lambda_t^s u}{\nabla_{\upnu}\Lambda_t^s u} \right\rbrace \\
&=& n\dscal{1}{L^2(\spinb(\Sigma_t))}{H_t \Lambda^s_t u}{\Lambda^s_t u}-2 \Re\mathfrak{e}\left\lbrace\dscal{1}{L^2(\spinb(\Sigma_t))}{\Lambda_t^s u}{[\nabla_{\upnu},\Lambda_t^s] u} \right\rbrace -2 \Re\mathfrak{e}\left\lbrace\dscal{1}{H^s(\spinb(\Sigma_t))}{u}{\nabla_{\upnu} u} \right\rbrace \\
&=& n\dscal{1}{L^2(\spinb(\Sigma_t))}{H_t \Lambda^s_t u}{\Lambda^s_t u}+\norm{u}{H^s(\spinb(\Sigma_t))}^2+\norm{[\nabla_{\upnu},\Lambda_t^s] u}{L^2(\spinb(\Sigma_t))}^2-\norm{(\Lambda^s_t+[\nabla_{\upnu},\Lambda_t^s]) u}{L^2(\spinb(\Sigma_t))}^2\\
&&-2 \Re\mathfrak{e}\left\lbrace\dscal{1}{H^s(\spinb(\Sigma_t))}{u}{\nabla_{\upnu} u} \right\rbrace \\
&\leq& n\dscal{1}{L^2(\spinb(\Sigma_t))}{H_t \Lambda^s_t u}{\Lambda^s_t u}+(1+c_1)\norm{u}{H^s(\spinb(\Sigma_t))}^2-2 \Re\mathfrak{e}\left\lbrace\dscal{1}{H^s(\spinb(\Sigma_t))}{u}{\nabla_{\upnu} u} \right\rbrace \\
&=& c_2 \norm{u}{H^s(\spinb(\Sigma_t))}^2+n\dscal{1}{L^2(\spinb(\Sigma_t))}{H_t \Lambda^s_t u}{\Lambda^s_t u} +2 \Re\mathfrak{e}\left\lbrace\dscal{1}{H^s(\spinb(\Sigma_t))}{u}{\upbeta D u}+\dscal{1}{H^s(\spinb(\Sigma_t))}{u}{B_t u} \right\rbrace \\
&\stackrel{(*)}{=}& c_2 \norm{u}{H^s(\spinb(\Sigma_t))}^2+n\dscal{1}{L^2(\spinb(\Sigma_t))}{H_t \Lambda^s_t u}{\Lambda^s_t u} \\
&&+2 \Re\mathfrak{e}\left\lbrace \dscal{1}{L^2(\spinb(\Sigma_t))}{\Lambda^s_t u}{B_t \Lambda^s_t u} + \dscal{1}{L^2(\spinb(\Sigma_t))}{\Lambda^s_t u}{[\Lambda^s_t,B_t] u} \right\rbrace +2 \Re\mathfrak{e}\left\lbrace\dscal{1}{L^2(\spinb(\Sigma_t))}{\Lambda^s_t u}{\upbeta \Lambda^s_t D u}\right\rbrace\\
&\leq & (c_2+2) \norm{u}{H^s(\spinb(\Sigma_t))}^2+\norm{\upbeta \Lambda^s_t D u}{L^2(\spinb(\Sigma_t))}^2+\norm{[\Lambda^s_t,B_t] u}{L^2(\spinb(\Sigma_t))}^2\\
&\leq& c \norm{u}{H^s(\spinb(\Sigma_t))}^2+\norm{D u}{H^s(\spinb(\Sigma_t))}^2= c \mathcal{E}_s(u,\Sigma_{t}) +\norm{D u}{H^s(\spinb(\Sigma_t))}^2 \quad, 
\end{eqnarray*}
\end{footnotesize}
where one has used polarization identities of the real parts and a Sobolev estimation for $\nabla_{\nu}$ as first order operator along the hypersurface $\Sigma_t$, which has led to the first inequality. The completeness of the hypersurface justifies the use of Lemma \ref{lemenest1} (b). \cref{extendB} and another polarization identity have been used after ($\ast$). In the last step \cref{isombeta} is used and that the remaining commutator acts as pseudo-differential operator of order $s$, which generates the last inequality by its Sobolev estimation. \\
\\
A closer look on the constant $c$, coming from the Sobolev estimates, needs to be taken before applying Grönwall's inequality: from the local theory it is known, that these constants depend on compact supersets of $\supp{u}$ (either for $\Sigma_t$ closed or $\mathfrak{K}$ for $\Sigma_t$ non-compact, depending on $K$ via the support of $u$ in the light cone) and compact subsets, appearing in the symbol estimation of the acting pseudo-differential operator. Further dependencies on the Sobolev regularity degree $s$ and on finitely many derivatives on the product of the section $u$ with the volume form prefactor are coming from partial integration (take $s$ to be in $\N_0$ first). By Jacobi's formula the derivatives on the volume form prefactor generate derivatives of the metric $\met_t$ under recreation of the volume form prefactor. This makes the constant $c$ time-dependent and since the product rule generates products of derivatives on the section and on the metric it depends smoothly on $t$, because $\met_t$ does. The derivatives on the section can be extracted, which generates the Sobolev norms in terms of \cref{sobnorm2} and the estimates in the end, which are used above. Since on a (here possibly auxiliary) closed manifold the space $H^{-s}$ is dual to $H^s$ for $s \in \N_0$, the estimate also holds for Sobolev degrees $s \in \Z$. Interpolating between $H^s$ and $H^{s\pm 1}$ for any of these degrees allows an extension to real valued orders. In summary the above estimation is precisely
\begin{equation*}  
\frac{\differ}{\differ t}\mathcal{E}_s(u,\Sigma_{t}) = c(\Vert\met_t\Vert_{\mathfrak{K}(K),m(s)})\mathcal{E}_s(u,\Sigma_{t}) +\norm{D u}{H^s(\spinb(\Sigma_t))}^2
\end{equation*}
with the (spatial) seminorm \cref{normveccomp} and $m(s) \in \N_0$, such that $\vert s\vert \leq m$. Now, Grönwall's Lemma gives
\begin{equation*}
\mathcal{E}_s(u,\Sigma_{t_1}) \leq \mathcal{E}_s(u,\Sigma_{t_0})\expe{\int_{t_0}^{t_1} c(\Vert\met_t\Vert_{\mathfrak{K}(K),m(s)}) \differ t}+\int_{t_0}^{t_1}\expe{\int_{\tau}^{t_1}c(\Vert\met_t\Vert_{\mathfrak{K}(K),m(s)})\differ t}\norm{Du\vert_{\Sigma_\tau}}{H^s(\spinb(\Sigma_\tau))}^2 \differ \tau \quad.
\end{equation*}
The extreme value theorem on closed (sub-) intervalls in $I$, applied on $c$, leads to the stated result, where $C=C(\Vert \met_\bullet \Vert_{\Jlight{}(K),m(s)})$ is the maximum on $I$. 
Going back to non-compact hypersurfaces the same procedure can be applied, where the duality between $H^s_\loc$ and $H^{-s}_\comp$ has been used instead. This explains the norm for local Sobolev sections in the claim. 
\end{proof}
\noindent From this basic result one can conclude several consequences and technicalities for the next subsection.
\begin{cor}\label{corenest1}
Given $\intervallc{t_0}{t_1}{} \subset \timef(M)$, $\tau \in \timef(M)$, $K \subset M$ compact and $s \in \R$; there exists a $C>0$, depending on $K$ and $s$, such that
\begin{equation*}
\mathcal{E}_s(u,\Sigma_{t})\leq C \left(\mathcal{E}_s(u,\Sigma_{\tau})+ \norm{D u}{\intervallc{t_0}{t_1}{},\Jlight{}(K),s}^2\right)
\end{equation*}
is valid for all $t \in \intervallc{t_0}{t_1}{}$, for all $u \in FE^{s+1}_{\scomp}(M,\timef,D)$ with $Du \in FE^s_{\scomp}(M,\timef,\spinb^{-}(M))$ and $\supp{u}\subset \Jlight{}(K)$.
\end{cor}
\begin{proof}
Assume again, that each leaf is closed, otherwise one extends again everything to a suitable double. Choose $\tau \in \intervallc{t_0}{t_1}{}$, otherwise take $t_0,t_1 \in \timef(M)$ in such a way, that it is true. Suppose first $t \in \intervallc{\tau}{t_1}{}$; since $u\in FE^{s+1}_{\scomp}(M,\timef,D)\subset FE^{s+1}_{\scomp}(M,\timef,\spinb^{+}(M))$ all preassumptions from this corollary coincides with the one from Proposition \ref{enesttheorem}, such that \cref{enesttheoremform} holds. By assumption $Du \in L^2_{\loc,J(K)}(\timef,H^s(\spinb^{-}(\Sigma_{\bullet})))$ implies integrability of the map $\lambda \mapsto \norm{Du\vert_{\Sigma_{\lambda}}}{H^s(\spinb(\Sigma_{\lambda}))}$, such that
\begin{equation*}
\mathcal{E}_s(u,\Sigma_{t}) \leq C \left(\mathcal{E}_s(u,\Sigma_\tau)+\norm{D u}{\intervallc{t_0}{t_1}{},\Jlight{}(K),s}^2\right)\quad .
\end{equation*}
For $t \in \intervallc{t_0}{\tau}{}$ a time-reversed variable is used, denoted by $\mathcal{T}(t)$. Lemma \ref{timeinvarinace} ensures, that $\mathcal{T}^\ast u$ solves the backward time Dirac equation with time-reversed data, such that the proof of Proposition \ref{enesttheorem} can be repeated for this situation: one concludes
\begin{eqnarray*}
\mathcal{E}_s(\mathcal{T}^{\ast}u,\Sigma_{\mathcal{T}(t)})&\leq& c\mathcal{E}_s(\mathcal{T}^{\ast}u,\Sigma_{\mathcal{T}(\tau)})\expe{c(\mathcal{T}(t)-\mathcal{T}(\tau))}+\int_{\mathcal{T}(\tau)}^{\mathcal{T}(t)}\expe{c(\mathcal{T}(t)-\lambda)}\norm{(\widetilde{D}\mathcal{T}^\ast u)\vert_{\Sigma_\lambda}}{H^s(\spinb(\Sigma_{\lambda}))}^2 \differ \lambda \\
&\leq & C \left(\mathcal{E}_s(\mathcal{T}^{\ast}u,\Sigma_{\mathcal{T}(\tau)})+\int_{\mathcal{T}(\tau)}^{\mathcal{T}(t)}\norm{(\widetilde{D}\mathcal{T}^\ast u)\vert_{\Sigma_\lambda}}{H^s(\spinb(\Sigma_{\lambda}))}^2 \differ \lambda\right) \quad,
\end{eqnarray*}
where $\widetilde{D}=\mathcal{T}^\ast \circ D \circ \mathcal{T}^\ast$. Because $\mathcal{T}$ is an involution, the $s$-energy and the last integral over the inhomogeneity are invariant under this time-orientation reversion according to \cref{timelambda}, which finally proves the statement.
\end{proof}
The following conclusion can be interpreted as \textit{uniqueness of the Cauchy problem for the Dirac equation}:
\begin{cor}\label{corenest2}
$u \in FE^{s}_\scomp(M,\timef,D)$ is uniquely determined by the inhomogeneity $Du$ and the initial condition $u\vert_{\Sigma_t}$ on a hypersurface $\Sigma_t$ for any $t \in \timef(M)$.
\end{cor}
\begin{proof}
The argumentation carries over literally from the proof of Corollary 11 in \cite{BaerWafo}, where Corollary 10 in the same reference is replaced here by Corollary \ref{corenest1}.
\end{proof}

\subsection{Cauchy problem for $D$}
\justifying
After all preparations the well-posedness of the inhomogeneous Cauchy problem for the Dirac equation for spinor fields with positive chirality can be proven in the same way, Bär and Wafo have done for the wave equation in \cite{BaerWafo}. The proof from this reference carries over literally to the setting of our interest, where the used corollaries 10 and 11 in \cite{BaerWafo} are replaced by Corollary \ref{corenest1} as well as by the uniqueness of solutions of the Dirac equation in Corollary \ref{corenest2} and only the restricted spinor on the initial hypersurface contributes. But for the sake of completeness and of later use we repeat the argument. 
\begin{theo}\label{inivpwell}
For a fixed $t \in \timef(M)$ and $s \in \R$ the map
\begin{equation}\label{inivpmap}
\mathsf{res}_t \oplus D \,\,:\,\, FE^s_\scomp(M,\timef,D) \,\,\rightarrow\,\, H^s_\comp(\spinb^{+}(\Sigma_t))\oplus L^2_{\loc,\scomp}(\timef(M),H^s_\loc(\spinb^{-}(\Sigma_\bullet)))
\end{equation}
is an isomorphism of topological vector spaces.
\end{theo}
\begin{proof}
One first checks the continuity of the map $\rest{t}\oplus D$, induced by the continuity of both summands: by definition $FE^s_\scomp(M,\timef,\spinb^{+}(M))$ is the union of all continuous functions from $\timef(M)$ to $H^s_\loc(\spinb^{+}(\Sigma_\bullet))$ with spatially compact support in $\mathcal{K} \subset M$. An intersection of $\mathcal{K}$ with any Cauchy hypersurface in the foliation of $M$ is a compact subset and since $H^s_\comp(\Sigma_{\bullet})$ is also defined as union over all compact subsets in any slice, it is enough to consider the restriction onto a fixed hypersurface $\Sigma_t$ as map between $C^0_\mathcal{K}(\timef(M),H^s_\loc(\spinb^{+}(\Sigma_\bullet)))$ and $H^s_{\mathcal{K}\cap\Sigma_t}(\spinb^{+}(\Sigma_t))$. The continuity follows immediatly from the estimation
\begin{equation*}
\norm{\rest{t}u}{H^s(\mathcal{K}\cap \Sigma_t ,\spinb(\Sigma_t))} \leq \max_{\tau \in \timef(M)}\SET{\norm{u\vert_{\Sigma_\tau}}{H^s(\mathcal{K}\cap \Sigma_{\tau},\spinb(\Sigma_t))}}=\norm{u}{\timef(M),\mathcal{K},0,s}
\end{equation*}
with the norm on $C^0_{\mathcal{K}}(\timef(M),H^s_\loc(\spinb^{+}(\Sigma_\bullet)))$ as defined in \cref{snclk}. The two inclusion mappings $H^s_{\mathcal{K}\cap\Sigma_\bullet}(\spinb^{+}(\Sigma_\bullet)) \hookrightarrow H^s_\comp(\spinb^{+}(\Sigma_\bullet))$ and $C^0_{\mathcal{K}}(\timef(M),H^s_\loc(\spinb^{+}(\Sigma_\bullet))) \hookrightarrow C^0_{\scomp}(\timef(M),H^s_\loc(\spinb^{+}(\Sigma_\bullet)))$ are continuous and the restriction map between $FE^s_{\scomp}(M,\timef,D)\subset FE^s_{\scomp}(M,\timef,\spinb^{+}(M))$ and $H^s_\comp(\spinb^{+}(\Sigma_t))$ for a fixed $\Sigma_t$ becomes continuous. $D$ as map from $FE^s_\scomp(M,\timef,D)$ into the codomain $L^2_{\loc,\scomp}(\timef(M),H^s_\loc(\spinb^{-}(\Sigma_\bullet)))$ is continuous on $FE^s_\scomp(M,\timef,D)$, implying the whole map to be continuous.\\
\\
Now one shows, that the map is bijective with continuous inverse. For this step one constructs the inverse map of \cref{inivpmap}. Take $K \subset \Sigma_t$ compact for $t \in 	\timef(M)$ fixed. The well-posedness of the Cauchy problem for the Dirac equation with smooth and compactly supported initial data on $\Sigma_t$ (see Theorem 4 in \cite{AndBaer}) states, that for given $u_0 \in C^\infty_K(\spinb^{+}(\Sigma_t))$ and $f \in C^\infty_{\Jlight{}(K)}(\spinb^{-}(M))$ there exists a solution $u\in C^\infty(\spinb^{+}(M))$ of the Dirac equation with inhomogeneity $f$ and initial value $u_0=u\vert_{\Sigma_t}$, which has support in $\Jlight{}(K)$ by finite propagation speed. Since $C^\infty(\spinb^{+}(M))\subset FE^s_\scomp(M,\timef,\spinb^{+}(M))$ for all $s\in \R$ one can apply Corollary \ref{corenest1} in order to estimate the norm of $C^0_\scomp(\timef(M),H^s_\loc(\spinb^{+}(\Sigma_\bullet)))$: let $I \subset \timef(M)$ be a subintervall and $t \in \timef(M)$ a fixed initial time, then
\begin{eqnarray*}
\norm{u}{I,\Jlight{}(K),0,s}^2 &=& \max_{\tau \in I}\SET{\norm{u\vert_{\Sigma_\tau}}{H^s_\loc(\spinb(\Sigma_t))}^2}= \max_{\tau \in I}\SET{\mathcal{E}(u,\Sigma_\tau)} \\
&\leq& C \max_{\tau \in I}\SET{\norm{u\vert_{\Sigma_t}}{H^s_\loc(\spinb(\Sigma_t))}^2+\norm{Du}{I,\Jlight{}(K),s}^2}\\
&=&C \left(\norm{u_0}{H^s(\spinb(\Sigma_t))}^2+\norm{f}{I,\Jlight{}(K),s}^2\right) \quad ,
\end{eqnarray*} 
where the estimation constant comes from the used corollary, thus it is not depending on the smooth data of the Cauchy problem. This result implies, that the continuous map $(u_0,f)\mapsto u$ from the Cauchy problem in \cite{AndBaer} can be extended to a continuous map 
\begin{equation*}
H^s_K(\spinb^{+}(\Sigma_t))\oplus L^2_{\loc,\Jlight{}(K)}(\timef(M),H^s_\loc(\spinb^{-}(\Sigma_\bullet)))\,\,\rightarrow\,\,C^0_{J(K)}(\timef(M),H^s_\loc(\spinb^{+}(\Sigma_\bullet)))
\end{equation*}
Since $\Jlight{\pm}(K)$ is closed for $K \subset \Sigma_t \subset M$ compact, also $\Jlight{}(K)$ is closed and thus spatially compact. Therefore the continuity of the map 
\begin{equation*}
H^s_\comp(\spinb^{+}(\Sigma_t))\oplus L^2_{\loc,\scomp}(\timef(M),H^s_\loc(\spinb^{-}(\Sigma_\bullet)))\,\,\rightarrow\,\,FE^s_\scomp(M,\timef,\spinb^{+}(M))
\end{equation*}
is proven, as the inclusions $H^s_K \hookrightarrow H^s_\comp$, $L^2_{\loc,\Jlight{}(K)}\hookrightarrow L^2_{\loc,\scomp}$ and $C^0_{\Jlight{}(K)}\hookrightarrow C^0_{\scomp}$ are continuous. The concrete inverse is the mapping
\begin{equation}\label{solmap}
H^s_\comp(\spinb^{+}(\Sigma_t))\oplus L^2_{\loc,\scomp}(\timef(M),H^s_\loc(\spinb^{-}(\Sigma_\bullet)))\,\,\rightarrow\,\,FE^s_\scomp(M,\timef,D) \quad.
\end{equation}
The composition \cref{inivpmap} after \cref{solmap} clearly gives the identity after fixing one hypersurface for the initial data. The converse composition starts with a solution, from which one extracts the initial data and the inhomogeneity and solves again by \cref{solmap}. From Corollary \ref{corenest2} the solution is unique and the result coincides with the input. This implies, that also this composition gives the identity and thus invertibility/ bijectivity of \cref{inivpmap}. The continuity of the inverse follows from the composition of $D$ with \cref{solmap}, which is a restriction on the second summand. Thus the composition is continuous and because $D$ is continuous on $FE^s_{\scomp}(M,\timef,D)$ the claim follows. Summarizing all results shows, that \cref{inivpmap} is indeed an isomorphism.  
\end{proof}
\newpage
\noindent The well-posedness of the homogeneous Cauchy problem for the Dirac equation follows immediately:
\begin{cor}\label{homivpwell}
For a fixed $t \in \timef(M)$ and $s \in \R$ the map
\begin{equation*}
\mathsf{res}_t  \,\,:\,\, FE^s_\scomp(M,\timef,\kernel{D}) \,\,\rightarrow\,\, H^s_\comp(\spinb^{+}(\Sigma_t))
\end{equation*}
is an isomorphism of topological vector spaces.
\end{cor}
These two results lead to the following consequences: 

\begin{cor}\label{corivp1}
For any $s\in \R$
\begin{itemize}
\item[(a)] $C^\infty_{\scomp}(\spinb^{+}(M)) \subset FE^s_\scomp(M,\timef,D)$ dense, 
\item[(b)] $C^\infty_{\scomp}(\spinb^{+}(M))\cap\kernel{D} \subset FE^s_\scomp(M,\timef,\kernel{D})$ dense.
\end{itemize}
\end{cor}
\begin{proof}
The proof can be taken from Corollary 15 in \cite{BaerWafo} with the only difference, that Theorem 13 and Corollary 14 in the reference are replaced by the well-posedness for the homogeneous and inhomogeneous Cauchy problem from above. 
\end{proof}
A consequence is, that solutions of the Dirac equation have \textit{finite propagation speed}:
\begin{cor}
A solution $u$ with $u\vert_{\Sigma_t}=u_0$ for any leaf $\Sigma_t$, $t \in \timef(M)$, and $Du=f$ satisfies \\
$\supp{u}\subset \Jlight{}(K)$, where $K \subset M$ is a compact subset, satisfying $\supp{u_0}\cup\supp{f} \subset K$.
\end{cor}
\begin{proof}
Since $C^\infty_{\scomp}(\spinb^{+}(M))$ is dense in $FE^s_\scomp(M,\timef,D)$, the finiteness of propagation speed in Theorem 4 of \cite{AndBaer} can be extended to finite $s$-energy spinors.
\end{proof}
Another consequence of the Corollary \ref{corivp1} is an optimization of the assumed regularity in Proposition \ref{enesttheorem}, which can be stated as Corollary 17 in \cite{BaerWafo}.
The last conclusion from the density of smooth sections in finite energy spinors is the independence of the Cauchy temporal function $\timef$ for those finite energy spinors, which are solutions of the homogeneous Dirac equation. 
\begin{cor}
Given two Cauchy temporal functions $\timef$ and $\timef'$ on $M$, then for all $s \in \R$ $$FE^s_\scomp(M,\timef,\kernel{D})=FE^s_\scomp(M,\timef',\kernel{D})\quad.$$
\end{cor}
\begin{proof}
The detailed proof can be taken from Corollary 18 in \cite{BaerWafo} with the uniqueness of the Cauchy problem for the Dirac equation in Corollary \ref{corenest2}.
\end{proof}
Thus we can simplify notation to $FE^s_\scomp(M,\kernel{D})$, in order to stress this independence of the temporal function.
\newpage

\subsection{Cauchy problem for $D_{-}$ and $\Dirac$}
\justifying
The proof of Theorem \ref{inivpmap} in detail provides a full proof of the main result, where only the dependencies on the chirality are left to check.
\begin{proof}
(Theorem \ref{maintheo}) If $(n+1)$ is even, the spaces $FE^s_\scomp(M,\timef,\spinb(M))$, $FE^s_\scomp(M,\timef,\spinb^{-}(M))$, \\ $FE^s_\scomp(M,\timef,D_{-})$, 
$FE^s_\scomp(M,\timef,\Dirac)$, $FE^s_\scomp(M,\timef,\kernel{D_{-}})$ and $FE^s_\scomp(M,\timef,\kernel{\Dirac})$ are defined as in \cref{finensec}, \cref{finensol} and \cref{finenkern}, where either one needs to swap the bundles $\spinb^{\pm}\rightarrow\spinb^{\mp}$ for $D_{-}$ and one replaces $\spinb^{\pm}(M)$ and $\spinb^{\pm}(\Sigma_\bullet)$ by $\spinb(M)$ and $\spinb(\Sigma_\bullet)^{\oplus 2}$ respectively. Recapitulating the proof of the energy estimate, the only influence of the chirality appears in expressing the spinorial covariant derivative with respect to $\upnu$ in terms of $D_{-}$ and operators on the hypersurfaces, where $B_t$ has to be replaced by $B_{t,-}$. But since the result of Lemma \ref{lemenest1} (b), (c) and the latters extension to the $L^2$-norm \cref{extendB} are independent of the chirality, the proof carries over for this case. Exchanging the chiralities in the assumptions then leads to 
\begin{prop}\label{enestdneg}
Let $I \subset \timef(M)$ be a closed intervall, $K\subset M$ compact and $s\in \R$; there exists a constant $C>0$, depending on $K$ and $s$, such that
\begin{equation*}
\mathcal{E}_s(u,\Sigma_{t_1}) \leq \mathcal{E}_s(u,\Sigma_{t_0})\expe{C(t_1-t_0)}+\int_{t_0}^{t_1}\expe{C(t_1-\tau)}\norm{D_{-}u\vert_{\Sigma_\tau}}{H^s_\loc(\spinb(\Sigma_\tau))}^2 \differ \tau
\end{equation*}
applies for all $t_0,t_1 \in I$ with $t_0 < t_1$ and for all $u \in FE^{s+1}_{\scomp}(M,\timef,\spinb^{-}(M))$ with support $\supp{u}\subset \Jlight{}(K)$ and $Du \in FE^{s}_{\scomp}(M,\timef,\spinb^{+}(M))$. 
\end{prop} 
A uniqueness result for the Dirac equation with respect to $D_{-}$ follows as conclusion.
One earns the well-posedness of the (in-) homogeneous Cauchy problem for the Dirac equation with respect to $D_{-}$ literally from the last subsection, where one only needs to swap the chiralities in the proof:
\begin{theo}\label{cauchyinvdneghom}
For a fixed $t \in \timef(M)$ and $s \in \R$ the maps
\begin{eqnarray*}
\mathsf{res}_t \oplus D_{-} &:& FE^s_\scomp(M,\timef,D_{-}) \,\,\rightarrow\,\, H^s_\comp(\spinb^{-}(\Sigma_t))\oplus L^2_{\loc,\scomp}(\timef(M),H^s_\loc(\spinb^{+}(\Sigma_\bullet)))\\
\text{and}\,\,\mathsf{res}_t &:& FE^s_\scomp(M,\timef,\kernel{D_{-}}) \,\,\rightarrow\,\, H^s_\comp(\spinb^{-}(\Sigma_t))
\end{eqnarray*}
are isomorphisms of topological vector spaces.
\end{theo}
All following side results can be taken over as well, but they are not important here.\\
\\
The inhomogeneous Cauchy problem for the Dirac equation with respect to $\Dirac$ is well posed, if and only if the inhomogeneous Cauchy problems for the Dirac equations with respect to $D_{\pm}$ are well posed. The main result is thus proven by using the implication from right to left and replacing the bundles $\spinb^{\pm}(\Sigma_t)$ to $\spinb(M)\vert_{\Sigma_t}=\spinb(\Sigma_t)^{\oplus 2}$.\\
\\
If the number of spatial dimensions is even, there is nothing more to show, since 
the whole Dirac operator $\Dirac$ along the hypersurface \cref{dirachypodd} is formally equivalent with $D$ along the hypersurface. 
This then proves the main result for $n$ odd with $FE^s_\scomp(M,\timef,\Dirac)$, which formally equals $FE^s_\scomp(M,\timef,D)$ and $\spinb^{+}(\Sigma_t)$ replaced by $\spinb(M)\vert_{\Sigma_t}=\spinb(\Sigma_t)$.
\end{proof}

\addtocontents{toc}{\vspace{-3ex}}
\section{Wave evolution operator}\label{chap:feynman}
The well-posedness of the homogeneous Cauchy problem in Corollary \ref{homivpwell} motivates the definition of a \textit{Dirac-wave evolution operator}.
\begin{defi}
For a globally hyperbolic manifold $M$ and $t_1,t_2 \in \timef(M)$ the (Dirac-) wave evolution operator is an isomorphism of topological vector spaces
\begin{equation*}
Q(t_2,t_1)\,\,:\,\,H^s_\comp(\spinb^{+}(\Sigma_{1}))\,\,\rightarrow\,\,H^s_\comp(\spinb^{+}(\Sigma_{2}))\,\, ,
\end{equation*}
defined as $Q(t_2,t_2)=\rest{t_2}\circ(\rest{t_1})^{-1}$.
\end{defi} 

The same operator occurs in \cite{BaerStroh} for compact hypersurfaces and in \cite{BaerStroh2} for square-integrable sections on the non-compact hypersurfaces. The wave evolution operator in our setting acts between compactly supported Sobolev sections of any degree over non-compact, but complete hypersurfaces. The same properties of $Q$ as in \cite{BaerStroh} are given as well. Only the unitarity property is more involved, since the Cauchy boundaries in our setting aren't compact:
\begin{lem}\label{propsofprop}
For any $s \in \R$ and $t,t_1,t_2,t_3 \in \timef(M)$ the following hold
\begin{itemize}
\item[(a)] $Q(t_3,t_2)\circ Q(t_2,t_1)=Q(t_3,t_1)$,
\item[(b)] $Q(t,t)=\id{H^s_\comp(\spinb^{+}(M)\vert_{\Sigma_t})}$ and $Q(t_1,t_2)=Q^{-1}(t_2,t_1)$,
\item[(c)] $Q(t_2,t_1)$ is unitary for $s=0$, if $M$ is temporal compact.
\end{itemize}
\end{lem}
\begin{proof} 
(a) and (b) follow by the same reasoning as in \cite{BaerStroh}. For (c) the well-posedness of the homogeneous Cauchy problem implies, that any initial value $u \in H^s_\comp(\spinb^{+}(M)\vert_{\Sigma_t})$ for $t \in \timef(M)$ is uniquely related to a finite energy spinor $\psi \in FE^s_{\scomp}(M,\kernel{D})$, such that $\psi\vert_{\Sigma_t}=u$. Proposition \ref{boots} ensures, that for $s > \frac{n}{2}+2$ fixed one even has $\psi \in C^1_{\scomp}(\spinb^{+}(M))$. Temporal compactness of $M$ even implies $\psi \in C^1_{\comp}(\spinb^{+}(M))$ and Propostion \ref{diracselfadprop} leads to the claim: let $u=\psi\vert_{\Sigma_{t_1}}$ for $t_1,t_2 \in \timef(M)$, such that $\intervallc{t_1}{t_2}{}\subset \timef(M)$, then for $\psi\in C^1_{\comp}(\spinb^{+}(M))\cap\kernel{D})$
\begin{eqnarray*}
0&=& \int_M \idscal{1}{\spinb(M)}{\Dirac \psi}{\psi}+ \idscal{1}{\spinb(M)}{\psi}{\Dirac \psi} \dvol{} = \int_M \idscal{1}{\spinb(M)}{\Dirac \psi}{\psi}+ \idscal{1}{\spinb(M)}{\psi}{\Dirac \psi} \dvol{}\\
&\stackrel{\cref{diracselfad},\cref{restspinbundmet}}{=}& \int_{\Sigma_{2}} \dscal{1}{\spinb(\Sigma_2)}{\psi\vert_{\Sigma_{2}}}{\psi\vert_{\Sigma_{2}}} \dvol{\Sigma_{2}}-\int_{\Sigma_{1}} \dscal{1}{\spinb(\Sigma_{1})}{\psi\vert_{\Sigma_{1}}}{\psi\vert_{\Sigma_{1}}} \dvol{\Sigma_{1}} \\
&=& \int_{\Sigma_{2}} \dscal{1}{\spinb(\Sigma_{2})}{Q(t_2,t_1)u}{Q(t_2,t_1)u} \dvol{\Sigma_{2}}-\int_{\Sigma_{t_1}} \dscal{1}{\spinb(\Sigma_{1})}{u}{u} \dvol{\Sigma_{1}} \\
&=&\norm{Q(t_2,t_1)u}{L^2(\spinb^{+}(\Sigma_{2}))}^2-\norm{u}{L^2(\spinb^{+}(\Sigma_{1}))}^2 \quad,
\end{eqnarray*} 
where w.l.o.g. it is assumed, that the compact support of $\psi$ intersects both boundary hypersurfaces $\Sigma_{1}$ and $\Sigma_{2}$; otherwise one gets a trivial identification $\norm{Q(t_2,t_1)u}{L^2(\spinb^{+}(\Sigma_{2}))}=0=\norm{u}{L^2(\spinb^{+}(\Sigma_{1}))}$. 
\end{proof}

The next step is relating this operator to boundary conditions on the Cauchy hypersurfaces. The collar neighborhood theorem implies, that the manifold is diffeomorphic to a product structure near each Cauchy boundary. The metric near the boundary $\bound M= \Sigma_{1}\sqcup \Sigma_2$ can be deformed in such a way, that it becomes a \textit{ultra-static} metric $\met=-\differ t^{\otimes 2}+\met_{t_j}$ near $\Sigma_j$ for $j\in \SET{1,2}$ and each mean curvature $H_t$ of the spacelike boundary hypersurfaces is vanishing identically. Both Dirac operators are then given by
\begin{equation}\label{diracprodneigh}
D_{\pm,j}:=\upbeta \left(\partial_t \mp\Imag A_j \right)
\end{equation}
along $\Sigma_j$ with past directed timelike vector $\upnu=-\partial_t$, $\upbeta=\cliff{\upnu}$ and $A_j=A_{t_j}$. Since $\Sigma_1$ and $\Sigma_2$ are Cauchy boundaries and the neigborhoods of each hypersurface can be considered as globally hyperbolic manifold with an ultra-static metric, both slices as well as all other in the neighborhood are complete. A proof of this equivalence is given in \cite{Kay}. Theorem 1.17 in \cite{GroLaw} guarantees, that $A_{1}$ and $A_{2}$ are essentially self-adjoint. According to \cite{Baerspec2} and \cite{Baerspec} the spectrum of these unique self-adjoint extensions, still denoted as $A_{1}$ and $A_{2}$, decomposes disjointly into a point and continuous spectrum. The eigenvalues with their multiplicities in the point spectrum are real and their eigenspaces are orthogonal to each other respectively, but these spaces have in general infinitely many dimensions (multiplicites) and the point spectrum does not need to be discrete as for example on closed manifolds. The continuous spectrum is real valued as well and their eigensections are smooth, but not square-integrable. \\
\\
For \cref{diracprodneigh} boundary conditions are introduced in the same way Atiyah, Patodi and Singer did in \cite{APSI}: define as spectral projections $P_I(t_j):=\chi_I(A_j)$, where $I$ is a measurable intervall in $\R$, representing a part of the real spectrum of the self-adjoint extented boundary Dirac operator $A_{j}$. In order to see, that these operators are indeed projectors (denote for $A$ either $A_1$ or $A_2$) use
\begin{theo}[Browder-Garding; Theorem 2.2.1 in \cite{ramachandran} and Theorem 1 in \cite{browder}]
If $A\in \Diff{m}{}(\Sigma,E)$ is essentially self-adjoint, then there exist a sequence $(f_j)_j$ of measurable sections $f_j\,:\,\R\times \Sigma\,\rightarrow\,E$ with $f_j(\lambda,\bullet)\in C^\infty(\Sigma,E)$ and a corresponding sequence of measures $(\mu_j)_j$ in $\R$, such that
\begin{enumerate}
\item $Af_j(\lambda,p)=\lambda f_j(\lambda,p)$;
\item the maps
\begin{eqnarray*}
\mathscr{V}_j &:& C^\infty_\comp(\Sigma,E)\quad \rightarrow \quad \C \\
& & u \quad \mapsto \quad (\mathscr{V}_j u)(\lambda):=\dscal{1}{L^2(\Sigma,E)}{u}{f_j(\lambda,\bullet)}
\end{eqnarray*}
can be extended to an isometry $\mathscr{V}\,:\,L^2(\Sigma,E)\,\rightarrow\,\bigoplus_{j}L^2(\R,\mu_j)$,
\item $F(A)(\mathscr{V}u)(\lambda)=F(\lambda)(\mathscr{V}u)(\lambda)$ for any real valued Borel function $F$ on $\R$,
\item $\dom{}{F(A)}=\SET{u \in L^2(\Sigma,E)\,\Big\vert\,\sum_j \int_\R \absval{F(\lambda)}^2\absval{(\mathscr{V}_j u)(\lambda)}^2 \differ \mu_j(\lambda) < \infty}$,
\item $\norm{\mathscr{V} u}{\bigoplus_j L^2(\R,\mu_j)}^2=\sum_{j}\int_\R \absval{(\mathscr{V}_j u)(\lambda)}^2 \differ \mu_j(\lambda)=\norm{u}{L^2(\Sigma,E)}^2$ and
\item $u(p)=\sum_j \int_\R (\mathscr{V}_j u)(\lambda) f_j(\lambda,p) \differ \mu_j(\lambda)$.
\end{enumerate} 
\end{theo}
The action of the projector is described by the third and sixth property of this theorem, leading to the following characteristics: given two intervalls $I,J \subset \R$, then 
\begin{itemize}
\item \textit{Idempotence}: $(\chi_I^2(A)u)(p)=(\chi_I(A)u)(p)$;
\item $\chi_{\R}(A)=\id{}$, $\chi_{\emptyset}(A)=0$ and $\chi_{I^\complement}(A)=\id{}-\chi_{I}(A)$;
\item \textit{Self-adjointness:} This follows from the essential self-adjointness of $A$ and that the characteristic function is real-valued: $(\chi_I^\ast(A)u)(p)=(\chi_I(A)u)(p)$;
\item $\chi_{I \cap J}(A)=\chi_{I}(A)\circ\chi_{J}(A)$, $\chi_{I \cup J}(A)=\chi_{I}(A)+\chi_{J}(A)-\chi_{I \cap J}(A)$ 
and for $J \subset I$ \\ $\chi_{I \setminus J}(A)=\chi_{I}(A)-\chi_{I}(A)\circ\chi_{J}(A)$. 
\end{itemize}
\begin{lem}\label{L2projector}
$\chi_I(A)\in \mathscr{L}(L^2(\Sigma,E))$, i.e. $\chi_I(A)$ is a linear bounded operator.
\end{lem}
\begin{proof}
This is a consequence of the fourth and fifth property in the Browder-Garding theorem.
\end{proof}
The \textit{Atiyah-Patodi-Singer (APS) boundary conditions} are then defined as follows: 
\begin{equation}\label{apsbc}
\begin{array}{ccl}
P_{\intervallro{0}{\infty}{}}(t_1)(u\vert_{\Sigma_{1}}) &=& 0 \\
P_{\intervalllo{-\infty}{0}{}}(t_2)(u\vert_{\Sigma_{2}}) &=& 0 \\
\text{for positive chirality} && 
\end{array}
\,\,\quad\text{and}\,\,\quad
\begin{array}{ccl}
P_{\intervalllo{-\infty}{0}{}}(t_1)(u\vert_{\Sigma_{1}}) &=& 0 \\
P_{\intervallro{0}{\infty}{}}(t_2)(u\vert_{\Sigma_{2}}) &=& 0 \\
\text{for negative chirality} && \quad.
\end{array}
\end{equation}
Another set of boundary conditions are the \textit{anti Atiyah-Patodi-Singer (aAPS or anti-APS) boundary conditions}:
\begin{equation}\label{aapsbc}
\begin{array}{ccl}
P_{\intervallo{-\infty}{0}{}}(t_1)(u\vert_{\Sigma_{1}}) &=& 0 \\
P_{\intervallo{0}{\infty}{}}(t_2)(u\vert_{\Sigma_{2}}) &=& 0 \\
\text{for positive chirality} &&
\end{array}
\,\,\quad\text{and}\,\,\quad
\begin{array}{ccl}
P_{\intervallo{0}{\infty}{}}(t_1)(u\vert_{\Sigma_{1}}) &=& 0 \\
P_{\intervallo{-\infty}{0}{}}(t_2)(u\vert_{\Sigma_{2}}) &=& 0 \\
\text{for negative chirality} && \quad.
\end{array}
\end{equation}
Using 
\begin{equation*}
\id{L^2(\spinb^{\pm}(\Sigma_1))}=P_{\intervallro{0}{\infty}{}}(t_1)\oplus P_{\intervallo{-\infty}{0}{}}(t_1)\quad\text{and}\quad \id{L^2(\spinb^{\pm}(\Sigma_2))}=P_{\intervallo{0}{\infty}{}}(t_2)\oplus P_{\intervalllo{-\infty}{0}{}}(t_2) 
\end{equation*}
we obtain $L^2$-orthogonal splittings:
\begin{equation*}
\begin{array}{ccl}
L^2(\spinb^{\pm}(\Sigma_1))&=&L^2_{\intervallro{0}{\infty}{}}(\spinb^{\pm}(\Sigma_1))\oplus L^2_{\intervallo{-\infty}{0}{}}(\spinb^{\pm}(\Sigma_1)) \\
L^2(\spinb^{\pm}(\Sigma_2))&=&L^2_{\intervallo{0}{\infty}{}}(\spinb^{\pm}(\Sigma_2))\oplus L^2_{\intervalllo{-\infty}{0}{}}(\spinb^{\pm}(\Sigma_2))
\end{array}
\quad,
\end{equation*}
where $L^2_{I}(\spinb^{\pm}(\Sigma_j)):=P_I(t_j)\left(L^2(\spinb^{\pm}(\Sigma_j))\right)$. For those cases, where Sobolev spaces are Hilbert spaces (e.g. $\Sigma$ compact), a similar splitting can be obtained: 
\begin{equation*}
H^s_{I}(\spinb^{\pm}(\Sigma_j)):=P_I(t_j)\left(H^s_{\bullet}(\spinb^{\pm}(\Sigma_j))\right)\quad.
\end{equation*}
Note, that APS and aAPS boundary conditions are orthogonal to each other.
Applying this splitting on the wave evolution operator allows to rewrite the operator in term of a ($2\times 2$)-matrix:
\begin{equation}\label{Qmatrix}
Q(t_2,t_1)=\left(\begin{matrix}
Q_{++}(t_2,t_1) & Q_{+-}(t_2,t_1) \\
Q_{-+}(t_2,t_1) & Q_{--}(t_2,t_1)
\end{matrix}\right) \quad ,
\end{equation}
where the entries are
\begin{footnotesize}
\begin{equation*}
\begin{array}{cl}
Q_{++}(t_2,t_1) :=& P_{\intervallro{0}{\infty}{}}(t_2)\circ Q(t_2,t_1)\circ P_{\intervallo{0}{\infty}{}}(t_1) \\
Q_{--}(t_2,t_1) :=& P_{\intervalllo{-\infty}{0}{}}(t_2)\circ Q(t_2,t_1)\circ P_{\intervallo{-\infty}{0}{}}(t_1)
\end{array}
\begin{array}{cl}
Q_{+-}(t_2,t_1) :=& P_{\intervallro{0}{\infty}{}}(t_2)\circ Q(t_2,t_1)\circ P_{\intervallo{-\infty}{0}{}}(t_1) \\
Q_{-+}(t_2,t_1) :=& P_{\intervalllo{-\infty}{0}{}}(t_2)\circ Q(t_2,t_1)\circ P_{\intervallo{0}{\infty}{}}(t_1) \,\,
\end{array}
\end{equation*}
\end{footnotesize}
\noindent These matrix entries will be referred on as \textit{spectral entries}. The rest of this subsection is devoted in investigating the operational character and mapping properties of $Q(t_2,t_1)$ and these entries. The unitarity property on $L^2$-sections in Lemma \ref{propsofprop} (c) implies, that the off-diagonal entries are isomorphisms between the kernels of the diagonal entries and their adjoints.
\begin{lem}\label{kernelisoQ}
The operators $Q_{+-}(t_2,t_1)$ and $Q_{-+}(t_2,t_1)$ restrict to isomorphisms 
\begin{equation*}
\begin{array}{ccccc}
Q_{+-}(t_2,t_2)&:&\kernel{Q_{--}(t_2,t_1)}&\rightarrow& \kernel{(Q_{++}(t_2,t_1))^{\ast}} \\
Q_{-+}(t_2,t_2)&:&\kernel{Q_{++}(t_2,t_1)}&\rightarrow& \kernel{(Q_{--}(t_2,t_1))^{\ast}} 
\end{array}
\quad .
\end{equation*}
\end{lem}
\begin{proof}
The result is shown in \cite{BaerStroh} as Lemma 2.5, which can be taken without any further modifications, since the proof is purely algebraic as explained below. The matrix representation \cref{Qmatrix} is done with respect to sections in $L^2(\spinb^{+}(\Sigma_j))$, $j\in \SET{1,2}$, then the above Lemma \ref{propsofprop} can be applied for $(Q(t_2,t_1))^\ast Q(t_2,t_1)=\id{L^2(\spinb^{+}(\Sigma_{1}))}$ and $Q(t_2,t_1)(Q(t_2,t_1))^{\ast}=\id{L^2(\spinb^{+}(\Sigma_{2}))}$. Translating these relations in terms of the matrix entries leads to the system of equations:
\begin{equation}\label{Qsys1}
\begin{array}{ccl}
(Q_{++}(t_2,t_1))^\ast Q_{++}(t_2,t_1)+(Q_{-+}(t_2,t_1))^\ast Q_{-+}(t_2,t_1)&=& \id{}, \\
(Q_{+-}(t_2,t_1))^\ast Q_{+-}(t_2,t_1)+(Q_{--}(t_2,t_1))^\ast Q_{--}(t_2,t_1)&=& \id{}, \\
(Q_{++}(t_2,t_1))^\ast Q_{+-}(t_2,t_1)+(Q_{-+}(t_2,t_1))^\ast Q_{--}(t_2,t_1)&=& 0, \\
(Q_{+-}(t_2,t_1))^\ast Q_{++}(t_2,t_1)+(Q_{--}(t_2,t_1))^\ast Q_{-+}(t_2,t_1)&=& 0 
\end{array}
\end{equation}
and similar equations from the second unitarity property.
\end{proof}

In order to achieve regularity properties, one needs to reinterpret $Q$ as Fourier integral operator: $\FIO{m}$ denotes the set of Fourier integral operators of order $m \in \R$, i.e. the set of integral operators with Schwartz kernels, which are locally oscillatory integrals. A coordinate invariant description of these kernels is realized by conormal distributions of order $m$ with respect to a conic Lagrangian submanifold $\mathsf{C}'$. This is a submanifold in the cotangent bundle of the domain of the Schwartz kernel as symplectic manifold, on which the symplectic form vanishes and the zero section is excluded. $\mathsf{C}$ without the apostrophe denotes the related canonical relation. We refer to the \hyperref[chap:app2]{appendix} for details.
\begin{prop}\label{Qfourier} 
For all $s \in \R$ one has $Q\in \FIO{0}(\Sigma_{2};\mathsf{C}'_{1\rightarrow 2};\Hom(\spinb^{+}(\Sigma_1),\spinb^{+}(\Sigma_2)))$ properly supported, with canonical relation 
\begin{equation}\label{Qcanrel}
\begin{split}
\mathsf{C}_{1\rightarrow 2} &=\mathsf{C}_{1\rightarrow 2\vert +}\sqcup \mathsf{C}_{1\rightarrow 2\vert -}\,\,, \quad\text{where}\\
\mathsf{C}_{1\rightarrow 2\vert \pm} &=\SET{((x_{\pm},\xi_{\pm}),(y,\eta))\in \dot{T}^\ast \Sigma_2\times \dot{T}^\ast \Sigma_1\,\vert\, (x_{\pm},\xi_{\pm})\sim (y,\eta)} \quad, 
\end{split}
\end{equation} 
which is a canonical graph with respect to the lightlike (co-) geodesic flow; and principal symbol
\begin{equation}\label{Qprinsym}
\bm{\sigma}_{0}(Q)(x,\xi_{\pm};y,\eta)=\pm\frac{1}{2} (\norm{\eta}{\met_{t_1}}(y))^{-1}\left(\mp\norm{\xi_{\pm}}{\met_{t_2}}(x)\upbeta+\Cliff{t_2}{\xi^{\sharp}_{\pm}}\right)\circ \mathpzc{P}_{(x,\varsigma_{\pm})\leftarrow (y,\zeta_{\pm})} \circ \upbeta
\end{equation}
with $(y,\zeta_{\pm})\in T^\ast_{\Sigma_{1}}M$ and $(x,\varsigma_{\pm})\in T^\ast_{\Sigma_{2}}M$ restricting to $(y,\eta)$ and $(x,\xi_{\pm})$ respectively.
\end{prop}
Here $\mathpzc{P}_{(x,\varsigma_{\pm})\leftarrow (y,\zeta_{\pm})}$ denotes parallel transport from $(y,\zeta^\sharp_{\pm})$ to $(x,\varsigma^\sharp_{\pm})$ and $\norm{\bullet}{\met_{t_j}}$ the norm of a covector with respect to the dual of the metric $\met_{t_j}$ on the hypersurface $\Sigma_j$. $T^\ast_{\Sigma_{j}}M$ is the restriction of $T^\ast M$ to $\Sigma_j$.
\begin{proof} 
The main features can be taken from Lemma 2.6 in \cite{BaerStroh}: in order to describe $Q$ as FIO one uses the fact, that $\Dirac^2$ and thus $D_{\pm}D_{\mp}$ are normally hyperbolic and Theorem \ref{cauchynormhyphom} in our appendix implies the existence of a solution operator with canonical relation $\mathcal{C}$ in \cref{canrelationnormhyp}: $\mathcal{G}(t_1)\in \FIO{-5/4}(M;\mathcal{C}';\Hom(\spinb^{-}(\Sigma_1),\spinb^{-}(M)))$. As shown in \cite{BaerStroh} the formal expression of $Q$ is
\begin{equation}\label{QasFIO}
Q(t_2,t_1)=\rest{\Sigma_2}\circ D_{-}\circ \mathcal{G}(t_1)\circ \upbeta \quad.
\end{equation}

In order to show, that \cref{QasFIO} is well defined on non-compact manifolds, the compositions of the two canonical relations $\mathsf{C}(\inclus^\ast)$ in \cref{relrest} \& $\mathcal{C}$ and of the operators are well defined. The latter one is the additional feature, we need to check on non-compact manifolds:
\begin{itemize}
\item[(1)] $\rest{\Sigma_2}$ is properly supported;
\item[(2)] $\mathcal{G}(t_1)$ is properly supported;
\item[(3)] $D_{-}\circ \mathcal{G}(t_1) \in \FIO{-1/4}(M;\mathcal{C}';\Hom(\spinb^{-}(\Sigma_1),\spinb^{+}(M)))$ and properly supported;
\item[(4)] $\mathsf{C}(\inclus^\ast)\circ \mathcal{C}$ transversal and proper.
\end{itemize}
(1) follows from the definition of the restriction and the corestriction operator \cref{restop} and \cref{corestop} in the appendix, where the claimed property is stated in Lemma \ref{restcorestfio}. \\
\\ 
(2) Theorem \ref{cauchynormhyphom} states, that $\mathcal{G}(t_1)$ is a continuous map between smooth sections. Pairing this with a compactly supported distribution shows, that the action of the adjoint $\mathcal{G}(t_1)^\dagger$ on a compactly supported distribution yields again an element of the dual space, hence is itself a compactly supported distribution, satisfying \cref{propsuppb}. Property \cref{propsuppa} can be shown directly: let $\pi_1:\,M\times\Sigma_1\rightarrow M$ be the projection on the first factor:
According to Theorem \ref{cauchynormhyphom} the support of (the Schwartz kernel of) $\mathcal{G}$ is contained in 
\begin{equation*}
S:=\SET{(p,x)\in M \times \Sigma_1\,\vert\, x \in \Jlight{-}(p)\cap \Sigma_1} \quad,
\end{equation*}
i.e. for every point $p$ in $M$ only those points on the initial hypersurface $\Sigma_1$ contribute, which are lying inside its causal past $\Jlight{-}(p)=\Jlight{-}(\SET{p})$. Since $\Jlight{-}(\SET{p})$ is itself spatially compact 
and $\Sigma_1$ a Cauchy hypersurface, $\Jlight{-}(\SET{p})\cap \Sigma_1$ is compact and hence only points from this subset  contribute to the solution at the point $p$. Take a compact subset $K$ in $M$; its preimage under the first projection is $(\pi_1)^{-1}(K)=K \times \Sigma_1$. The intersection $S\cap (\pi_1)^{-1}(K_1)$ contains only points in $K \times (\Jlight{-}(K) \cap \Sigma_1)$. As $\Jlight{-}(K) \subset \Jlight{}(K)$ is again spatially compact by the same argument, it implies $\Jlight{-}(K)\cap\Sigma_1$ to be compact. Hence the intersection with $S$ is compact and thus $\pi_1$ restricted to the support is a compact subset. \\  
\\
(3) Since $D_{-} \in \Diff{1}{}(M,\Hom(\spinb^{-}(M),\spinb^{+}(M)))$ is properly supported, one can take the composition with $\mathcal{G}(t_1)\in \FIO{-5/4}(\Sigma,M;\mathcal{C}';\Hom(\spinb^{-}(\Sigma_1),\spinb^{-}(M)))$: since differential operators on $M$ can be interpreted as FIO from $M$ to $M$ with the conormal bundle of the diagonal/the graph of the identity with respect to $M$ (see (d) in Lemma \ref{fioprop}), the composition $N^\ast \mathrm{diag}(M)\circ \mathcal{C}$ is proper and transversal and results in $\mathcal{C}$. Since both operators are properly supported part (c) from Lemma \ref{fioprop} then implies, that $(D_{-}\circ \mathcal{G}(t_1))\in \FIO{-1/4}(\Sigma,M;\mathcal{C}';\Hom(\spinb^{-}(\Sigma_1),\spinb^{+}(M)))$ . The composition of this two properly supported operators is again properly supported, since differential operators are local operators and do not increase the support.\\
\\
(4) The construction of the solution operators as FIO has been done in such a way, that the canonical relation $\mathsf{C}(\inclus^\ast)\circ\mathcal{C}$ as composition is transversal and proper. We refer to our appendix and to chapter 5 of \cite{duistfio}. The Dirac operator $D_{-}$ does not affect the argument, since its canonical relation corresponds to the conormal bundle of the diagonal, as explained in (3).\\
\\
Lemma \ref{fioprop} (b) implies, that the composition \cref{QasFIO} is well defined and thus $Q$ a FIO of order $0$. It is itself properly supported as composition of properly supported Fourier integral operators. 
The intepretation of the canonical relation as canonical graph and the calculation of the principal symbol can be taken from \cite{BaerStroh} with our convention \cref{prinsymbdiff}.
\end{proof} 
\begin{rem}
The properly supportness of the Dirac-wave evolution operator implies, that it can be extended as map from local Sobolev sections to local Sobolev sections. As a topological isomorphism, its adjoint does as well.
\end{rem}

The same procedure can be applied with a Dirac-wave evolution operator $\tilde{Q}$, coming from the well-posedness of the homogeneous Cauchy problem in Theorem \ref{cauchyhomdneg}:
\begin{defi}\label{defevopneg}
For a globally hyperbolic manifold $M$ and $t_1,t_2 \in \timef(M)$ we define as (Dirac-) wave evolution operator (for negative chirality)
\begin{equation*}
\tilde{Q}(t_2,t_1)\,\,:\,\,H^s_\comp(\spinb^{-}(\Sigma_{1}))\,\,\rightarrow\,\,H^s_\comp(\spinb^{-}(\Sigma_{2}))\,\,,
\end{equation*}
defined as $\tilde{Q}(t_2,t_2)=\rest{t_2}\circ(\rest{t_1})^{-1}$.
\end{defi} 

\noindent The spectral entries $\tilde{Q}_{\pm\pm}(t_2,t_1)$ and $\tilde{Q}_{\pm\mp}(t_2,t_1)$ for $\tilde{Q}(t_2,t_1)$ are defined as in \cref{Qmatrix} and forthcoming. All properties are collected in the following result. The proofs and calculations just differs in notation.  
\begin{prop}
\noindent\label{colpropqneg1}
\begin{itemize}
\item[(a)] $\tilde{Q}(t_3,t_2)\circ \tilde{Q}(t_2,t_1)=\tilde{Q}(t_3,t_1)$,
\item[(b)] $\tilde{Q}(t,t)=\id{H^s_\comp(\spinb^{-}(M)\vert_{\Sigma_t})}$ and $\tilde{Q}(t_1,t_2)=\tilde{Q}^{-1}(t_2,t_1)$,
\item[(c)] $\tilde{Q}(t_2,t_1)$ is unitary for $s=0$, if $M$ is temporal compact.
\item[(d)] The operators $\tilde{Q}_{+-}(t_2,t_1)$ and $\tilde{Q}_{-+}(t_2,t_1)$ restrict to isomorphisms 
\begin{equation*}
\begin{array}{ccccc}
\tilde{Q}_{+-}(t_2,t_2)&:&\kernel{\tilde{Q}_{--}(t_2,t_1)}&\rightarrow& \kernel{(\tilde{Q}_{++}(t_2,t_1))^{\ast}} \\
\tilde{Q}_{-+}(t_2,t_2)&:&\kernel{\tilde{Q}_{++}(t_2,t_1)}&\rightarrow& \kernel{(\tilde{Q}_{--}(t_2,t_1))^{\ast}} 
\end{array}
\end{equation*}
\end{itemize}
\end{prop}

From the proof of the last assertion, one gets the important formulas
\begin{equation}\label{Qnegsys1}
\begin{array}{ccl}
(\tilde{Q}_{++}(t_2,t_1))^\ast \tilde{Q}_{++}(t_2,t_1)+(\tilde{Q}_{-+}(t_2,t_1))^\ast \tilde{Q}_{-+}(t_2,t_1)&=& \id{} \\
\text{and} && \\
(\tilde{Q}_{+-}(t_2,t_1))^\ast \tilde{Q}_{+-}(t_2,t_1)+(\tilde{Q}_{--}(t_2,t_1))^\ast \tilde{Q}_{--}(t_2,t_1)&=& \id{}\,\, . \\
\end{array}
\end{equation}
\newpage
\noindent A representation as Fourier integral opertor is possible and substantiated with the same proof:
\begin{prop}\label{Qnegfourier} 
For all $s \in \R$ one has $\tilde{Q}\in \FIO{0}(\Sigma_{2};\mathsf{C}'_{1\rightarrow 2};\Hom(\spinb^{-}(\Sigma_1),\spinb^{-}(\Sigma_2)))$ properly supported, with canonical relation \cref{Qcanrel} as in Proposition \ref{Qfourier} and principal symbol
\begin{equation*}
\bm{\sigma}_{0}(\tilde{Q})(x,\xi_{\pm};y,\eta)=\pm\frac{1}{2} (\norm{\eta}{\met_{t_1}}(y))^{-1}\left(\mp\norm{\xi_{\pm}}{\met_{t_2}}(x)\upbeta-\Cliff{t_2}{\xi^{\sharp}_{\pm}}\right)\circ \mathpzc{P}_{(x,\varsigma_{\pm})\leftarrow (y,\zeta_{\pm})} \circ \upbeta
\end{equation*}
with $(y,\zeta_{\pm})\in T^\ast_{\Sigma_{1}}M$ and $(x,\varsigma_{\pm})\in T^\ast_{\Sigma_{2}}M$ restricting to $(y,\eta)$ and $(x,\xi_{\pm})$ respectively.
\end{prop}



\addtocontents{toc}{\vspace{-3ex}}
\section{Galois coverings}\label{chap:galois}
From now on the spatial non-compactness of $M$ is specified in terms of equipping hypersurfaces with a special kind of bounded geometry\bnote{f28}: the non-compact Cauchy hypersurface $\Sigma$ is a Galois covering of a closed hypersuface $\widetilde{\Sigma}$ with respect to a Galois group, which we will denote by $\Upgamma$. This holds true for all other slices in the foliation, too. After introducing the differential geometry of coverings, some functional analytic notations and concepts with respect to this setting are fixed in order to introduce the concept of $\Upgamma$-Fredholmness. The content of this chapter heavily refers to \cite{shub} and in particular to \cite{atiyahellvn} and \cite{lee}, chapter 7 and 21.

\subsection{$\Upgamma$-manifolds}
\justifying
Consider a smooth and non-compact manifold $M$ and a vector bundle $E\rightarrow M$. Let $\Upgamma$ be a discrete group of automorphisms/deck transformations, such that the left action $(\upgamma,p)\,\mapsto\, \upgamma p$ on $M$ are smooth, freely and cocompact for all $\upgamma \in \Upgamma$. The unit element is denoted by $\upepsilon$.
Right actions can be converted to left actions by taking the inverse group element. This makes the left action a diffeomorphism on $M$. 
Since discrete groups of automorphisms are discrete Lie groups and act proper in addition, the quotient manifold theorem states, that $\quotspace{M}{\Upgamma}$ is a smooth manifold and the covering map 
\begin{equation}\label{canproj}
\mcyrl_\Upgamma\,:\, M \,\rightarrow\, \quotspace{M}{\Upgamma}
\end{equation}  
is a smooth submersion and \textit{Galois}, i.e. the cover is a principal bundle with respect to the deck transformation group. In this case, $\Upgamma$ acts transitively\bnote{f17}: for a pair of points $p,q \in M$ there exists a $\upgamma\in\Upgamma$, such that $q=\upgamma p$. If the orbit space with respect to $\Upgamma$ with above properties is again a smooth, but compact manifold, they are referred on as $\Upgamma$\textit{-manifolds}. If $M$ is equipped with a smooth and complex vector bundle $\pi\,:\,E\rightarrow\,M$, one can define an action $\pi_\Upgamma\vert_p\,:\,E_p\,\rightarrow\,E_{\upgamma p}$ for all $p \in M, \upgamma \in \Upgamma$, which is a linear isomorphism. If the projection $\pi$ is $\Upgamma$-\textit{equivariant} (i.e. $\pi(\upgamma p,v_{\upgamma p})=\upgamma \pi(p,v_p)$ for all $p,\upgamma$ and $v_p$ vector at $p$), then $\quotspace{E}{\Upgamma}$ is again a smooth vector bundle over $\quotspace{M}{\Upgamma}$. The linear isomorphism as action on $E$ than covers the action on $M$, for which reason they are called $\Upgamma$-\textit{vector bundle}. 
The group $\Upgamma$ comes with a notion of \textit{left translation operators}: 
\begin{eqnarray}
(L_{\upgamma_1} u)(\upgamma_2)&=&u(\upgamma_1^{-1}\upgamma_2)\label{leftrighttrans}\\
\text{with} \quad (L_\upgamma)^{-1}=L_{\upgamma^{-1}} & \text{and}& \quad L_{\upgamma_1\upgamma_2}=L_{\upgamma_1}L_{\upgamma_2} \,\, .\label{leftrighttransprop}
\end{eqnarray}
The corresponding action of $\Upgamma$ on a section $u$ of a $\Upgamma$-vector bundle is then described by $(L_\upgamma u)(p)=\pi_\Upgamma u(\upgamma^{-1}p)$. Any section, which is invariant under the left action, will be called $\Upgamma$-\textit{invariant}: $(L_\upgamma u)(x)=u(x)$. Two objects of this kind are a smooth positive measure $\mu$ on $M$ or a metric of $M$, inducing this smooth positive measure. Both can be obtained by lifting a smooth positive measure on the orbit space to $M$ by pullback with \cref{canproj} as local diffeomorphism. In a similar way one can define any bundle metric on a vector bundle $E\rightarrow M$ by lifting the Hermitian inner product of any fiber $\quotspace{E}{\Upgamma}$. For this $\Upgamma$ should act by isometries, such that $\pi_\Upgamma$ becomes an isometry on its own. 
Another way of defining a $\Upgamma$-vector bundle is by lifting the $\Upgamma$-action from $M$ to the tangent bundle, making it a real $\Upgamma$-vector bundle with induced action of the group. Any complexified tensor bundle over $M$ becomes a $\Upgamma$-vector bundle and a $\Upgamma$-invariant metric induces a Hermitian inner product on it. Fixing a $\Upgamma$-invariant smooth density (or metric), and if necessary a $\Upgamma$-invariant bundle metric, one can define square-integrable sections of $E$ with respect to this $\Upgamma$-invariant density. We denote these spaces by $L^2_\Upgamma(M,E)$. \\    
\\
The \textit{fundamental domain of} $\Upgamma$ is an open subset $\mathcal{F}\subset M$, which satisfies
\begin{align*}
\text{(1)} & \quad(\upgamma_1\mathcal{F})\cap(\upgamma_2\mathcal{F})=\emptyset\quad\forall\,\upgamma_1,\upgamma_2\in \Upgamma \quad\Rightarrow\quad \upgamma_1\neq\upgamma_2& \\
\text{(2)} & \quad M=\bigcup_{\upgamma \in \Upgamma}\upgamma \overline{\mathcal{F}} &\text{(3)} \quad \overline{\mathcal{F}}\setminus\mathcal{F}\quad\text{is a null set.}
\end{align*}
$\Upgamma$-invariant locally finite coverings and a $\Upgamma$-invariant partition of union subordinated to this covering are defined as follows: since $\quotspace{M}{\Upgamma}$ is compact, one can take a finite covering by open balls $\mathring{\mathbb{B}}_j$, $j \in J$ finite index set, and lift them to $M$. The covering then takes the form
\begin{equation*}
M=\bigcup_{\substack{j \in J \\ \upgamma \in \Upgamma}}\upgamma \mathring{\mathbb{B}}_j \quad .
\end{equation*}
A subordinated partition of unity $(\upgamma \mathring{\mathbb{B}}_j, \phi_{j,\upgamma})_{\substack{j \in J\\\upgamma \in \Upgamma}}$ can be constructed by taking compactly supported functions $\phi_{j,\upgamma}\in C^\infty_\comp(\upgamma \mathring{\mathbb{B}}_j,\R_{\geq 0})$, satisfying $\phi_{j,\upgamma}(p)=\phi_{j,\upepsilon}(\upgamma^{-1}p)=(L_\upgamma \phi_{j,\upepsilon})(p)$ as well as
\begin{equation}\label{gammainvpartition}
\sum_{\substack{j \in J \\\upgamma \in \Upgamma}} \phi_{j,\upgamma}=1 \quad .
\end{equation}
For $M$ being a globally hyperbolic manifold, consider a $\Upgamma$-invariant Lorentzian metric and bundle metrices of $\Upgamma$-vector bundles. As $M$ is diffeomorphic to $\timef(M)\times \Sigma$ with Cauchy hypersurface $\Sigma$ and time intervall $\timef(M) \subset \R$, one could consider $\Upgamma$-actions on either the whole manifold or on the temporal or the spatial part. 
We define as \textit{spatial} $\Upgamma$-\textit{manifold}  a $\Upgamma$-manifold, where the action is induced by a group $\Upgamma$, acting on the hypersurface as $\Upgamma$-manifold. In that case $\quotspace{\Sigma}{\Upgamma}$ is compact. The corresponding canonical projection is
\begin{equation}\label{canprojglobhyp}
\mcyrl_{\Upgamma,s}\,:\, M \,\rightarrow\, \timef(M)\times\quotspace{\Sigma}{\Upgamma} \quad.
\end{equation} 
The covering $\mcyrl_{\Upgamma,s}$ is of interest in this paper, where the Cauchy hypersurfaces $\Sigma$ are given, such that the orbit space with respect to $\Upgamma$ is compact without boundary. Corollary 12 in chapter 7 of \cite{ONeill} ensures, that for a group $\Upgamma$ as introduced above $\mcyrl_{\Upgamma,s}$ is indeed a covering between Lorentzian manifolds and that $\Upgamma$ is a deck transformation group, if $M$ is connected. This has been used to show, that $M$ is globally hyperbolic if and only if $\quotspace{M}{\Upgamma}$ is globally hyperbolic. The proof is given in \cite{harris} as Proposition 1.4, relating the proof of Lemma 4.1 in \cite{garhar}. The covering of interest has been taken into account, since an intermediate step in the references shows, that 
\begin{center}
$\Sigma$ is Cauchy hypersurface in $M$ $\quad \Leftrightarrow \quad$ $\quotspace{\Sigma}{\Upgamma}$ is Cauchy hypersurface in $\quotspace{M}{\Upgamma}$ .
\end{center}
Thus $\mcyrl_{\Upgamma,s}$ in \cref{canprojglobhyp} is indeed a Galois covering of globally hyperbolic manifolds by globally hyperbolic, compact quotients.

\subsection{von Neumann algebra and Fredholm theory according to a $\Upgamma$-action}
\justifying
On the space of $L^2$-functions over $\Upgamma$
\begin{equation*}
L^2(\Upgamma):=\SET{u\,:\,\Upgamma\,\rightarrow\,\C\,\bigg\vert\,\sum_{\upgamma\in\Upgamma}\absval{u(\upgamma)}^2 <\infty}
\end{equation*}
the translation operator \cref{leftrighttrans} is a unitary operator with respect to the inner product in the above space, which can be seen by applying \cref{leftrighttransprop}: $(L_\upgamma)^\ast =L_{\upgamma^{-1}}$.
The map $\upgamma\,\mapsto\,L_\upgamma$ 
becomes a unitary representation of the discrete group in this Hilbert space and spans a (group-) von Neumann algebra $\mathcal{N}_r(\Upgamma)$ in $L^2(\Upgamma)$, thus a $\ast$-subalgebra of the space of bounded operators acting between $L^2$-function $\mathscr{L}(L^2(\Upgamma))$. 
Since $L^2(\Upgamma)$ is a Hilbert space one can introduce an orthonormal basis, which is conventionally denoted by $\SET{\delta_\upgamma}_{\upgamma \in \Upgamma}$, satisfying $\delta_{\upgamma_1}(\upgamma_2)=1$ if $\upgamma_1=\upgamma_2$ and zero otherwise. Applying the left translation operator on this basis leads to $L_{\upgamma_1}\delta_{\upgamma_2}=\delta_{\upgamma_1\upgamma_2}$ (Kronecker delta). 
Each operator $A\in \mathscr{L}(L^2(\Upgamma))$ is then related to a matrix element: $A_{\upgamma_1,\upgamma_2}=\dscal{1}{L^2(\Upgamma)}{A \delta_{\upgamma_1}}{\delta_{\upgamma_2}}$ for $\upgamma_1,\upgamma_2 \in \Upgamma$. Following section 2.10 in \cite{shub} one can express the von Neumann algebra in terms of invariance of this matrix elements: 
\begin{equation*}
\mathcal{N}_l(\Upgamma) :=\SET{A\in \mathscr{L}(L^2(\Upgamma))\,\vert\,\forall\,a,b,\upgamma \in \Upgamma\quad A_{a\upgamma,b\upgamma}=A_{a,b}} \quad .
\end{equation*}
For an operator $A$ in $\mathcal{N}_l(\Upgamma)$ a \textit{trace} is defined as
\begin{equation*}
\uptau_\Upgamma(A):=\dscal{1}{L^2(\Upgamma)}{A \delta_{\upgamma}}{\delta_{\upgamma}}=\dscal{1}{L^2(\Upgamma)}{A \delta_{\upepsilon}}{\delta_{\upepsilon}} 
\end{equation*}
for $\upgamma,\upepsilon \in \Upgamma$; the definition is independent of the choice of $\upgamma$, which explains the second equivalence. Let $\mathcal{H}$ be another Hilbert space and consider the tensor product $L^2(\Upgamma)\otimes\mathcal{H}$. The orthonormal basis in $L^2(\Upgamma)$ induces an orthogonal decomposition 
\begin{equation*}
L^2(\Upgamma)\otimes\mathcal{H}=\bigotimes_{\upgamma \in \Upgamma}(\delta_\upgamma \otimes \mathcal{H}) \quad,
\end{equation*} 
which implies, that any operator $A \in \mathscr{L}(L^2(\Upgamma)\otimes\mathcal{H})$ has matrix entries $A_{\upgamma_1,\upgamma_2} \in \mathscr{L}(\mathcal{H})$ for $\upgamma_1,\upgamma_2\in \Upgamma$. Tensor products of von Neumann algebras are again von Neumann algebras: 
\begin{equation*}
\mathcal{N}_l(\Upgamma)\otimes \mathscr{L}(\mathcal{H}) :=\SET{A\in \mathscr{L}(L^2(\Upgamma)\otimes\mathcal{H})\,\vert\,\forall\,a,b,\upgamma \in \Upgamma\quad A_{a\upgamma,b\upgamma}=A_{a,b}} \quad.
\end{equation*}
The trace on this tensor product of algebras is the tensor product of the traces:
\begin{equation*}
(\uptau_\Upgamma \otimes \mathrm{Tr})(A)=\Tr{}{A_{\upepsilon,\upepsilon}} \quad
\end{equation*}
with $\Tr{}{A}:=\Tr{\mathscr{L}(\mathcal{H})}{A}=\sum_{j \in J}\dscal{1}{\mathcal{H}}{A e_j}{e_j}$, where $\SET{e_j}_{j \in J}$ is an orthonormal basis of $\mathcal{H}$. These tensor product occurs as an important concept for the analysis in this paper:
\begin{defi}
Let $\mathscr{H}$ and $\mathcal{H}$ be Hilbert spaces;
\begin{itemize}
\item[a)] $\mathscr{H}$ is a \textit{general Hilbert} $\Upgamma$-\textit{module}, if it carries a left unitary action of $\Upgamma$.
\item[b)] A general Hilbert $\Upgamma$-module $\mathscr{H}$ is called \textit{free}, if it is unitarily isomorphic to a $\Upgamma$-module $L^2(\Upgamma)\otimes \mathcal{H}$ and the representation of $\Upgamma$ is given by $\upgamma \,\mapsto\,L_\upgamma\otimes\id{\mathcal{H}}$ for $\upgamma \in \Upgamma$.
\item[c)] A general Hilbert $\Upgamma$-module $\mathscr{H}$ is defined to be \textit{projective}, if it is unitarily isomorphic to a closed $\Upgamma$-invariant subspace in $L^2(\Upgamma)\otimes \mathcal{H}$.
\end{itemize}
\end{defi}
The unitary isomorphisms are understood as unitary maps, commuting with the action of $\Upgamma$. The focus will rely on projective Hilbert $\Upgamma$-modules, which we will shortly refer on as Hilbert $\Upgamma$-modules, unless not otherwise indicated. Having such a Hilbert $\Upgamma$-module $\mathscr{H}$ one can define a von Neumann algebra, which consists of those operators in $\mathscr{L}(\mathscr{H})$, which commute with the left translation $L_\upgamma$ for all $\upgamma \in \Upgamma$ in $\mathscr{H}$:
\begin{equation}\label{vNcomm}
\mathscr{L}_\Upgamma(\mathscr{H}):=\SET{A \in \mathscr{L}(\mathscr{H})\,\vert\,AL_\upgamma=L_\upgamma A \quad \forall\, \upgamma \in \Upgamma}\quad.
\end{equation} 
The trace on this von Neumann algebra is $\Tr{\Upgamma}{\bullet}:=(\uptau_\Upgamma \otimes \mathrm{Tr})(\bullet)$. The orthogonal projection $P_{\mathscr{H}}$ commutes with the unitary representation in $L^2(\Upgamma)\otimes\mathcal{H}$, where $\mathscr{H}\subset L^2(\Upgamma)\otimes\mathcal{H}$ is a projective Hilbert $\Upgamma$-submodule. The important concept of the $\Upgamma$-\textit{dimension} is then well-defined by this trace:
\begin{equation*}
\dim_\Upgamma(\mathscr{H}):=\Tr{\Upgamma}{P_\mathscr{H}}\quad. 
\end{equation*}
It has the following properties:
\begin{lem}[section 2.17 in \cite{shub}]
\noindent\label{propgammadim}
\begin{itemize} 
\item[(1)] $\dim_\Upgamma(\mathscr{H})$ is independent of the inclusion $\mathscr{H}\subset L^2(\Upgamma)\otimes\mathcal{H}$;
\item[(2)] $\dim_\Upgamma(L^2(\Upgamma))=1$, $\dim_\Upgamma(\SET{0})=0$ and $\dim_\Upgamma(L^2(\Upgamma)\otimes\mathcal{H})=\dim_\C(\mathcal{H})$;
\item[(3)] $\dim_\Upgamma(\mathscr{H}_1\oplus\mathscr{H}_2)=\dim_\Upgamma(\mathscr{H}_1)+\dim_\Upgamma(\mathscr{H}_2)$ for two Hilbert $\Upgamma$-modules $\mathscr{H}_1,\mathscr{H}_2$;
\item[(4)] $\mathscr{H}_1\subset\mathscr{H}_2\,\Rightarrow\,\dim_\Upgamma(\mathscr{H}_1)\leq\dim_\Upgamma(\mathscr{H}_2)$ and $\dim_\Upgamma(\mathscr{H}_1)=\dim_\Upgamma(\mathscr{H}_2)\,\Leftrightarrow\,\mathscr{H}_1=\mathscr{H}_2$;
\item[(5)] $\dim_\Upgamma(\mathscr{H}_1)=\dim_\Upgamma(\mathscr{H}_2)\,\Leftrightarrow$ $\mathscr{H}_1$ and $\mathscr{H}_2$ are unitarily isomorphic to each other. 
\end{itemize}
\end{lem}
Linear operators between Hilbert $\Upgamma$-modules $\mathscr{H}_1$ and $\mathscr{H}_2$ can be introduced as well: 
\begin{defi}
Given a linear, in general unbounded operator $A\,:\,\mathscr{H}_1\supset \dom{}{A}\,\rightarrow\,\mathscr{H}_2$;
\begin{itemize}
\item[a)] $A$ is called $\Upgamma$-\textit{operator}, if $AL_\upgamma=L_\upgamma A$ for all $\upgamma \in \Upgamma$.
\item[b)] $A$ is called $\Upgamma$-\textit{morphism}, if $A$ is a $\Upgamma$-operator and $A \in \mathscr{L}(\mathscr{H}_1,\mathscr{H}_2)$.
\end{itemize}
\end{defi}
Part a) is equivalent by saying, that the domain $\dom{}{A}$ is $\Upgamma$-invariant. The second part in the definition contains operators in \cref{vNcomm} for $\mathscr{H}_1=\mathscr{H}_2=\mathscr{H}$, wherefore we will denote the space of $\Upgamma$-morphisms by $\mathscr{L}_\Upgamma(\mathscr{H}_1,\mathscr{H}_2)$. 
\\
\\
In order to define $\Upgamma$-Fredholmness for operators in $\mathscr{L}_\Upgamma(\mathscr{H}_1,\mathscr{H}_2)$, 
further characterisations of operators are needed, in order to have finite $\Upgamma$-traces for finite $\Upgamma$-dimensions.
\begin{defi}
Let $\mathscr{H}_1,\mathscr{H}_2$ be Hilbert $\Upgamma$-modules and $A \in \mathscr{L}_\Upgamma(\mathscr{H}_1,\mathscr{H}_2)$;
\begin{itemize}
\item[a)] $A$ is a $\Upgamma$-\textit{Hilbert-Schmidt operator}, denoted by $A \in \mathscr{HS}_\Upgamma(\mathscr{H}_1,\mathscr{H}_2)$, if $\Tr{\Upgamma}{A^\ast A} < \infty$.
\item[b)] $A$ is a $\Upgamma$-\textit{trace-class operator}, denoted by $A \in \mathscr{TC}_\Upgamma(\mathscr{H}_1,\mathscr{H}_2)$, if $\Tr{\Upgamma}{\absval{A}} < \infty$, where $\absval{A}=\sqrt{A^\ast A}$ is defined by the polar decomposition. 
\item[c)] $A$ is a $\Upgamma$-\textit{compact operator}, denoted by $A \in \mathscr{K}_\Upgamma(\mathscr{H}_1,\mathscr{H}_2)$, if $A$ lies in the norm closure of $\mathscr{TC}_\Upgamma(\mathscr{H}_1,\mathscr{H}_2)$ in $\mathscr{L}(\mathscr{H}_1,\mathscr{H}_2)$.
\item[d)] $A$ is a $\Upgamma$-\textit{Fredholm operator}, denoted by $A \in \mathscr{F}_\Upgamma(\mathscr{H}_1,\mathscr{H}_2)$, if there exists a $B\in \mathscr{L}_\Upgamma(\mathscr{H}_2,\mathscr{H}_1)$, such that $(\Iop{\mathscr{H}_1}-BA) \in \mathscr{TC}_\Upgamma(\mathscr{H}_1)$ and $(\Iop{\mathscr{H}_2}-AB) \in \mathscr{TC}_\Upgamma(\mathscr{H}_2)$; this $B$ is the $\Upgamma$-\textit{Fredholm parametrix} and the Atiyah-Bott formula defines the $\Upgamma$-\textit{index}:
\begin{equation*}
\Index_\Upgamma(A):=\Tr{\Upgamma}{\Iop{\mathscr{H}_1}-BA}-\Tr{\Upgamma}{\Iop{\mathscr{H}_2}-AB}\quad .
\end{equation*}
\end{itemize} 
\end{defi}
The trace class and Hilbert-Schmidt operators in the $\Upgamma$-setting carry the same properties as in the usual case. Moreover c) implies $\mathscr{TC}_\Upgamma \subset \mathscr{K}_\Upgamma$. 
As in ordinary Hilbert space calculus it is also possible to define invertible and unitary elements in $\mathscr{L}_\Upgamma(\mathscr{H}_1,\mathscr{H}_2)$ as in the sense of any (unital) $\mathrm{C}^\ast$-algebra. These are special $\Upgamma$-Fredholm operators with $\Upgamma$-index 0. From general Hilbert module theory it is also possible to think of $\Upgamma$-compact operators as closure of $\Upgamma$-\textit{finite rank} operators in $\mathscr{L}_\Upgamma$, where the latter one are defined as for Hilbert spaces via finitely many elements in the module. As for maps between Hilbert spaces the $\Upgamma$-finite rank, $\Upgamma$-compact, $\Upgamma$-Hilbert Schmidt and $\Upgamma$-trace-class operators are two-sided ideals in $\mathscr{L}_\Upgamma$; see \cite{blackadar} chapter I,II and III for several details. In order to show, that $\Upgamma$-morphism is $\Upgamma$-Fredholm, it is sufficient to find two $B_1, B_2\in \mathscr{L}_\Upgamma(\mathscr{H}_2,\mathscr{H}_1)$, so that $(\Iop{\mathscr{H}_1}-B_1A) \in \mathscr{TC}_\Upgamma(\mathscr{H}_1)$, $(\Iop{\mathscr{H}_2}-AB_2) \in \mathscr{TC}_\Upgamma(\mathscr{H}_2)$ and 
\begin{equation*}
\Index_\Upgamma(A):=\Tr{\Upgamma}{\Iop{\mathscr{H}_1}-B_1A}-\Tr{\Upgamma}{\Iop{\mathscr{H}_2}-AB_2}\quad .
\end{equation*}
The original definition follows from the fact, that these parametrices differ by a $\Upgamma$-trace-class operator: let $R_1=\Iop{}-B_1 A$ and $R_2=\Iop{}-AB_2$, then
\begin{equation*}
B_1-B_2=B_1 A B_2 + B_1 R_2 - B_2 = B_1 R_2 - R_1 B_2 \in \mathscr{TC}_\Upgamma(\mathscr{H}_1,\mathscr{H}_2) \quad.
\end{equation*}
Some properties of $\Upgamma$-Fredholm operators and their $\Upgamma$-indices are listed in the statement below, from which it becomes clear, that this change neither affects Fredholmness nor the index:
\begin{lem}[section 3.10 of \cite{shub}]
\noindent\label{propgammafred}
\begin{itemize}
\item[(1)] $\Index_\Upgamma(A)\,:\,\mathscr{F}_\Upgamma(\mathscr{H}_1,\mathscr{H}_2)\,\rightarrow\,\R$ is locally constant.
\item[(2)] If $A \in \mathscr{F}_\Upgamma(\mathscr{H}_1,\mathscr{H}_2)$, its $\Upgamma$-Fredholm parametrix $B$ is in $\mathscr{F}_\Upgamma(\mathscr{H}_2,\mathscr{H}_1)$ with \\$\Index_\Upgamma(B)=-\Index_\Upgamma(A)$.
\item[(3)] If $A \in \mathscr{F}_\Upgamma(\mathscr{H}_1,\mathscr{H}_2)$ and $B \in \mathscr{F}_\Upgamma(\mathscr{H}_2,\mathscr{H}_3)$, then $BA \in \mathscr{F}_\Upgamma(\mathscr{H}_1,\mathscr{H}_3)$ and \\$\Index_\Upgamma(BA)=\Index_\Upgamma(B)+\Index_\Upgamma(A)$.
\item[(4)] $A \in \mathscr{F}_\Upgamma(\mathscr{H}_1,\mathscr{H}_2)$ if and only if there exists a $B\in \mathscr{L}_\Upgamma(\mathscr{H}_2,\mathscr{H}_1)$, such that \\ $(\Iop{\mathscr{H}_1}-BA) \in \mathscr{K}_\Upgamma(\mathscr{H}_1)$ and $(\Iop{\mathscr{H}_2}-AB) \in \mathscr{K}_\Upgamma(\mathscr{H}_2)$.
\item[(5)] If $A \in \mathscr{F}_\Upgamma(\mathscr{H}_1,\mathscr{H}_2)$ and $C \in \mathscr{K}_\Upgamma(\mathscr{H}_1,\mathscr{H}_2)$, then $\Index_\Upgamma(A+C)=\Index_\Upgamma(A)$.
\item[(6)] If $A \in \mathscr{F}_\Upgamma(\mathscr{H}_1,\mathscr{H}_2)$, then $A^\ast \in \mathscr{F}_\Upgamma(\mathscr{H}_2,\mathscr{H}_1)$ with $\Index_\Upgamma(A^\ast)=-\Index_\Upgamma(A)$.
\item[(7)] If $A \in \mathscr{F}_\Upgamma(\mathscr{H}_1,\mathscr{H}_2)$, then $\dim_\Upgamma(\kernel{A}) <\infty$, $\dim_\Upgamma(\kernel{A^\ast}) <\infty$ and 
\begin{equation*}
\Index_\Upgamma(A)=\dim_\Upgamma(\kernel{A})-\dim_\Upgamma(\kernel{A^\ast})\quad.
\end{equation*}
\item[(8)] $A \in \mathscr{F}_\Upgamma(\mathscr{H}_1,\mathscr{H}_2)$ if and only if $\dim_\Upgamma(\kernel{A}) <\infty$ and there exists a closed set $W \in \mathscr{H}_2$, such that $W \subset \range{A}$ and $\codim_\Upgamma(W):=\dim_\Upgamma\left(\quotspace{\mathscr{H}_2}{W}\right)<\infty$.
\end{itemize} 
\end{lem}
Note, that in (7) the reversed implication is in general wrong; see \cite{shub} 3.10 Excercise 1. 
This subsection closes with two valuable lemmas for the further progress:
\begin{lem}\label{helpinglemma}
Let $\mathscr{H},\mathscr{H}_1,\mathscr{H}_2$ be Hilbert $\Upgamma$-modules and $A \in \mathscr{L}_\Upgamma(\mathscr{H}_1,\mathscr{H}_2)$, then
\begin{itemize}
\item[(1)]
$\kernel{A}$ is a projective Hilbert $\Upgamma$-submodule of $\mathscr{H}_1$; 
\item[(2)]
If $\range{A}$ is closed, then it is a projective Hilbert $\Upgamma$-submodule of $\mathscr{H}_2$; 
\item[(3)] $\mathscr{H}_1\oplus \mathscr{H}_2$ and $\left(\quotspace{\mathscr{H}}{W}\right)$ for any closed, $\Upgamma$-invariant subspace $W\subset \mathscr{H}$ are projective Hilbert $\Upgamma$-(sub)modules.
\end{itemize}
\end{lem}
\begin{proof}
Any closed subspace of a Hilbert space, the direct sum of two Hilbert spaces and the quotient of a Hilbert space with one of its closed subspaces are themselve Hilbert spaces. One checks, whether $(L_\Upgamma)(X)\subseteq X$, where $L_\Upgamma$ is the left unitary action of a (general) Hilbert $\Upgamma$-module (without specifying the concrete element) and $X$ a Hilbert (sub-) space.
Since $A$ is a $\Upgamma$-morphism, it commutes with the unitary action of $\Upgamma$. 
\begin{itemize}
\item[(1)]
Suppose $u \in \kernel{A}$, then the commuting of $A$ and $L_\Upgamma$ shows, that also $L_\Upgamma u \in \kernel{A}$:
\begin{equation*}
Au=0 \quad \Rightarrow \quad AL_\Upgamma u = L_\Upgamma A u =0 \quad \Rightarrow \quad L_\Upgamma(\kernel{A})\subseteq\kernel{A}\quad.
\end{equation*}
\item[(2)]
If $v \in \range{A}$, then there exists a $u \in \mathscr{H}_1$, such that $v=Au$. Applying $L_\Upgamma$ from the left and using the commutation property of $A$ gives
\begin{equation*}
L_\Upgamma v = L_\Upgamma A u = A L_\Upgamma u \quad \Rightarrow \quad L_\Upgamma(\range{A})\subseteq \SET{v \in \mathscr{H}_2\,\vert\, \exists\,w\in \range{L_\Upgamma}\,:\,v=Aw}\quad.
\end{equation*} 
But since $\mathscr{H}_1$ is already a Hilbert $\Upgamma$-module, the unitary representation of $\Upgamma$ has range (in a dense subset of) $\mathscr{H_1}$, such that
\begin{equation*}
L_\Upgamma(\range{A})\subseteq \SET{v \in \mathscr{H}_2\,\vert\, \exists\,u\in \mathscr{H}_1\,:\,v=Au}=\range{A}\quad.
\end{equation*} 
\item[(3)] Since $\mathscr{H}_1$ and $\mathscr{H}_2$ are both Hilbert $\Upgamma$-modules, the ranges of $L_\Upgamma$ are (dense subsets in) $\mathscr{H}_1$ or $\mathscr{H}_2$. The diagonal action of $\Upgamma$ on the direct sum implies
\begin{equation*}
L_{\Upgamma\times\Upgamma}(\mathscr{H}_1\oplus\mathscr{H}_2)=L_\Upgamma(\mathscr{H}_1)\oplus L_\Upgamma(\mathscr{H}_2)\subseteq \mathscr{H}_1\oplus\mathscr{H}_2 \quad.
\end{equation*}
The quotient Hilbert space consists of equivalence classes for each element in $\mathscr{H}$, where two Hilbert vectors are equivalent, if the difference is an element in $W$. First check, that this equivalence relation is also true for transformed Hilbert vectors: as $\mathscr{H}$ is a Hilbert $\Upgamma$-module the elements $L_\Upgamma v_1, L_\Upgamma v_2$ are in $L_\Upgamma(\mathscr{H})\subseteq \mathscr{H}$ for $v_1,v_2 \in \mathscr{H}$ and $L_\Upgamma(v_1-v_2)\in L_\Upgamma(W)\subseteq W$, since $W$ is a $\Upgamma$-submodule. The underlying vector space structure yields a linear group action, such that
\begin{equation*}
L_\Upgamma (v_1-v_2)=\left(L_\Upgamma v_1-L_\Upgamma v_2\right) \in W\quad. 
\end{equation*} 
Thus each equivalence class is $\Upgamma$-invariant: $L_\Upgamma\left(\quotspace{\mathscr{H}}{W}\right)\subseteq \left(\quotspace{\mathscr{H}}{W}\right)$. 
\end{itemize}
Hence all Hilbert subspaces from (1) to (3) are $\Upgamma$-invariant and thus carry the structure of a Hilbert $\Upgamma$-module. It is left to check, that they are projective: let $\mathcal{H},\mathcal{H}_1$ and $\mathcal{H}_2$ be Hilbert spaces. If $\mathscr{H},\mathscr{H}_1$ and $\mathscr{H}_2$ are projective, then they are unitarily isomorphic to a closed $\Upgamma$-submodule of $L^2(\Upgamma)\otimes\mathcal{H}, L^2(\Upgamma)\otimes\mathcal{H}_1$ and $L^2(\Upgamma)\otimes\mathcal{H}_2$ respectively. In (3) the direct sum of the two unitary isomorphisms implies a unitary isomorphism on the direct sum of the closed $\Upgamma$-submodules, which is again a closed $\Upgamma$-submodule of 
\begin{equation*}
(L^2(\Upgamma)\otimes\mathcal{H}_1)\oplus(L^2(\Upgamma)\otimes\mathcal{H}_2)\cong L^2(\Upgamma)\otimes(\mathcal{H}_1\oplus\mathcal{H}_2)\quad.
\end{equation*}
Restricting the unitary isomorphisms on the kernel and on the closed range leads to closed subspaces in $L^2(\Upgamma)\otimes\mathcal{H}_1$ and $L^2(\Upgamma)\otimes\mathcal{H}_2$ respectively. 
Since these isomorphisms commute with the action of $\Upgamma$, they are again $\Upgamma$-submodules, hence the kernel and a closed range are indeed projective Hilbert $\Upgamma$-submodules.\\
\\
In order to show, that the quotient is projective, consider the orthogonal complement $W^\perp$ of a closed Hilbert $\Upgamma$-submodule. Since $W^\perp \subset \mathscr{H}$ is always closed, it is itself a Hilbert space. Furthermore, it is $\Upgamma$-invariant: suppose $v \in W^\perp$, i.e. $v\in \mathscr{H}$, such that $\dscal{1}{\mathscr{H}}{v}{u}=0$ for all $u \in W$. Then $L_\Upgamma v \in W^\perp$, since the action of $\Upgamma$ is unitary:
\begin{equation*}
\dscal{1}{\mathscr{H}}{L_\Upgamma v}{L_{\upepsilon}u}=\dscal{1}{\mathscr{H}}{L_\Upgamma v}{L_\Upgamma (L_\Upgamma)^{-1}u}=\dscal{1}{\mathscr{H}}{v}{(L_\Upgamma)^{-1}u}=0\quad \forall\,u\in W\quad.
\end{equation*}
One concludes, that $L_\Upgamma(W^\perp)\subseteq W^\perp$, implying $W^\perp$ to be a general Hilbert $\Upgamma$-submodule. Projectivity follows from restricting the unitary isomorphism on $W^\perp$, which again maps to a closed subspace in $L^2(\Upgamma)\otimes \mathcal{H}$. The commuting of the isomorphisms with the group action implies $\Upgamma$-invariance of this closed subspace, such that $W^\perp$ becomes a projective Hilbert $\Upgamma$-submodule. Since $W^\perp \cong \left(\quotspace{\mathscr{H}}{W}\right)$, this isomorphism induces a unitary isomorphism between the orthogonal complement and the quotient space; see Proposition 2.16 in \cite{shub}. The composition of this and the former unitary isomorphism gives again a unitary isomorphism from the quotient space to a closed subspace in $L^2(\Upgamma)\otimes \mathcal{H}$, implying $\left(\quotspace{\mathscr{H}}{W}\right)$ to be a projective Hilbert $\Upgamma$-submodule. 
\end{proof}

With this a $\Upgamma$-version of Lemma A.1 in \cite{BaerBall} can be proven, which was used in the proof of the Fredholmness in \cite{BaerStroh} for compact Cauchy boundary. 
\begin{lem}\label{funcanagammalemma}
Let $\mathscr{H},\mathscr{H}_1,\mathscr{H}_2$ be Hilbert $\Upgamma$-modules and $A \in \mathscr{L}_\Upgamma(\mathscr{H},\mathscr{H}_1)$ and $B \in \mathscr{L}_\Upgamma(\mathscr{H},\mathscr{H}_1)$, which is onto; define $C=A\vert_{\kernel{B}}\oplus \Iop{\mathscr{H}_2}$, then
\begin{itemize}
\item[(1)] $\dim_\Upgamma(\kernel{C})=\dim_\Upgamma(\kernel{A\oplus B})\,\,;$
\item[(2)] $\range{C}$ is closed if and only if $\range{A\oplus B}$ is closed and 
\begin{equation*}
\codim_{\Upgamma}(\range{C})=\codim_{\Upgamma}(\range{A\oplus B})\quad;
\end{equation*} 
\item[(3)] $C$ is $\Upgamma$-Fredholm if and only if $A\oplus B$ is $\Upgamma$-Fredholm and $\Index_{\Upgamma}(C)=\Index_{\Upgamma}(A\oplus B)$.
\end{itemize}
\end{lem}
\begin{proof}
Since $A$ and $B$ are bounded, also $A\vert_{\kernel{B}}$, $C$ and $A\oplus B$ are bounded 
and the $\Upgamma$-invariance follows trivially in each summand: $C, (A \oplus B) \in \mathscr{L}_\Upgamma(\mathscr{H}^{\oplus 2},\mathscr{H}_1\oplus \mathscr{H}_2)$. 
\begin{itemize}
\item[(1)] From the proof in \cite{BaerBall} one concludes in the same way, that there exist two isomorphisms $\mathcal{I}\in \mathscr{L}(\mathscr{H}^{\oplus 2})$ and $\mathcal{J}\in \mathscr{L}(\mathscr{H}_1\oplus \mathscr{H}_2)$,
\begin{equation*}
\mathcal{I}=\Iop{\kernel{B}}\oplus \left(B\vert_{\kernel{B}^\perp}\right)^{-1} \quad \text{and} \quad \mathcal{J}=\left(
\begin{matrix}
\Iop{\mathscr{H}_1} & -A\vert_{\kernel{B}^\perp}\circ \left(B\vert_{\kernel{B}^\perp}\right)^{-1} \\
\zerop{} & \Iop{\mathscr{H}_2}
\end{matrix}
\right) \quad,
\end{equation*}
such that
\begin{equation}\label{operatorisomorph}
C=\mathcal{J}\circ (A\oplus B)\circ \mathcal{I} \quad .
\end{equation}
The surjectivity of $B$ has been used to guarantee, that $B$ is bijective on $\kernel{B}^{\perp}$. Thus the null spaces of both $\Upgamma$-invariant operators are the same and since both subspaces are projective $\Upgamma$-submodules according to the previous lemma, they have the same $\Upgamma$-dimension according to (4) in Lemma \ref{propgammadim}.
\begin{equation*}
\dim_\Upgamma(\kernel{C})=\dim_\Upgamma(\kernel{\mathcal{J}\circ (A\oplus B)\circ \mathcal{I}}) \quad .
\end{equation*}
Observe, that $\mathcal{J}$ and $\mathcal{I}^{-1}$ restrict to isomorphisms between the kernels as projective $\Upgamma$-submodules.
Thus the composition is a Hilbert space isomorphism:
\begin{equation*}
\mathcal{J}\circ\mathcal{I}^{-1}\,\,:\,\, \kernel{A\oplus B}\,\,\leftrightarrow\,\,\kernel{\mathcal{J}\circ(A\oplus B)\circ \mathcal{I}} \quad.
\end{equation*}
According to Proposition of 2.16 in \cite{shub} every linear topological isomorphism of projective $\Upgamma$-submodules implies the existence of a unitary isomorphism between the same $\Upgamma$-modules, 
which proves the claim.
\item[(2)] Since closedness does not involve any property with respect to the group $\Upgamma$, this claim follows as in Lemma A.1 in \cite{BaerBall} from the closed range theorem.
\\
\\
The previous lemma and closedness implies, that 
\begin{equation*}
\quotspace{\left(\mathscr{H}_1\oplus\mathscr{H}_2\right)}{\range{C}}\quad \text{and} \quad \quotspace{\left(\mathscr{H}_1\oplus\mathscr{H}_2\right)}{\range{A\oplus B}}
\end{equation*}  
are projective Hilbert $\Upgamma$-submodules. As in (1) a unitary isomorphism between these two spaces needs to be found. Since both quotient spaces are again Hilbert spaces, they are isomorphic to their dual spaces by the Frechet-Riesz theorem. The dual of a quotient of a Hilbert space and a closed subset is isomorphic to the orthogonal complement of the closed subset:
\begin{equation*}
\quotspace{\left(\mathscr{H}_1\oplus\mathscr{H}_2\right)}{\range{C}} \cong \left(\quotspace{\left(\mathscr{H}_1\oplus\mathscr{H}_2\right)}{\range{C}}\right)' \cong \range{C}^{\perp}
\end{equation*}  
and analogous for the other quotient. The closed range theorem implies, that the orthogonal complement of the ranges are equal to the null space of their adjoint operators:
\begin{equation*}
\quotspace{\left(\mathscr{H}_1\oplus\mathscr{H}_2\right)}{\range{C}} \cong \kernel{C^\ast}\quad \text{and}\quad \quotspace{\left(\mathscr{H}_1\oplus\mathscr{H}_2\right)}{\range{A\oplus B}}\cong \kernel{(A\oplus B)^\ast}\quad.
\end{equation*}  
It is left to show, that the kernels of the adjoint operators are isomorphic to each other: by the closed range theorem the image equals the inverse annihilator of the kernel of the adjoint operator. In the Hilbert space setting this annihilator coincides the orthogonal complement and because it is closed, its orthogonal complement is again the null space:
\begin{equation*}
\range{C}^\perp=\kernel{C^\ast}\quad\text{and}\quad \range{A\oplus B}^\perp=\kernel{(A\oplus B)^\ast} \quad. 
\end{equation*} 
Adjoining \cref{operatorisomorph} one gets
\begin{equation*}
C^\ast=\mathcal{I}^\ast\circ (A\oplus B)^\ast \circ \mathcal{J}^\ast\quad,
\end{equation*}  
where $I^\ast \in \mathscr{L}(\mathscr{H}^{\oplus 2})$ and $J^\ast \in \mathscr{L}(\mathscr{H}_1\oplus\mathscr{H}_2)$ are again isomorphisms with inverses, given by the adjoints of the inverses of $\mathcal{I}$ and $\mathcal{J}$. Following a similar procedure as in (1) one can show, that both kernels are indeed isomorphic to each other. 
Thus the composition with the other isomorphisms, used to reduce the quotients, implies again a topological isomorphism between projective $\Upgamma$-submodules, where the same referred result as above finally leads to the claim:
\begin{align*}
\codim_\Upgamma(\range{C})=\dim_\Upgamma(\cokernel{C})=\dim_\Upgamma\left(\quotspace{\mathscr{H}_1\oplus \mathscr{H}_2}{\range{C}}\right)&=\dim_\Upgamma(\kernel{C^\ast} \\
& \quad\quad\quad \| \\
\codim_\Upgamma(\range{A\oplus B})=\dim_\Upgamma(\cokernel{A\oplus B})&=\dim_\Upgamma(\kernel{(A\oplus B)^\ast}
\end{align*}
\item[(3)] $\Upgamma$-Fredholmness of $C$ means $\dim_\Upgamma(\kernel{C})<\infty$, $\range{C}$ is closed and $\codim_\Upgamma(\range{C})<\infty$. The previous claims imply $\dim_\Upgamma(\kernel{A\oplus B})<\infty$, $\range{A\oplus B}$ is closed and $\codim_\Upgamma(\range{A\oplus B})<\infty$ and thus $\Upgamma$-Fredholmness of $A\oplus B$. The equivalence of the $\Upgamma$-indices follows from the equivalence of the $\Upgamma$-dimensions and -codimensions and (7) of Lemma \ref{propgammafred}.
\end{itemize}
\end{proof}
Another consequence of Lemma \ref{helpinglemma} (3) is
\begin{cor}\label{helpingcorollary}
Given Hilbert $\Upgamma$-modules $\mathscr{H}_j$ and $\mathscr{H}_j^{'}$, $j \in \SET{1,2}$, and two $\Upgamma$-morphisms \\ $A_j:\,\mathscr{H}_j\rightarrow \mathscr{H}_j^{'}$, then
\begin{equation*}
(A_1\oplus A_2)\in \mathscr{F}_\Upgamma(\mathscr{H}_1\oplus\mathscr{H}_2,\mathscr{H}_1^{'}\oplus\mathscr{H}_2^{'})\quad\text{if and only if}\quad A_j\in \mathscr{F}_\Upgamma(\mathscr{H}_j,\mathscr{H}_j^{'})
\end{equation*} 
and the index is
\begin{equation*}
\Index_\Upgamma(A_1\oplus A_2)=\Index_\Upgamma(A_1)+\Index_\Upgamma(A_2)\quad .
\end{equation*}
\end{cor}

\subsection{Operators and Sobolev spaces on $\Upgamma$-manifolds}
\justifying
Let $E\rightarrow M$ be a $\Upgamma$-vector bundle over a $\Upgamma$-manifold $M$, equipped with a $\Upgamma$-invariant smooth density $\differ \mu$, which is either defined by a $\Upgamma$-invariant Lorentzian or Riemannian metric or otherwise, and with a $\Upgamma$-invariant inner product $\idscal{1}{E_{p}}{\cdot}{\cdot}$ on each fiber $E_p$, i.e. 
\begin{equation*}
\differ \mu (\upgamma p)=\differ \mu(p)\quad \text{and} \quad \idscal{1}{E_{\upgamma p}}{\cdot}{\cdot}=\idscal{1}{E_{p}}{\cdot}{\cdot}\quad \forall\,\upgamma\in\Upgamma\quad.
\end{equation*} 
These properties ensure, that 
$\Upgamma$ has a unitary representation in $L^2_\Upgamma(M,E)$. Choosing a fundamental domain $\mathcal{F}\opset M$ allows to decompose $M$ isomorphically into $\Upgamma\times \mathcal{F}$, such that the action of the group on the $\Upgamma$-manifold becomes an action on $\Upgamma\times \mathcal{F}$ : $\upgamma_2 (\upgamma_1,p)=\upgamma_2 (\upgamma_2\upgamma_1,p)$ for each $p \in \mathcal{F}$ and $\upgamma_1,\upgamma_2 \in \Upgamma$, where $(\upgamma_1,p)\mapsto\upgamma_1 p$ is an isomorphism. This induces a unitary isomorphism between $L^2$-sections:
\begin{equation*}
L^2_\Upgamma(M,E) \quad\stackrel{\sim}{\longleftrightarrow}\quad L^2(\Upgamma)\otimes L^2(\mathcal{F},E) \quad,
\end{equation*}
where the $\Upgamma$-vector bundle is restricted to the fundamental domain. If both $L^2$-spaces are identical and the action of $\Upgamma$ is given by $L_\upgamma \otimes \id{\mathcal{F}}$, one observes, that $L^2_\Upgamma(M,E)$ and any closed $\Upgamma$-invariant subset \textit{are Hilbert} $\Upgamma$-\textit{modules}! 
The space $\mathscr{L}_\Upgamma(L^2_\Upgamma(M,E))$ of operators in $\mathscr{L}(L^2_\Upgamma(M,E))$, which commute with the left translation operator, is a von Neumann algebra with trace $\mathrm{Tr}_\Upgamma:=\uptau_\Upgamma\otimes\mathrm{Tr}$, which is independent of the choice of the fundamental domain, see Theorem in section 2.14 of \cite{shub}. The $\Upgamma$-trace can be expressed in terms of Schwartz kernels, see subsections 2.21 and 2.22 in \cite{shub}. To identify an operator to be $\Upgamma$-Hilbert-Schmidt or $\Upgamma$-trace-class in terms of their Schwartz kernels the reader should refer to subsection 2.23 in \cite{shub} and paragraph 4 of \cite{atiyahellvn}. Speaking about Schwartz kernels, for an operator $A \in \mathscr{L}_\Upgamma(L^2_\Upgamma(M,E))$ the $\Upgamma$-invariance of its kernel function up isometries is implied as follows:  

\begin{lem}[cf. Lemma in 2.24 in \cite{shub} for the scalar case]\label{lemgammakernel}
Let $A \in \mathscr{L}_\Upgamma(L^2_\Upgamma(M,E))$ with Schwartz kernel $K_A$, then
\begin{equation*}
K_A(\upgamma p,\upgamma q)=\pi_\Upgamma(p)K_A(p,q)\pi_\Upgamma^{-1}(q)
\end{equation*}
for all $\upgamma \in \Upgamma$ and $(p,q)\in M \times M$.
\end{lem}

This implies, that any element in the von Neumann algebra $\mathscr{L}_\Upgamma(L^2_\Upgamma(M,E))$ has a Schwartz kernel, which is a distribution on the compact orbit space $\quotspace{\left(M \times M\right)}{\Upgamma}$ under the diagonal action of $\Upgamma$. Moreover
\begin{cor} Under the same preassumptions as in Lemma \ref{lemgammakernel} one has
\begin{equation*}
\tr{E_{\upgamma p}}{K_A(\upgamma p,\upgamma p)}=\tr{E_p}{K_A(p,p)}\quad.
\end{equation*}
\end{cor}

The operators of interest are differential operators and their extensions to pseudo-differential and Fourier integral operators. Since these operators are a priori unbounded, we refer to them as $\Upgamma$-operators, if they commute with the left translation. We denote by
\begin{equation*}
\begin{split}
\Diff{m}{\Upgamma}(M,\Hom(E,F))&:=\SET{A \in \Diff{m}{}(M,\Hom(E,F))\,\vert\, AL_\upgamma=L_\upgamma A \quad \forall\, \upgamma \in \Upgamma}\quad; \\
\ydo{s}{\Upgamma}(M,\Hom(E,F))&:=\SET{A \in \ydo{s}{}(M,\Hom(E,F))\,\vert\, AL_\upgamma=L_\upgamma A \quad \forall\, \upgamma \in \Upgamma}\quad; \\
\FIO{s}_{\Upgamma}(M,N;\mathsf{C}';\Hom(E,G))&:=\SET{A \in \FIO{s}(M,N;\mathsf{\Lambda};\Hom(E,G))\,\vert\, AL_\upgamma=L_\upgamma A \quad \forall\, \upgamma \in \Upgamma} 
\end{split}
\end{equation*} 
the sets of $\Upgamma$-differential, $\Upgamma$-pseudo-differential and $\Upgamma$-Fourier integral operators, where $M,N$ are $\Upgamma$-manifolds, $E,F\rightarrow M$ and $G\rightarrow N$ are $\Upgamma$-vector bundles, $\mathsf{\Lambda}$ a Lagrangian submanifold and $m\in \N_0$, $s\in\R$. If an operator in one of these operator spaces enjoys the property of being properly supported, its Schwartz kernel is compactly supported in $\quotspace{\left(M\times M\right)}{\Upgamma}$ and $\quotspace{\left(M\times N\right)}{\Upgamma}$ respectively. Thus, any $\Upgamma$-invariant differential operator has compactly supported Schwartz kernel on the orbit space. The function spaces of interest next to $L^2_\Upgamma(M,E)$ are Sobolev spaces. From \cite{shub} section 3.9 they are defined by an appropriate norm with the $\Upgamma$-invariant partition of unity \cref{gammainvpartition}:
\begin{equation*}
\norm{u}{H^s_\Upgamma(M,E)}^2:=\sum_{\substack{j \in J\\ \upgamma \in \Upgamma}}\norm{\phi_{j,\upgamma}u}{H^s(\supp{\phi_{j,\upgamma}},E)}^2\quad.
\end{equation*} 
The corresponding $\Upgamma$-Sobolev spaces are then defined for any $s\in \R$ via
\begin{equation*}
H^s_\Upgamma(M,E):=\SET{u\in H^s_\loc(M,E)\,\vert\,\norm{u}{H^s_\Upgamma(M,E)} < \infty} \quad.
\end{equation*} 
A certain subclass of $\Upgamma$-invariant pseudo-differential operators are \textit{classical} operators: 
\begin{equation*}
\begin{split}
A \in \ydo{m}{\mathsf{cl},\Upgamma}(M,\Hom(E,F))\quad :\Leftrightarrow\quad &A\in \ydo{m}{\Upgamma}(M,\Hom(E,F))\,:\, A=\hat{A}+R\,,\,\text{where}\\ 
&\hat{A}\in\ydo{m}{\Upgamma}(M,\Hom(E,F))\,\,\text{has}\,\,\text{classical symbol}\\
&\text{and}\,R\in\ydo{-\infty}{\Upgamma}(M,\Hom(E,F))
\end{split}
\end{equation*}
for $M$ a $\Upgamma$-manifold and $E,F$ $\Upgamma$-vector bundles over $M$. 
Theorem 1 of subsection 3.9 in \cite{shub} ensures, that any properly supported element in $\ydo{m}{\Upgamma}(M,\Hom(E,F))$ extends to an operator in $\mathscr{L}(H^s_\Upgamma(M,E),H^{s-m}_\Upgamma(M,F))$, which commutes with the action in the Sobolev spaces. Together with uniform elliptic regularity (Theorem 2 in subsection 3.9 of \cite{shub}) they imply, that $H^s_\Upgamma(M,E)$ are \textit{Hilbert} $\Upgamma$-\textit{modules for any} $s\in \R$. As a consequnce any properly supported operator in $\ydo{m}{\Upgamma}(M,\Hom(E,F))$ is part of the von Neumann algebra $\mathscr{L}_\Upgamma(H^s_\Upgamma(M,E),H^{s-m}_\Upgamma(M,F))$ (see Corollary 2 of these two results).
\begin{rems*}
\noindent 
\begin{itemize}
\item[(a)] Suppose $A \in \ydo{m}{\mathsf{cl},\Upgamma}(M,E)$ is elliptic, properly supported with $m >0$; if $A=A^\ast$ in $L^2_\Upgamma(M,E)$, then it is essentially self-adjoint with $\dom{}{\overline{A}}=H^m_\Upgamma(M,E)$, where $\overline{A}$ is the closure/self-adjoint extension of $A$.
\item[(b)] According to Corollary 2 in section 3.11 of \cite{shub} the $\Upgamma$-Sobolev spaces are in particular \textit{free Hilbert} $\Upgamma$-\textit{modules}.
\end{itemize}
\end{rems*}

\noindent The requirement of being properly supported in order to have all nice properties, is always fullfilled for differential operators, but too restrictive for a general class of operators. This motivates to extend the class of properly supported $\Upgamma$-operators to a wider class in such a way, that smoothing operators have smooth kernels and act appropiately between Sobolev spaces, such that all results for properly supported $\Upgamma$-operators can be recovered, without insisting on this feature.
\begin{defi}\label{ssmoothing}
Given $E_1\rightarrow M_1$ and $E_2\rightarrow M_2$ $\Upgamma$-vector bundles over $\Upgamma$-manifolds $M_1,M_2$ and let $A\,:\,C^\infty_\comp(M_1,E_1)\,\rightarrow\,C^{-\infty}(M_2,E_2)$ be a $\Upgamma$-operator; $A$ is said to be \textit{s-smoothing}, if it extends to a continuous linear operator between $\Upgamma$-Sobolev spaces for any orders $r,p$: 
\begin{equation*}
A\,:\,H^r_\Upgamma(M_1,E_1)\,\rightarrow\,H^p_\Upgamma(M_2,E_2)\quad.
\end{equation*} 
\end{defi}

\noindent This is a slight generalization of s-smoothing operators, introduced in \cite{shub} subsection 3.11; if the operator is known to be a pseudo-differential operator, we write $S\ydo{-\infty}{\Upgamma}(M,\Hom(E_1,E_2))$ (for $M_1=M=M_2)$. Note, that in this case $S\ydo{-\infty}{\Upgamma}(M,\Hom(E_1,E_2)) \subset \ydo{-\infty}{\Upgamma}(M,\Hom(E_1,E_2))$. Especially for $\Upgamma$-invariant pseudo-differential operators one can define a class of operators, which differ from a properly supported $\Upgamma$-invariant pseudo-differential operator by a s-smoothing pseudo-differential operator\bnote{f29}:
\begin{defi}[Definition 2 in \cite{shub} 3.11]
Let $E,F$ be $\Upgamma$-vector bundles over the $\Upgamma$-manifold $M$; an operator $A \in \ydo{m}{\Upgamma}(M,\Hom(E,F))$ is called \textit{s-regular}, if there is a properly supported $\hat{A}\in \ydo{m}{\Upgamma}(M,\Hom(E,F))$, such that $(A-\hat{A}) \in S\ydo{-\infty}{\Upgamma}$. If $\hat{A}\in \ydo{m}{\mathsf{cl},\Upgamma}(M,\Hom(E,F))$, then they are defined as \textit{classical s-regular}.
\end{defi}

\noindent As in \cite{shub} we denote both operator spaces with $S\ydo{m}{\Upgamma}(M,\Hom(E,F))$ and respectively \\ $S\ydo{m}{\mathsf{cl},\Upgamma}(M,\Hom(E,F))$. Properties of these classes of operators are listed below:
\begin{rems} 
\noindent \label{remarkssmoothing}
\begin{itemize}
\item[(1)] If $A$ is properly supported $\Upgamma$-operator with smooth Schwartz kernel, then $A$ is s-smoothing.
\item[(2)] If $A$ is s-smoothing, then it has a smooth Schwartz kernel and is a $\Upgamma$-trace-class operator.
\item[(3)] If $A \in S\ydo{m}{\Upgamma}(M,\Hom(E,F))$, then $A\in \mathscr{L}_\Upgamma(H^s_\Upgamma(M,E),H^{s-m}_\Upgamma(M,F))$ for any $s\in \R$.
\item[(4)] If $A \in S\ydo{m}{\Upgamma}(M,\Hom(E,F))$ elliptic with $m>0$ and formally self-adjoint, then $A$ is essentially self-adjoint and the spectral projections $\chi_I(A)\in S\ydo{-\infty}{\Upgamma}(M,\Hom(E,F))$ for any bounded Borel set $I \subset \R$; see also Proposition 3.1 in \cite{atiyahellvn}.
\item[(5)] If $A \in S\ydo{m}{\Upgamma}(M,\Hom(E,F))$ is elliptic with $m>0$ and $P_{\kernel{A}}=\chi_{\SET{0}}$ an orthogonal projection, then $P_{\kernel{A}}\in S\ydo{-\infty}{\Upgamma}(M,\End(E))$.
\end{itemize}
\end{rems}
A generalization of property (4) to unbounded intervalls in $\R$ for $A$ being an elliptic $\Upgamma$-invariant differential operator is needed. Denote by
\begin{equation}\label{basicspecproj}
P_{+}:=P_{\intervallo{0}{\infty}{}}=\chi_{\intervallo{0}{\infty}{}}(A)\quad \text{and}\quad P_{-}:=P_{\intervallo{-\infty}{0}{}}=\chi_{\intervallo{-\infty}{0}{}}(A)=\chi_{\R\setminus\intervallo{0}{\infty}{}}(A)
\end{equation}
the projectors to the positive and negative half line. We want a description as a $\Upgamma$-invariant operator. The idea is to rewrite the characteristic function as a sum of the identity and the signum of $A$, which can be realized by the inverse of the positive square root $\vert A \vert := \sqrt{A^\ast A}$ in composition with $A$ itself. Roughly speaking
\begin{equation}\label{apsprojectors}
P_{\pm} =\frac{1}{2}\left(\Iop{}\pm\vert A \vert^{-1}\circ A \right) \quad .
\end{equation} 
For compact manifolds Seeley's theorem for complex powers of elliptic differential operators implies, that $\vert A \vert^{-1}$ is defined and an elliptic $\Psi$DO of order 1. In order to extend this idea to differential operators on $\Upgamma$-manifolds one needs to take a detour to the $\Psi$DO-theory on bounded geometry. We follow the reference \cite{kordyu} and refer to \cite{shubinspec} or \cite{kordyu2} for more details: suppose $M$ is a manifold of bounded geometry and $E\rightarrow M$ a vector bundle of bounded geometry. An element $A \in B\ydo{m}{\mathsf{prop}}(M,E)$ is a properly supported pseudo-differential operator of order $m \in \R$ with uniformly bounded symbol.
If the complete symbol moreover admits an uniformly asymptotic expansion $a \sim \sum_{j=0}^\infty a_{m-j}$, where each $a_{m-j}$ is a complete symbol and homogeneous of degree $(m-j)$ away from the origin in $\R^n$, one calls these pseudo-differential operators \textit{classical} and denotes it by $B\ydo{m}{\mathsf{prop,cl}}(M,E)$. Operators with complete uniformly bounded symbol, which map between $C^\infty_\comp(M,E)$ and $C^\infty(M,E)$ and are continuous mappings between $H^s(M,E)$ and $H^\infty(M,E)=\bigcap_{t \in \R}H^t(M,E)$ for any $s\in \R$ (i.e. they map between any two Sobolev spaces), will be denoted by $B\ydo{-\infty}{}(M,E)$. As in classical $\Psi$DO-theory any uniformly bounded pseudo-differential operator of order $m \in \R$ can be represented as a sum of uniformly bounded pseudo-differential operators, where one is properly supported with the same order and the other is smoothing: for $A \in B\ydo{m}{}(M,E)$ exists a $\bar{A}\in B\ydo{m}{\mathsf{prop}}(M,E)$, such that
\begin{equation*}
(A-\bar{A}) \in  B\ydo{-\infty}{}(M,E)\quad;
\end{equation*}
if furthermore $\bar{A} \in B\ydo{m}{\mathsf{prop,cl}}(M,E)$, write $A \in B\ydo{m}{\mathsf{cl}}(M,E)$. Let $A \in B\ydo{m}{\mathsf{prop,cl}}(M,E)$ be formally self-adjoint and uniformly elliptic, it is shown in \cite{kordyu}, that the resolvent $(A-\lambda \id{})$ for certain $\lambda \in \C$ is defined and an element in $B\ydo{m}{\mathsf{cl}}(M,E)$, such that for each function $f$, extendable as entire function on the whole complex plane, the operator function $f(A)$ is defined by
\begin{equation}\label{cauchyoperator}
f(A):=\frac{\Imag}{2\uppi}\int_\gamma f(\lambda)(A-\lambda \id{})^{-1} \differ \lambda \quad,
\end{equation}
where $\gamma$ is a Hankel contour; see formula (6) in section 2 of \cite{kordyu}. As a consequence complex powers of $A$ become well defined: let $A \in B\ydo{m}{\mathsf{prop,cl}}(M,E)$ be formally self-adjoint and uniformly elliptic. For any $z \in \C$, such that $\Rep{z}<0$, the operator $A^z$ is defined by \cref{cauchyoperator} with $f(\lambda)=\lambda^z$ with the branch chosen in such a way, that $\lambda^z=\expe{z\log{\lambda}}$ for $\lambda >0$. With $A^z=A^{z-k}\circ A^k$ for $k \in \N_0$, such that $\Rep{z}<k$, one can extend this to all complex powers. The corollary after Proposition 1 in the \cite{kordyu} ensures, that $A^z \in B\ydo{m\Rep{z}}{\mathsf{cl}}(M,E)$. 

\newpage
For the purpose of this detour any $\Upgamma$-invariant (pseudo-) differential operator corresponds to a uniformly bounded (pseudo-) differential operator on a $\Upgamma$-manifold, which acts on a $\Upgamma$-vector bundle and commutes with the left representation of $\Upgamma$. Under this point of view the space $B\ydo{-\infty}{}(M,E)$ corresponds to s-smoothing operators $S\ydo{-\infty}{\Upgamma}(M,E)$ and $B\ydo{m}{(\mathsf{cl})}(M,E)$ to (classical) s-regular operators $S\ydo{m}{\Upgamma,(\mathrm{cl})}(M,E)$. A formally self-adjoint and elliptic element $A \in \Diff{m}{\Upgamma}(M,E)\subset \ydo{m}{\Upgamma,\mathsf{prop,cl}}(M,E)$, $m \in \N_0$, satisfies $A^\ast\circ A=A^2 \in \ydo{2m}{\Upgamma,\mathsf{prop,cl}}(M,E)$. Since $(A^2-\lambda \id{})$ commutes with the left representation for all $\lambda \in \C$, the resolvent does as well.
Thus according to \cref{cauchyoperator} $f(A^2)$ commutes with $L_\Upgamma$. Interpreting $A^2$ as operator in $B\ydo{2m}{\mathsf{prop,cl}}(M,E)$ allows to identify $\absval{A}^{-1}$ with $(A^\ast A)^{-\frac{1}{2}}=(A^2)^{-\frac{1}{2}}$, such that $\absval{A}^{-1}$ commutes with the left representation. It becomes an element in $B\ydo{-1}{\mathsf{cl}}(M,E)$ and thus corresponds to a s-regular pseudo-differential operator on a $\Upgamma$-manifold. Herewith the following assertion is shown:
\begin{lem}\label{gammaapsprojectors}
Let $A \in \Diff{m}{\Upgamma}(M,E)$ be formally self-adjoint and elliptic on a $\Upgamma$-manifold $M$ with $\Upgamma$-vector bundle $E$; the projectors \cref{basicspecproj} are s-regular pseudo-differential operators of order 0, i.e.
\begin{equation*}
P_{\pm} \in S\ydo{0}{\Upgamma}(M,E) \quad.
\end{equation*} 
\end{lem}
The properly supportness of $A$ ensures, that the composition is well-defined and a classical pseudo-differential operator\bnote{f19} of order $0$.

\begin{rems}
\begin{itemize}
\item[]
\item[a)] One has $P_{\pm}=p_{\pm}+r_{\pm}$, where $p_{\pm} \in \ydo{m}{\Upgamma,\mathsf{prop,cl}}(M,E)$ and $r_{\pm}$ is s-smoothing. As the principal symbol is defined modulo smoothing terms, one has $\bm{\sigma}_m(P_{\pm})=\bm{\sigma}_m(p_{\pm})$.
\item[b)] In order to introduce (anti-) Atiyah-Patodi-Singer boundary conditions on the hypersurfaces one might need to include the eigenvalue zero. But since the indicator function of a disjoint union is the sum of indicator functions for each component and $P_{\kernel{A}}=\chi_{\SET{0}}(A)$ is s-smoothing, $P_{+}$ and $P_{\intervallro{0}{\infty}{}}$ as well as $P_{-}$ and $P_{\intervalllo{-\infty}{0}{}}$ differ by a s-smoothing operator, which can be considered as part of $r_{\pm}$.
\item[c)] Since $A$ has been chosen to be self-adjoint, the complex power $z$ in the construction of $A^z$ does not depend on the angle of the rays in the Hankel contour. The same treatment could be applied to non-self-adjoint operators $A$, since the formally self-adjoint operator $A^\ast A$ occurs in the construction above, but now each complex power $z$ depends on the choice of the ray by fixing an angle via $\mathrm{arg}(z)$. This has been used in \cite{BaerStroh2} for the index theorem with compact Cauchy boundaries, but a-priori non-self-adjoint Dirac operators.
\end{itemize}
\end{rems}


\addtocontents{toc}{\vspace{-3ex}}
\section{Fredholmness on Galois coverings}\label{chap:atiyah}
From now on $M$ is a temporal compact, globally hyperbolic, spatial $\Upgamma$-manifold with spin structure, $\spinb(M)$ a $\Upgamma$-spin bundle and $\Dirac$ is interpreted as $\Upgamma$-invariant lift of a Dirac operator $\widetilde{\Dirac}$ from the compact orbit space $\quotspace{M}{\Upgamma}$. For $n$ odd, the chiralitiy decomposition then implies, that $\spinb^{\pm}(M)$ are also $\Upgamma$-spin bundles $D_{\pm}$ are $\Upgamma$-invariant lifts from the compact orbit space. $M$ has two Cauchy boundaries $\Sigma_1$ and $\Sigma_2$, such that $\quotspace{M}{\Upgamma}$ has two Cauchy boundaries $\quotspace{\Sigma_1}{\Upgamma}$ and $\quotspace{\Sigma_2}{\Upgamma}$, which are assumed to be closed submanifolds.

\subsection{Preparations for the $\Upgamma$-Fredholmness}
\justifying
If $\timef(M)$ is compact, any spatially compact subset of $M$ is itself compact, since it is a closed subset in $M$ and $\Jlight{}(K)\subset M$ for any $K\Subset M$ becomes compact. Consequently any section with spatially compact support becomes a section with compact support, i.e.
\begin{eqnarray*}
C^l_\scomp(\timef(M),H^s_\loc(\spinb^{\pm}(\Sigma_\bullet))) &\longrightarrow& C^l(\timef(M),H^s_\loc(\spinb^{\pm}(\Sigma_\bullet))) \\
L^2_{\loc,\scomp}(\timef(M),H^s_\loc(\spinb^{\pm}(\Sigma_\bullet))) &\longrightarrow& L^2(\timef(M),H^s_\loc(\spinb^{\pm}(\Sigma_\bullet)))
\end{eqnarray*} 
for $l \in \N_0$ and $s\in \R$. The first correspondence leads to the notion of \textit{(compactly supported) finite} $s$\textit{-energy spinors} for $l=0$:
\begin{equation*}
FE^s_\comp(M,\timef,\spinb^{\pm}(M)):= C^0(\timef(M),H^s_\loc(\spinb^{\pm}(\Sigma_\bullet))) \quad.
\end{equation*}
The definition of the subspaces $FE^s_\comp(M,\timef,D)$ and $FE^s_\comp(M,\kernel{D})$ are clear from the general case. Theorem \ref{inivpwell} and Corollary \ref{homivpwell} are still valid in this setting, since the isomorphism, restricted to the subset $FE^s_\comp(M,\timef,D)\subset FE^s_\scomp(M,\timef,D)$, maps isomorphically to $L^2(\timef(M),H^s_\loc(\spinb^{\pm}(\Sigma_\bullet)))$. Taking $M$ to be a spatial $\Upgamma$-manifold, we define the following spaces for $H^s_\Upgamma \subset H^s_\loc$ and any $s \in \R$:
\begin{equation*}
\begin{split}
\text{a)}\quad & FE^s_\Upgamma(M,\timef,\spinb^{\pm}(M)):= C^0(\timef(M),H^s_\Upgamma(\spinb^{\pm}(\Sigma_\bullet))) \quad; \\
\text{b)}\quad & FE^s_\Upgamma(M,\timef,D_{\pm}):=\SET{u \in FE^s_\Upgamma(M,\timef,\spinb^{\pm}(M))\,\vert\,D_{\pm}u\in L^2(\timef(M),H^s_\Upgamma(\spinb^{\mp}(\Sigma_\bullet)))} \\
\text{c)}\quad & FE^s_\Upgamma(M,\kernel{D_{\pm}}):=\SET{u \in FE^s_\Upgamma(M,\timef,\spinb^{\pm}(M))\,\vert\,D_{\pm}u=0} \quad .
\end{split}
\end{equation*}
The seminorm on $FE^s_\Upgamma(M,\timef,\spinb^{\pm}(M))$ is defined as in \cref{snclk}, where the seminorm of local Sobolev sections is replaced by the norm for $\Upgamma$-Sobolev spaces:
\begin{equation}\label{snclkgamma}
\Vert u \Vert_{I,K,l,s,\Upgamma}:= \max_{k \in [0,l]\cap \N_0} \max_{t \in I} \Vert (\nabla_t)^{k} u\Vert_{H^s_\Upgamma(E\vert_{\Sigma_t})}
\end{equation}
for any vector bundle $E$, such that $E\vert_{\Sigma_t}$ is a $\Upgamma$-vector bundle. The fact, that the hypersurfaces are Galois coverings, implies, that the left unitary action on a section of a $\Upgamma$-vector bundle is given by
\begin{equation*}
(L_\upgamma u)(t,x):= \pi_\Upgamma u(t,\upgamma^{-1}x)\quad,
\end{equation*}
where $\pi_\Upgamma\vert_{(t,p)}\,:\,E_{(t,p)}\rightarrow E_{(t,\upgamma p)}$. This induces a left unitary action on $L^2(\timef(M),H^s_\loc(\spinb^{\pm}(\Sigma_\bullet)))$ and this Hilbert bundle over Hilbert $\Upgamma$-modules becomes itself a free Hilbert $\Upgamma$-module. Note, that
\begin{equation*}
L^2(\timef(M),H^0_\Upgamma(\spinb^{\pm}(\Sigma_\bullet)))=L^2(\timef(M),L^2_\Upgamma(\spinb^{\pm}(\Sigma_\bullet)))=L^2_\Upgamma(\spinb^{\pm}(M)) \quad.
\end{equation*}
After this clarification a well-posedness result for the inhomogeneous Dirac equation for an $\Upgamma$-invariant Dirac operator $D$ can be proven.
\begin{prop}\label{inivpwellgamma}
For a fixed $t \in \timef(M)$ with $M$ a temporal compact globally hyperbolic spatial $\Upgamma$-manifold with spin structure and any $s\in \R$ the map 
\begin{equation*}
\mathsf{res}_t \oplus D \,\,:\,\, FE^s_\Upgamma(M,\timef,D) \,\,\rightarrow\,\, H^s_\Upgamma(\spinb^{+}(\Sigma_t))\oplus L^2(\timef(M),H^s_\Upgamma(\spinb^{-}(\Sigma_\bullet)))
\end{equation*}
is an isomorphism of Hilbert $\Upgamma$-modules.
\end{prop}
\begin{proof}
The continuity of $D$ as map from $FE^s_\Upgamma(M,\timef,D)$ to $L^2(\timef(M),H^s_\Upgamma(\spinb^{-}(\Sigma_\bullet)))$ follows by construction of the domain. The continuity of the restriction is given by the following modified argument: For all $t\in \timef(M)$ the diffeomorphism between $M$ and the product manifold $\timef(M) \times \Sigma$ by Theorem \ref{theo22-1} implies, that a time independent $\Upgamma$-invariant partition of unity $\SET{\phi_{i,\upgamma}}_{\substack{i \in I \\ \upgamma \in \Upgamma}}$, subordinated to a covering of $\Sigma$, for each $\Upgamma$-hypersurface can be chosen. Thus every slice in $\SET{\Sigma_t}_{t \in I}$ has the same partition of unity. As in the the proof of Theorem \ref{inivpwell} with $\mathcal{K} \cap \Sigma_t$ replaced by $K(t,i,\upgamma)=\mathcal{K} \cap \Sigma_t \cap \supp{\phi_{i,\upgamma}}$ for each $i \in I$, $\upgamma \in \Upgamma$ and $\mathcal{K}$ spatially compact we obtain 
\begin{eqnarray*}
\norm{\rest{t}u}{H^s_\Upgamma(\spinb^{+}(\Sigma_t))}^2 
&\leq& \max_{\tau \in \timef(M)}\sum_{\substack{i\in I \\ \upgamma \in \Upgamma}}\norm{\phi_{i,\upgamma}\rest{\tau}u}{H^s(K(\tau,i,\upgamma),\spinb^{+}(\Sigma_\tau))}^2 \\
&=& \max_{\tau \in \timef(M)}\SET{\norm{\rest{\tau}u}{H^s_{\Upgamma}(\spinb^{+}(\Sigma_\tau))}^2} \stackrel{\cref{snclkgamma}}{\leq } \Vert u \Vert_{\timef(M),\mathcal{K},0,s,\Upgamma}^2\quad.
\end{eqnarray*}
The continuous inclusion of $C^0_{\mathcal{K}}$ into $C^0_\comp$ then leads to the wanted feature. The second estimate in the proof of Theorem \ref{inivpwell} can be modified after introducing a suitable energy in the $\Upgamma$-setting: define the $\Upgamma$\textit{-s-energy} as square of the $\Upgamma$-Sobolev norm:
\begin{equation*}
\mathcal{E}_{s,\Upgamma}(u,\Sigma_t) = \norm{u}{H^s_\Upgamma(\spinb(\Sigma_t))}^2=\sum_{\substack{i \in I\\\upgamma \in \Upgamma}}\norm{\phi_{i,\upgamma} u}{H^s(K(t,i,\upgamma),\spinb(\Sigma_t))}^2= \sum_{\substack{i \in I\\\upgamma \in \Upgamma}} \mathcal{E}_{s}(\phi_{i,\upgamma}u,\Sigma_t)\quad.
\end{equation*} 
Let $\phi:=\phi_{i,\upgamma}$ and $\tilde{u}=\phi u$ as abbreviations. The estimation of the energy $\mathcal{E}_{s}(\phi_{i,\upgamma}u,\Sigma_t)$ in Proposition \ref{enesttheoremform} can be performed up to the fourth line, where one makes use of the time independence of the partition of unity: 
\begin{eqnarray*}
\frac{\differ}{\differ t}\mathcal{E}_s(\tilde{u},\Sigma_{t}) 
&\leq& n\dscal{1}{L^2(\spinb(\Sigma_t))}{H_t \Lambda^s_t \tilde{u}}{\Lambda^s_t \tilde{u}}+c_2\norm{\tilde{u}}{H^s(\spinb(\Sigma_t))}^2-2 \Re\mathfrak{e}\left\lbrace\dscal{1}{H^s(\spinb(\Sigma_t))}{\tilde{u}}{\phi \nabla_{\upnu} u} \right\rbrace \\
&=& n\dscal{1}{L^2(\spinb(\Sigma_t))}{H_t \Lambda^s_t \tilde{u}}{\Lambda^s_t \tilde{u}}+c_2\norm{\tilde{u}}{H^s(\spinb(\Sigma_t))}^2\\
&+&2 \Re\mathfrak{e}\left\lbrace\dscal{1}{H^s(\spinb(\Sigma_t))}{\tilde{u}}{\upbeta \phi D u} \right\rbrace + 2 \Re\mathfrak{e}\left\lbrace\dscal{1}{H^s(\spinb(\Sigma_t))}{\tilde{u}}{\phi B_t u} \right\rbrace \\
&\stackrel{(\ast)}{\leq}& n\absval{\dscal{1}{L^2(\spinb(\Sigma_t))}{H_t \Lambda^s_t \tilde{u}}{\Lambda^s_t \tilde{u}}}{}+c_2\norm{\tilde{u}}{H^s(\spinb(\Sigma_t))}^2\\
&+&2 \Re\mathfrak{e}\left\lbrace\dscal{1}{L^2(\spinb(\Sigma_t))}{\Lambda^s_t\tilde{u}}{\upbeta \Lambda^s_t \phi D u} \right\rbrace + 2 \Re\mathfrak{e}\left\lbrace\dscal{1}{L^2(\spinb(\Sigma_t))}{\Lambda^s_t \tilde{u}}{ \Lambda^s_t \phi B_t u} \right\rbrace \\
&\stackrel{(\ast\ast)}{\leq}&  n\norm{H_t \Lambda^s_t \tilde{u}}{L^2(\spinb(\Sigma_t))}^2+c_3\norm{\tilde{u}}{H^s(\spinb(\Sigma_t))}^2\\
&+& \norm{\upbeta \Lambda^s_t \phi D u}{L^2(\spinb(\Sigma_t))}^2 + \norm{\Lambda^s_t \phi B_t u}{L^2(\spinb(\Sigma_t))}^2 \\
&\leq & c_4\norm{\tilde{u}}{H^s(\spinb(\Sigma_t))}^2+\norm{\phi D u}{H^s(\spinb(\Sigma_t))}^2 + \norm{\phi B_t u}{H^s(\spinb(\Sigma_t))}^2\quad,
\end{eqnarray*}
where in $(\ast)$ Lemma \ref{lemenest1} (b) and in $(\ast\ast)$ the Cauchy-Schwarz inequality and polarization estimation has been used. The last step contains a $\Psi$DO-estimation with $H_t \Lambda^s_t \in \ydo{s}{}(\Sigma_t,\End(\spinb^{+}(\Sigma_t)))$ and the isometry of $\upbeta$ in \cref{isombeta}. Since $B_t$ is properly supported, $B_t u\vert_{\Sigma_t} \in H^s_\loc(\spinb^{+}(\Sigma_t))$ for $u\vert_{\Sigma_t} \in H^{s+1}_\loc(\spinb^{+}(\Sigma_t))$. Summing over all covering balls and $\Upgamma$-actions gives
\newpage
\begin{equation*}
\sum_{\substack{i \in I\\\upgamma \in \Upgamma}} \norm{\phi_{i,\upgamma} B_t u}{H^s(\spinb(\Sigma_t))}^2 \leq c \sum_{\substack{i \in I\\\upgamma \in \Upgamma}} \norm{\phi_{i,\upgamma} u}{H^{s+1}(\spinb(\Sigma_t))}^2 \leq \tilde{c} \mathcal{E}_{s,\Upgamma}(u,\Sigma_{t}) \quad.
\end{equation*}
The first inequality comes from the local boundness of PDO/$\Psi$DO (see for instance Theorem 1 in 3.9 of \cite{shub}) and the second inequality from the continuous inclusion of Sobolev spaces. One finally yields
\begin{equation*}
\frac{\differ}{\differ t}\mathcal{E}_s(\tilde{u},\Sigma_{t}) \leq c_5\mathcal{E}_{s,\Upgamma}(u,\Sigma_t)+\norm{D u}{H^s_\Upgamma(\spinb(\Sigma_t))}^2
\end{equation*}
and repeating all further steps from the proof as well as from Corollary \ref{corenest1} shows, that for $\tau \in \timef(M)$, $K \subset M$ compact and $s \in \R$ there exists a $C>0$, such that
\begin{equation}\label{enestgamma}
\mathcal{E}_{s,\Upgamma}(u,\Sigma_{t})\leq C \left(\mathcal{E}_{s,\Upgamma}(u,\Sigma_{\tau})+ \norm{D u}{\timef(M),\Jlight{}(K),s,\Upgamma}^2\right) 
\end{equation}
is valid for all $t \in \timef(M)$ (temporal compactness) and for all $u \in FE^{s+1}_{\Upgamma}(M,\timef,D)$, such that $Du \in FE^s_{\Upgamma}(M,\timef,\spinb^{-}(M))$ and $\supp{u}\subset \Jlight{}(K)$. In comparison to the statements in Proposition \ref{enesttheoremform} and Corollary \ref{corenest1} the spatial compact support has been ensured by the $\Upgamma$-invariant partition of unity, which was used in order to carry over the proof up to these modifications. Hence the constant $C$ depends on the projection of the support onto $\quotspace{\Sigma}{\Upgamma}$, but since this base is compact by our general preassumption, the constant $C$ is now independent of the support of $u$. A $\Upgamma$-version of Corollary \ref{corenest1} and a uniqueness like in Corollary \ref{corenest2} then follows with identical arguments and can be used for the well-posedness of the Cauchy problem on $\Upgamma$-manifolds.\\
\\
As in the referred proof any finite energy section in $FE^s_\Upgamma(M,\timef(M),D)$ can be estimated with an initial value $u_0 \in H^s_\Upgamma(\Sigma_1,\spinb^{+}(\Sigma_1))$ and inhomogeneity $f=Du \in L^2_\Upgamma(\timef(M),H^s_\Upgamma(\spinb^{-}(\Sigma_\bullet)))$ by using the above shown energy estimate \cref{enestgamma} 
with $t=t_1$ as initial time:
\begin{equation*}
\Vert u \Vert_{\timef(M),\Jlight{}(K),0,s,\Upgamma}^2 \leq \max_{\tau \in \timef(M)}\SET{\mathcal{E}_{s,\Upgamma}(u,\Sigma_{\tau})} \leq C \left(\mathcal{E}_{s,\Upgamma}(u_0,\Sigma_{1})+ \norm{f}{\timef(M),\Jlight{}(K),s,\Upgamma}^2\right) \quad .
\end{equation*}
The rest of the proof works analogously, such that one yields an isomorphism between the topological vector spaces $FE^s_\Upgamma(M,\timef,D)$ and $H^s_\Upgamma(\spinb^{+}(\Sigma_t))\oplus L^2_\Upgamma(\timef(M),H^s_\Upgamma(\spinb^{-}(\Sigma_\bullet)))$. Since the latter carries a Hilbert space structure, $FE^s_\Upgamma(M,\timef,D)$ consequently does as well by this isomorphism. The direct sum of Hilbert $\Upgamma$-modules is again a Hilbert $\Upgamma$-module by Lemma \ref{helpinglemma}. $FE^s_\Upgamma(M,\timef,D)$ inherits the same left unitary action as subset of $L^2(\timef(M),H^s_\Upgamma(\spinb^{+}(M))$, hence it is a general Hilbert $\Upgamma$-module. Proposition and Corollary 1 in subsection 2.16 of \cite{shub} ensure, that the Hilbert space isomorphism is a unitary isomorphism between these general Hilbert $\Upgamma$-modules and consequently $FE^s_\Upgamma(M,\timef,D)$ a (projective) Hilbert $\Upgamma$-module.
\end{proof}

As a corollary one deduces, that also $FE^s_\Upgamma(M,\kernel{D})$ is a Hilbert $\Upgamma$-module.
\begin{cor}\label{homivpwellgamma}
For a fixed $t \in \timef(M)$ with $M$ a temporal compact globally hyperbolic spatial $\Upgamma$-manifold and any $s\in \R$ the map
\begin{equation*}
\mathsf{res}_t  \,\,:\,\, FE^s_\Upgamma(M,\kernel{D}) \,\,\rightarrow\,\, H^s_\Upgamma(\spinb^{+}(\Sigma_t))
\end{equation*}
is an isomorphism of Hilbert $\Upgamma$-modules.
\end{cor}

\subsection{$\Upgamma$-Fredholmness of $Q$ and its spectral entries}
\justifying
The $\Upgamma$-Fredholmness of the Dirac operator $D$ under (anti-) APS boundary conditions is going to be related to the $\Upgamma$-Fredholmness of the spectral parts of the evolution operator $Q$. This agrees with the procedure, used in \cite{BaerStroh} for closed spacelike Cauchy boundaries. 
\begin{lem}
\begin{itemize}
\item[]
\item[a)] $A_t \in \Diff{1}{\Upgamma}(\Sigma_t,\spinb^{\pm}(\Sigma_t))$ for each fixed $t \in \timef(M)$.
\item[b)] $Q \in \FIO{0}_\Upgamma(\Sigma_{1},\Sigma_{2};\mathsf{C}'_{1\rightarrow 2};\Hom(\spinb^{+}(\Sigma_1),\spinb^{+}(\Sigma_2)))$ properly supported as in Proposition \ref{Qfourier} and maps  $H^s_\Upgamma(\spinb^{+}(\Sigma_1))$ to $H^s_\Upgamma(\spinb^{+}(\Sigma_2))$ isomorphically.
\item[c)] $Q \in \mathscr{F}_\Upgamma(L^2_\Upgamma(\spinb^{+}(\Sigma_1)),L^2_\Upgamma(\spinb^{+}(\Sigma_2)))$ with $\Index_\Upgamma(Q)=0$.
\end{itemize}
\end{lem}
\begin{proof}
Claim a) is clear from \cref{dirachyppos} and the preassumption $L_\upgamma\Dirac=\Dirac L_\upgamma$. The last statement is based on the following observation: as in section \ref{chap:feynman} the well-posedness of the homogeneous Dirac equation on a $\Upgamma$-manifold motivates to define a Dirac-wave evolution operator as isomorphism
\begin{equation*}
\hat{Q}(t_2,t_1)\,:\,H^s_\Upgamma(\spinb^{+}(\Sigma_1))\,\,\rightarrow\,\,H^s_\Upgamma(\spinb^{+}(\Sigma_2))\quad,
\end{equation*}
defined by $\hat{Q}(t_2,t_1):=\rest{t_2}\circ(\rest{t_1})^{-1}$ . If $u \in FE^s_{\Upgamma}(M,\kernel{D})$, then $L_\upgamma u \in FE^s_{\Upgamma}(M,\kernel{D})$, since $D$ commutes with the left action for all $\upgamma \in \Upgamma$ and $FE^s_{\Upgamma}(M,\kernel{D})$ is a Hilbert $\Upgamma$-module according to Corollary \ref{homivpwellgamma}. The bijective restriction operator commutes with $L_\upgamma$ for all times, such that 
\begin{equation*}
L_\upgamma \hat{Q} \rest{t_1}u=L_\upgamma \rest{t_2}u = \rest{t_2}L_\upgamma u= \hat{Q}\rest{t_1}L_\upgamma u=\hat{Q} L_\upgamma \rest{t_1} u \quad.
\end{equation*}
The proofs of Lemma \ref{propsofprop} c) and Proposition \ref{Qfourier} carry over literally, such that $\hat{Q}(t_2,t_1)$ 
becomes a properly supported zeroth order Fourier integral operator with stated canonical relation and is unitary for $s=0$. The commuting with $L_\upgamma$ concludes the proof of b). Unitarity of $\hat{Q}$ implies $\Upgamma$-Fredholmness with index $\Index_{\Upgamma}(\hat{Q})=0$. Since all other properties are shown to be valid as well, we do not discern between this operator and the one in section \ref{chap:feynman} and identify them. 
\end{proof}
The (anti-) APS boundary conditions in the $\Upgamma$-setting induce $L^2$-splittings as well: 
\begin{equation*}
\begin{array}{ccl}
L^2_\Upgamma(\spinb^{\pm}(\Sigma_1))&=&L^2_{\Upgamma,\intervallro{0}{\infty}{}}(\spinb^{\pm}(\Sigma_1))\oplus L^2_{\Upgamma,\intervallo{-\infty}{0}{}}(\spinb^{\pm}(\Sigma_1)) \\
L^2_\Upgamma(\spinb^{\pm}(\Sigma_2))&=&L^2_{\Upgamma,\intervallo{0}{\infty}{}}(\spinb^{\pm}(\Sigma_2))\oplus L^2_{\Upgamma,\intervalllo{-\infty}{0}{}}(\spinb^{\pm}(\Sigma_2))
\end{array}
\end{equation*} 
with $L^2_{\Upgamma,I}(\spinb^{\pm}(\Sigma_j)):=P_I(t_j)\left(L^2_\Upgamma(\spinb^{\pm}(\Sigma_j))\right)$. A similar splitting is induced for $\Upgamma$-Sobolev spaces: $H^s_{\Upgamma,I}(\spinb^{\pm}(\Sigma_j)):=P_I(t_j)\left(H^s_{\Upgamma}(\spinb^{\pm}(\Sigma_j))\right)$. The essential self-adjointness of each $A_j$ for $j \in \SET{1,2}$ is given by Property (4) in Remarks \ref{remarkssmoothing}. The projectors are well defined and s-regular pseudo-differential operators of order 0 by Lemma \ref{gammaapsprojectors}. According to \cref{L2projector} for positive chirality these spaces are closed, since the projectors are continuous as maps between $L^2$-spaces, hence the range satisfies $\range{P_I(t_j)}=\kernel{\Iop{}-P_I(t_j)}$ and is closed for all $j\in\SET{1,2}$ and $I \subset \R$. Because closed $\Upgamma$-invariant subspaces of Hilber $\Upgamma$-modules are itself preojective Hilbert $\Upgamma$-submodules, all parts in these orthogonal splittings do as well. \\
\\
The spectral decomposition of $Q$ can be seen as sum of a $\Upgamma$-invariant Fourier integral operator and a s-smoothing operator. The spectral entries satisfy the following important properties.
\begin{prop} 
\noindent \label{propQfred} 
\begin{itemize}
\item[a)] $Q_{\pm\pm} \in \FIO{0}_\Upgamma(\Sigma_1,\Sigma_2;\mathsf{C}'_{1\rightarrow 2};\Hom(\spinb^{+}(\Sigma_1),\spinb^{+}(\Sigma_2)))$ ; 
\item[b)]  $Q_{\pm\mp} \in \FIO{-1}_\Upgamma(\Sigma_1,\Sigma_2;\mathsf{C}'_{1\rightarrow 2};\Hom(\spinb^{+}(\Sigma_1),\spinb^{+}(\Sigma_2)))$ ;
\item[c)] $Q_{\pm\pm} \in \mathscr{F}_\Upgamma(L^2_\Upgamma(\spinb^{+}(\Sigma_1)),L^2_\Upgamma(\spinb^{+}(\Sigma_2)))$ with indices $\Index_\Upgamma(Q_{++})=-\Index_\Upgamma(Q_{--})$.
\end{itemize}
\end{prop} 
\begin{proof}
Let $P_{\pm}(t_j)$ be either $P_{\intervallo{0}{\infty}{}}(t_j)$ for the plus case or $P_{\intervallo{-\infty}{0}{}}(t_j)$ for the minus case; from Lemma \ref{gammaapsprojectors} these projectors are elements in $S\ydo{0}{\Upgamma}(\Sigma_j,\spinb^{+}(\Sigma_j))$, thus they can be decomposed as $P_{\pm}(t_j)=p_{\pm}(t_j)+r_{\pm}(t_j)$, where $r_{\pm}(t_j)\in S\ydo{-\infty}{\Upgamma}(\Sigma_j,\spinb^{+}(\Sigma_j))$ and $p_{\pm}(t_j) \in \ydo{0}{\Upgamma,\mathsf{cl}}(\Sigma_j,\spinb^{+}(\Sigma_j))$ are properly supported. The other projectors in the APS boundary conditions can be decomposed as well:
\begin{equation*}
P_{\pm}(t_j)=p_{\pm}(t_j)+\tilde{r}_{\pm}(t_j)\quad,
\end{equation*}
where $\tilde{r}_{\pm}(t_j)\in S\ydo{-\infty}{\Upgamma}(\Sigma_j,\spinb^{+}(\Sigma_j))$ is the sum of $r_{\pm}(t_j)$ and $\chi_{\SET{0}}(A_j)$.\\
\\
Regarding this the spectral entries of $Q=Q(t_2,t_1)$ in \cref{Qmatrix} and following can be split up into a sum of a properly supported $\Upgamma$-invariant FIO and a s-smoothing operator:  
\begin{eqnarray*}
Q_{\pm\pm}(t_2,t_1)&=& q_{\pm\pm}(t_2,t_1)+ R_{\pm\pm}(t_2,t_1)\quad \text{with} \quad q_{\pm\pm}(t_2,t_1):=p_{\pm}(t_2)\circ Q\circ p_{\pm}(t_1) \\
\text{and}&& R_{\pm\pm}(t_2,t_1):=\tilde{r}_{\pm}(t_2)\circ Q \circ p_{\pm}(t_1)+p_{\pm}\circ Q\circ r_{\pm}(t_1)+\tilde{r}_{\pm}(t_2)\circ Q \circ r_{\pm}(t_1)\quad; \\
Q_{\pm\mp}(t_2,t_1)&=& q_{\pm\mp}(t_2,t_1)+ R_{\pm\mp}(t_2,t_1)\quad \text{with} \quad q_{\pm\mp}(t_2,t_1):=p_{\pm}(t_2)\circ Q\circ p_{\mp}(t_1) \\
\text{and} && R_{\pm\mp}(t_2,t_1):=\tilde{r}_{\pm}(t_2)\circ Q \circ p_{\mp}(t_1)+p_{\pm}\circ Q\circ r_{\mp}(t_1)+\tilde{r}_{\pm}(t_2)\circ Q \circ r_{\mp}(t_1)\quad.
\end{eqnarray*} 
The triple compositions $q_{\pm\pm}$ and $q_{\pm\mp}$ of properly supported operators are well defined and in terms of FIO compositions again properly supported and of order 0. 
The remainders $R_{\pm\pm}(t_2,t_1)$ and $R_{\pm\mp}(t_2,t_1)$ are sums of compositions between s-smoothing pseudo-differential operators and a properly supported FIO. The first two summands are s-smoothing Fourier integral operators. Since $r$ is s-smoothing, it maps between any two $\Upgamma$-Sobolev spaces according to its Definition \ref{ssmoothing}: $r:H^s_\Upgamma(\spinb^{+}(\Sigma_2)) \,\rightarrow \,H^r_\Upgamma(\spinb^{+}(\Sigma_1))$ for any two $s,r \in \R$. $Q$ maps $H^r_\Upgamma(\spinb^{+}(\Sigma_1))$ to $H^r_\Upgamma(\spinb^{+}(\Sigma_2))$, such that the second s-smoothing operator maps $H^r_\Upgamma(\spinb^{+}(\Sigma_2))$ to $H^l_\Upgamma(\spinb^{+}(\Sigma_2))$ for any $r,l \in \R$, thus $\tilde{r}\circ Q \circ r$ maps between any $\Upgamma$-Sobolev spaces, wherefore it becomes s-smoothing. Thus all three summands in $R_{\pm\pm}$ and $R_{\pm \mp}$ are s-smoothing. 
The commuting with $L_\upgamma$ for all $\upgamma \in \Upgamma$ is clear, since every part in the composition is a $\Upgamma$-invariant operator on its own right. Summing up, this shows, that $Q_{\pm\pm}$ and $Q_{\pm\mp}$ are sums of properly supported $\Upgamma$-invariant FIO and s-smoothing FIO, thus they are themselve elements of $\FIO{0}_\Upgamma(\Sigma_1,\Sigma_2;\mathsf{C}'_{1\rightarrow 2};\Hom(\spinb^{+}(\Sigma_1),\spinb^{+}(\Sigma_2)))$. Note, that the composition of the canonical relation $\mathsf{C}_{1\rightarrow 2}$ of $Q$ and the conormal bundle of the diagonal from the pseudo-differential operators gives back $\mathsf{C}_{1\rightarrow 2}$. 
The statements for the adjoints follows from property (a) in Lemma \ref{fioprop}. 
\\
\\
For b) it is left to show, that the order of the FIO is -1: consider the properly supported part of $Q_{\pm\mp}$; the principal symbols of $q_{\pm\mp}(t_2,t_1)$ are the same as the principal symbols of $Q_{\pm\mp}(t_2,t_1)$ up to smooth contributions. Thus one has
\begin{equation*}
\begin{split}
\mathpzc{q}_{\pm\mp}(x,\xi_{\pm};y,\eta)&:=\bm{\sigma}_0(q_{\pm\mp})(x,\xi_{\pm};y,\eta)=\bm{\sigma}_0(Q_{\pm\mp})(x,\xi_{\pm};y,\eta)\\
&=\bm{\sigma}_0(P_{\pm})(x,\xi_{\pm})\circ \bm{\sigma}_0(Q)(x,\xi_{\pm};y,\eta) \circ \bm{\sigma}_0(P_{\mp})(y,\eta)
\end{split}
\end{equation*} 
with $(x,\xi_{\pm})\in T^\ast_x \Sigma_2$ and $(y,\eta) \in T^\ast_y\Sigma_2$. The principal symbol of $Q$ is given as in \cref{Qprinsym} of Proposition \ref{Qfourier}. The principal symbols of the projectors follows from the one of $A_j$ for $j\in \SET{1,2}$ and from \cref{apsprojectors}.
The same calculations as in Lemma 2.6 of \cite{BaerStroh} shows, that the principal symbols of $q_{\pm\mp}(t_2,t_1)$ are vanishing. The exact sequence property in Lemma \ref{fioprop} implies, that the order is -1.
Following the proof of Theorem 25.3.1 in \cite{hoerm4} (see also e) in Lemma \ref{fioprop}) $Q^\ast_{\pm\mp}\circ Q_{\pm\mp}$ becomes a $\Upgamma$-invariant pseudo-differential operator of order 0. A precise analysis of the composition $Q^\ast_{\pm\mp}\circ Q_{\pm\mp}$ with idempotence and self-adjointness of the projections shows:
\begin{eqnarray*}
Q^\ast_{\pm\mp}\circ Q_{\pm\mp}&=& p_{\mp}(t_1)\circ Q^\ast(t_1,t_2)\circ q_{\pm\mp}(t_2,t_1) + \tilde{r}_{\mp}(t_1)\circ Q^\ast(t_1,t_2)\circ q_{\pm\mp}(t_2,t_1) \\
&& +\,\, \quad p_{\mp}(t_1)\circ Q^\ast(t_1,t_2)\circ R_{\pm\mp}(t_2,t_1) + \tilde{r}_{\mp}(t_1) \circ Q^\ast(t_1,t_2)\circ R_{\pm\mp}(t_2,t_1) \quad.
\end{eqnarray*}  
The first triple composition is properly supported, since it is a composition of properly supported operators; the three other terms are s-smoothing by the same argument as mentioned in the calculation for $Q_{\pm\pm}$ and $Q_{\pm\mp}$. Hence $Q^\ast_{\pm\mp}\circ Q_{\pm\mp} \in S\ydo{0}{\Upgamma}(\Hom(\spinb^{+}(\Sigma_{1}),\spinb^{+}(\Sigma_{1})))$. Proposition 2 in section 3.11 of \cite{shub} shows, that $Q^\ast_{\pm\mp}\circ Q_{\pm\mp}$ becomes a bounded linear map in $H^s_\Upgamma(\spinb^{+}(\Sigma_1))$ for all $s \in \R$. Moreover the principal symbol of $(Q_{\pm\mp})^\ast(t_1,t_2)\circ Q_{\pm\mp}(t_2,t_1)$ is vanishing as well, because
\begin{equation*} 
\bm{\sigma}_0((Q_{\pm\mp})^\ast(t_1,t_2)\circ Q_{\pm\mp}(t_2,t_1))=\bm{\sigma}_0((Q_{\pm\mp})^\ast(t_1,t_2))\circ \bm{\sigma}_0(Q_{\pm\mp}(t_2,t_1)))=0 \quad.
\end{equation*}
Hence $Q^\ast_{\pm\mp}\circ Q_{\pm\mp} \in S\ydo{-1}{\Upgamma}(\Hom(\spinb^{+}(\Sigma_{1}),\spinb^{+}(\Sigma_{1})))$. The first two equations in \cref{Qsys1} show, $Q^\ast_{\pm\pm}$ are initial parametrices of $Q_{\pm\pm}$ with remainders $Q^\ast_{\mp\pm}\circ Q_{\mp\pm}$ respectively. As for elliptic equations we can construct an even better parametrix, so that the order of the error becomes sufficiently negative. Ellipticity would become an important property in constructing an initial parametrix, but becomes irrelevant as unitarity of $Q$ replaced this step. Define $\mathcal{Q}_{\pm\mp}:=-Q^\ast_{\pm\mp}\circ Q_{\pm\mp}$ and the two operators (upper and lower signs refer to one of them separately)
\begin{equation*}
\mathcal{P}_{\pm\mp}:=\left(\sum_{l=0}^{N-1} (-1)^l \mathcal{Q}^l_{\pm\mp}+\mathcal{Q}_N\right)Q^\ast_{\mp\mp}
\end{equation*} 
for an $N \in \N$, $\mathcal{Q}_N \in S\ydo{-N}{\Upgamma}(\Hom(\spinb^{+}(\Sigma_{1}),\spinb^{+}(\Sigma_{1})))$ and $\mathcal{Q}^l_{\pm\mp}\in S\ydo{-l}{\Upgamma}(\Hom(\spinb^{+}(\Sigma_{1}),\spinb^{+}(\Sigma_{1})))$ for $l \in \SET{0,...,N-1}$. 
Applying on $Q_{\mp\mp}$ shows with the first two equations in \cref{Qsys1}:
\begin{eqnarray*}
\mathcal{P}_{\pm\mp} Q_{\mp\mp} &=& \left(\sum_{l=0}^{N-1} (-1)^l \mathcal{Q}^l_{\pm\mp}+\mathcal{Q}_N\right)Q^\ast_{\mp\mp}Q_{\mp\mp} = \left(\sum_{l=0}^{N-1} (-1)^l \mathcal{Q}^l_{\pm\mp}+\mathcal{Q}_N\right)\left(\Iop{}+\mathcal{Q}_{\pm\mp}\right)\\
&=&\sum_{l=0}^{N-1} (-1)^l \mathcal{Q}^l_{\pm\mp}+\mathcal{Q}_N + \mathcal{Q}_N\left(\Iop{}+\mathcal{Q}_{\pm\mp}\right) + \sum_{l=0}^{N-1} (-1)^l \mathcal{Q}^{l+1}_{\pm\mp} \\
&=& \Iop{} + (-1)^N \mathcal{Q}^N_{\pm\mp}+\mathcal{Q}_N\left(\Iop{}+\mathcal{Q}_{\pm\mp}\right)\quad,  \\
\end{eqnarray*}
hence
\begin{equation} \label{Qdiaggammatraceclass}
\left(\mathcal{P}_{+-} Q_{--}-\Iop{}\right),\left(\mathcal{P}_{-+} Q_{++}-\Iop{}\right) \in S\ydo{-N}{\Upgamma}(\Hom(\spinb^{+}(\Sigma_{1}),\spinb^{+}(\Sigma_{1}))) \quad .
\end{equation}
\newpage
If we choose $N > \dim(\Sigma)$ and recall the above shown $L^2_\Upgamma$-boundness, Corollary 1 in section 3.11 in \cite{shub} implies, that $\left(\mathcal{P}_{\pm\mp} Q_{\mp\mp}-\Iop{}\right)$ are $\Upgamma$-trace-class operators and thus $\mathcal{P}_{\pm\mp}$ suitable left parametrices. The adjoint of the first two equations in \cref{Qsys1} imply, that 
\begin{equation*}
\mirror{\mathcal{P}}_{\pm\mp}:=Q^\ast_{\mp\mp}\left(\sum_{l=0}^{N-1} (-1)^l (\mathcal{Q}^\ast_{\pm\mp})^l+\mathcal{Q}^\ast_N\right)
\end{equation*}
are suitable right parametrices for $Q_{\mp\mp}$. The relation between the $\Upgamma$-indices is a consequence of Lemma \ref{kernelisoQ}, carried over to our setting without modification. As in the proof of Lemma \ref{funcanagammalemma} we can find for any topological isomorphism between $\Upgamma$-modules like the null spaces a unitary isomorphism between these two spaces, implying
\begin{equation*}
\dim_\Upgamma\kernel{Q_{\pm\pm}}=\dim_\Upgamma\kernel{Q^\ast_{\mp\mp}}
\end{equation*}
and the index formula in Lemma \ref{propgammafred} (7) completes the proof. 
\end{proof}

\begin{rem}
The used fact, that $Q^\ast_{\pm\mp}\circ Q_{\pm\mp} \in S\ydo{0}{\Upgamma}(\Hom(\spinb^{+}(\Sigma_{1}),\spinb^{+}(\Sigma_{1})))$ implies \\ $Q^\ast_{\pm\mp}\circ Q_{\pm\mp} \in \mathscr{L}_\Upgamma(H^s_\Upgamma(\spinb^{+}(\Sigma_1),H^s_\Upgamma(\spinb^{+}(\Sigma_2))$ for all $s \in \R$ can be used to show
\begin{equation}\label{qoffdiagbound}
Q_{\pm\mp} \in \mathscr{L}_\Upgamma(L^2_\Upgamma(\spinb^{+}(\Sigma_1),L^2_\Upgamma(\spinb^{+}(\Sigma_2))\quad.
\end{equation}
On one hand, we have 
\begin{equation*}
\dscal{1}{L^2_\Upgamma(\spinb(\Sigma_1))}{Q_{\pm\mp}^\ast \circ Q_{\pm\mp}u}{u} \leq C \norm{u}{L^2_\Upgamma(\spinb(\Sigma_1))}^2
\end{equation*}
by $L^2_\Upgamma$-boundness; on the other hand the defintion of the formal adjoint operator implies
\begin{equation*}
\norm{Q_{\pm\mp} u}{L^2_\Upgamma(\spinb(\Sigma_2))}^2=\dscal{1}{L^2_\Upgamma(\spinb(\Sigma_2))}{Q_{\pm\mp}u}{Q_{\pm\mp} u} \leq C \norm{u}{L^2_\Upgamma(\spinb(\Sigma_1))}^2
\end{equation*}
and the claim is shown.
\end{rem}

\subsection{$\Upgamma$-Fredholmness of $D_{\mathrm{APS}}$ and $D_{\mathrm{aAPS}}$}
\justifying
We define finite energy spinors of the Dirac equation, which satisfy either APS or aAPS boundary conditions in the setting of non-compact manifolds coming from a Galois covering:
\begin{eqnarray*}
FE^s_{\Upgamma,\mathrm{APS}}(M,\timef,D)&:=&\SET{u \in FE^s_{\Upgamma}(M,\timef,D)\,\vert\,P_{\intervallro{0}{\infty}{}}(t_1)u\vert_{\Sigma_1}=0=P_{\intervalllo{-\infty}{0}{}}(t_2)u\vert_{\Sigma_2}}\\
FE^s_{\Upgamma,\mathrm{aAPS}}(M,\timef,D)&:=&\SET{u \in FE^s_{\Upgamma}(M,\timef,D)\,\vert\,P_{\intervallo{0}{\infty}{}}(t_2)u\vert_{\Sigma_2}=0=P_{\intervallo{-\infty}{0}{}}(t_1)u\vert_{\Sigma_1}} \quad.
\end{eqnarray*}
These spaces are $\Upgamma$-invariant, since the unitary left action commutes with both the restriction and the spectral projector. 
Proposition \ref{inivpwellgamma} and (3) in Remarks \ref{remarkssmoothing} imply in addition, that these spaces are closed subsets of $FE^s_{\Upgamma}(M,\timef,D)$ and consequently projective Hilbert $\Upgamma$-submodules on its own right. For the $\Upgamma$-Fredholmness we focus on $s=0$. The \textit{Dirac operator under (anti-) APS boundary conditions} is
\begin{eqnarray*}
D_{\mathrm{(a)APS}}\,:\,FE^0_{\Upgamma,\mathrm{(a)APS}}(M,\timef,D)\,\rightarrow L^2_\Upgamma(\spinb^{-}(M))\quad.
\end{eqnarray*}
The main piece in proving Theorem \ref{maintheoII} is given as follows:
\begin{theo}\label{indexDaAPS}
Let $M$ be a temporal compact, globally hyperbolic spatial $\Upgamma$-manifold with spin structure, $\spinb^{\pm}(M)\rightarrow M$ $\Upgamma$-spin bundles of positive/negative chirality; the $\Upgamma$-invariant Dirac operators $D_{\mathrm{APS}}$ under APS and $D_{\mathrm{aAPS}}$ under aAPS boundary conditions on the Cauchy boundary $\Upgamma$-hypersurfaces $\Sigma_{1}=\Sigma_{t_1}$ and $\Sigma_{2}=\Sigma_{t_2}$ with closed base are $\Upgamma$-Fredholm with indices
\begin{align*}
\Index_{\Upgamma}(D_{\mathrm{APS}})&=\Index_{\Upgamma}\left(\left[(P_{\intervallro{0}{\infty}{}}(t_1)\circ\rest{{t_1}})\oplus (P_{\intervalllo{-\infty}{0}{}}(t_2)\circ\rest{{t_2}})\right]\oplus D\right)=\Index_{\Upgamma}(Q_{--}(t_2,t_1)) \nonumber\\
\text{and} & \\ 
\Index_{\Upgamma}(D_{\mathrm{aAPS}})&=\Index_{\Upgamma}\left(\left[(P_{\intervallo{-\infty}{0}{}}(t_1)\circ\rest{{t_1}})\oplus (P_{\intervallo{0}{\infty}{}}(t_2)\circ\rest{{t_2}})\right]\oplus D\right)=\Index_{\Upgamma}(Q_{++}(t_2,t_1)) \quad. \nonumber
\end{align*}
\end{theo}
\begin{proof}
The steps presented here are the same as in \cite{BaerStroh}, but modified for our setting: we show, that
\begin{eqnarray*}
\mathbb{P}_{+}\oplus D &:=& \left[(P_{\intervallro{0}{\infty}{}}(t_1)\circ\rest{{t_1}})\oplus (P_{\intervalllo{-\infty}{0}{}}(t_2)\circ\rest{{t_2}})\right]\oplus D \,: \\
&& FE^0_{\Upgamma,\mathrm{APS}}(M,\timef,D)\,\rightarrow \,\left[L^2_{\Upgamma,\intervallro{0}{\infty}{}}(\spinb^{+}(\Sigma_1))\oplus L^2_{\Upgamma,\intervalllo{-\infty}{0}{}}(\spinb^{+}(\Sigma_2))\right]\oplus L^2_\Upgamma(\spinb^{-}(M)) \\
&&\text{and} \\
\mathbb{P}_{-}\oplus D &:=&\left[(P_{\intervallo{-\infty}{0}{}}(t_1)\circ\rest{{t_1}})\oplus (P_{\intervalllo{0}{\infty}{}}(t_2)\circ\rest{{t_2}})\right]\oplus D \,: \\
&& FE^0_{\Upgamma,\mathrm{aAPS}}(\timef(M),D)\,\rightarrow \,\left[L^2_{\Upgamma,\intervallo{-\infty}{0}{}}(\spinb^{+}(\Sigma_1))\oplus L^2_{\Upgamma,\intervallo{0}{\infty}{}}(\spinb^{+}(\Sigma_2))\right]\oplus L^2_\Upgamma(\spinb^{-}(M))
\end{eqnarray*}
are $\Upgamma$-Fredholm with claimed indices. To do so, Lemma \ref{funcanagammalemma} will be applied for 
\begin{align*}
\mathscr{H}&=FE^0_{\Upgamma}(M,\timef,D) \quad &\mathscr{H}_2=L^2_\Upgamma(\spinb^{-}(M)) \quad \mathscr{H}_1&=L^2_{\Upgamma,\intervallro{0}{\infty}{}}(\spinb^{+}(\Sigma_1))\oplus L^2_{\Upgamma,\intervalllo{-\infty}{0}{}}(\spinb^{+}(\Sigma_2)) \\
B&=D=D_{+} & \quad A&=\mathbb{P}_{+} 
\end{align*}
in order to prove $\Upgamma$-Fredholmness of $\mathbb{P}_{+}\oplus D$ by checking, that $C=A\vert_{\kernel{B}}=\mathbb{P}_{+}\vert_{\kernel{D}}$ is $\Upgamma$-Fredholm. This will be done by relating the analysis of $\Upgamma$-Fredholmness to the known $\Upgamma$-Fredholmness of $Q_{--}(t_2,t_1)$ from Proposition \ref{propQfred}. The use is legitimated since $\mathscr{H}_1$ and $\mathscr{H}_2$ are Hilbert $\Upgamma$-modules and $\mathscr{H}$ is a Hilbert $\Upgamma$-module by Proposition \ref{inivpwellgamma}. The boundness and surjectivity of $D_{+}$ follow from Theorem \ref{inivpwellgamma}. For $\mathbb{P}_{\pm}$ surjectivity and boundness follow from Corollary \ref{homivpwellgamma} and the facts, that projections are surjective as well as bounded, if the kernel and the range of the projection are closed $L^2_\Upgamma$-subspaces.\\
\\
The needed algebraic manipulations and arguments from Theorem 3.2 in \cite{BaerStroh} can be applied here as well, where one needs to express everything in terms of projective Hilbert $\Upgamma$-submodules, $\Upgamma$-dimensions as well as $\Upgamma$-compact and -Fredholm operators. Thus one gets
\begin{itemize}
\item $\kernel{C}\cong \kernel{Q_{--}(t_2,t_1)}\oplus \SET{0}$,
\item $\range{C}=\SET{((u_1,Q_{-+}u_1+Q_{--}w),f) \in \mathscr{H}_1\oplus\mathscr{H}_2 \,\vert\,w \in  L^2_{\Upgamma,\intervallo{-\infty}{0}{}}(\spinb^{+}(\Sigma_1)) \, , f \in \mathscr{H}_2}$\\
is closed and 
\item $\cokernel{C} \cong \kernel{Q_{--}^\ast(t_2,t_1)}\oplus\SET{0}$ .
\end{itemize}
Closedness of the range follows as in \cite{BaerStroh} with \cref{qoffdiagbound} and $Q_{--}$ being $\Upgamma$-Fredholm. Here and in the following every appearing topological isomorphism between Hilbert $\Upgamma$-modules can be considered as unitary isomorphism; see Proposition in section 2.16 of \cite{shub}. Lemma \ref{propgammadim} (5) implies, that the $\Upgamma$-dimensions coincide, such that
\begin{equation*}
\begin{split}
&\dim_\Upgamma(\kernel{C})=\dim_\Upgamma(\kernel{Q_{--}(t_2,t_1)}\oplus \SET{0}) < \infty \\
\text{and}\,\,\quad &\dim_\Upgamma(\cokernel{C})=\dim_\Upgamma(\kernel{Q_{--}^\ast(t_2,t_1)}\oplus\SET{0}) < \infty \quad,
\end{split}
\end{equation*}
since $Q_{--}(t_2,t_1)$ is $\Upgamma$-Fredholm, so $Q_{--}^\ast(t_2,t_1)$ does as well. $C$ and $A\oplus B=\mathbb{P}_{+}\oplus D$ are $\Upgamma$-Fredholm with claimed index.
The second statement follows in the same way by choosing
\begin{equation*}
\mathscr{H}_1=L^2_{\Upgamma,\intervallo{-\infty}{0}{}}(\spinb^{+}(\Sigma_1))\oplus L^2_{\Upgamma,\intervallo{0}{\infty}{}}(\spinb^{+}(\Sigma_2))\quad\text{and}\quad A=\mathbb{P}_{-}
\end{equation*}
and $\mathscr{H}$, $\mathscr{H}_2$ and $B$ as before. The conditions are still the same to legimitate the use of Lemma \ref{funcanagammalemma}. Again one relates the kernel, the image and the cokernel to the properties of a spectral entry of $Q(t_2,t_1)$; the $\Upgamma$-Fredholm-property follows from the spectral part $Q_{++}(t_2,t_1)$ up to an s-smoothing operator: the kernel of $C=\mathbb{P}_{-}\vert_{\kernel{D}}\oplus \Iop{\mathscr{H}_2}$ consists of those sections $u \in FE^0_\Upgamma(M,\kernel{D})$, such that the restriction onto the initial hypersurface $\Sigma_1$ is a section in $L^2_{\Upgamma,\intervallro{0}{\infty}{}}(\spinb^{+}(\Sigma_1))$ and its restriction onto $\Sigma_2$, denoted by $v=Qu\vert_1$ yields
\begin{eqnarray*}
0&=&P_{\intervallo{0}{\infty}{}}(t_2)(v)=(P_{\intervallro{0}{\infty}{}}(t_2)-P_0(t_2))\circ Q(t_2,t_1)\circ (P_{\intervallo{0}{\infty}{}}(t_1)+P_0(t_1))u\vert_{\Sigma_1} \\
&=& \left[Q_{++}(t_2,t_1)+\mathpzc{R}_{\,0+}(t_2,t_1)\right]u\vert_{\Sigma_1}\quad \text{with} \\
\mathpzc{R}_{\,0+}(t_2,t_1)&=& P_{\intervallro{0}{\infty}{}}(t_2)\circ Q(t_2,t_1) \circ P_0(t_1) - P_0(t_2)\circ Q(t_2,t_1) \circ P_{\intervallo{0}{\infty}{}}(t_1)\\
&& - P_0(t_2)\circ Q(t_2,t_1)\circ P_0(t_1) \quad,
\end{eqnarray*}
which is again a s-smoothing operator since the spectral projection on the null space of the hypersurface Dirac operators $P_0(t_j)$ are s-smoothing and the projectors of the boundary conditions are a sum of a properly supported pseudo-differential operator and a s-smoothing operator. $\mathpzc{R}_{\,0+}(t_2,t_1)$ is s-smoothing, because it is a sum of triple compositions of either two properly supported operators with one s-smoothing operator or vice versa. Hence
\begin{itemize}
\item $\kernel{C}\cong \kernel{Q_{++}(t_2,t_1)+\mathpzc{R}_{\,0+}(t_2,t_1)}\oplus \SET{0}$,
\item $\range{C}$ is closed and
\item $\cokernel{C} \cong \kernel{(Q_{++}(t_2,t_1)+\mathpzc{R}_{\,0+}(t_2,t_1))^\ast}\oplus\SET{0}$ .
\end{itemize}
Since $Q_{++}(t_2,t_1)$ is $\Upgamma$-Fredholm and $\mathpzc{R}_{\,0,+}(t_2,t_1) \in \mathscr{TC}_\Upgamma$ by s-smoothness, their sum is $\Upgamma$-Fredholm and
\begin{equation*}
\begin{split}
\dim_\Upgamma(\kernel{C})&=\dim_\Upgamma(\kernel{Q_{++}(t_2,t_1)+\mathpzc{R}_{\,0+}(t_2,t_1)})+\dim_\Upgamma(\SET{0}) < \infty\\
\text{and}\,\,\dim_\Upgamma(\cokernel{C})&=\dim_\Upgamma(\kernel{(Q_{++}(t_2,t_1)+\mathpzc{R}_{\,0+}(t_2,t_1))^\ast})+\dim_\Upgamma(\SET{0}) < \infty \quad.
\end{split}
\end{equation*}
$\Upgamma$-Fredholmness of $\mathbb{P}_{-}\vert_{\kernel{D}}$ follows as conclusion. The invariance of the $\Upgamma$-index under $\Upgamma$-compact pertubations shows, that $\mathbb{P}_{-}\oplus D$ is $\Upgamma$-Fredholm with claimed index.
\\
\\ 
The claim thus follows by applying (3) in Lemma \ref{funcanagammalemma} with 
\begin{align*}
\mathscr{H}&=FE^0_{\Upgamma}(M,\timef,D) \quad &\mathscr{H}_1=L^2_\Upgamma(\spinb^{-}(M)) \quad \mathscr{H}_2&=L^2_{\Upgamma,\intervallo{-\infty}{0}{}}(\spinb^{+}(\Sigma_1))\oplus L^2_{\Upgamma,\intervallo{0}{\infty}{}}(\spinb^{+}(\Sigma_2)) \\
A&=D=D_{+} & \quad B&=\mathbb{P}_{+} \quad,
\end{align*}
such that $D_{\mathrm{APS}}$ is $\Upgamma$-Fredholm with $\Index_{\Upgamma}(Q_{--}(t_2,t_1))$, and applying the same Lemma with
\begin{align*}
\mathscr{H}&=FE^0_{\Upgamma}(M,\timef,D) \quad &\mathscr{H}_1=L^2_\Upgamma(\spinb^{-}(M)) \quad \mathscr{H}_2&=L^2_{\Upgamma,\intervallo{-\infty}{0}{}}(\spinb^{+}(\Sigma_1))\oplus L^2_{\Upgamma,\intervallo{0}{\infty}{}}(\spinb^{+}(\Sigma_2)) \\
A&=D=D_{+} & \quad B&=\mathbb{P}_{-} \quad,
\end{align*}
proves the $\Upgamma$-Fredholmness of $D_{\mathrm{aAPS}}$ with the stated $\Upgamma$-index.
\end{proof}
So far we assumed, that $n$ has been odd. In case of $n$ even, the $\Upgamma$-Fredholmness $\Dirac$ follows as in Theorem \ref{indexDaAPS}, because the full Dirac operator basically coincides with $D$. Thus Theorem \ref{maintheoII} is proven for $n$ even.

\subsection{$\Upgamma$-Fredholmness of $\tilde{D}_{\mathrm{APS}}$ and $\tilde{D}_{\mathrm{aAPS}}$}
\justifying
This subsection is dedicated to the $\Upgamma$-Fredholm property of $\tilde{D}=D_{-}$ under (anti-) APS boundary conditions from \cref{apsbc} and \cref{aapsbc} on a temporal compact globally hyperbolic spatial $\Upgamma$-manifold $M$ with spin structure, which acts on spinors with negative chirality. We will basically collect all results and explain the differences in the proofs. 
\\
\\
Proposition \ref{inivpwellgamma} has the following pendant for the Dirac operator $\tilde{D}=D_{-}$:
\begin{prop}
For a fixed $t \in \timef(M)$ with $M$ a temporal compact globally hyperbolic spatial $\Upgamma$-manifold with spin structure and any $s\in \R$ the map
\begin{equation*}
\mathsf{res}_t \oplus D_{-} \,\,:\,\, FE^s_\Upgamma(M,\timef,D_{-}) \,\,\rightarrow\,\, H^s_\Upgamma(\spinb^{-}(\Sigma_t))\oplus L^2(\timef(M),H^s_\Upgamma(\spinb^{+}(\Sigma_\bullet)))
\end{equation*}
is an isomorphism of Hilbert $\Upgamma$-modules.
\end{prop} 
The proof is literally the same by changing the chiralities, replacing $D$ by $D_{-}$ and using Proposition \ref{enestdneg} for the estimates of the $\Upgamma$-s-energy. One concludes, that $FE^s_\Upgamma(M,\timef,D_{-})$ and $FE^s_\Upgamma(M,\kernel{D_{-}})$ are Hilbert $\Upgamma$-modules for all $s\in \R$. The related Dirac-wave evolution operator $\tilde{Q}(t_2,t_1)$ and its spectral entries has been defined in Definition \ref{defevopneg} and forthcoming. We collect all necessary results:
\begin{prop}
\noindent\label{Qneggammalemma}
\begin{itemize}
\item[(a)] $\tilde{Q} \in \FIO{0}_\Upgamma(\Sigma_{1},\Sigma_{2};\mathsf{C}'_{1\rightarrow 2};\Hom(\spinb^{-}(\Sigma_1),\spinb^{-}(\Sigma_2)))$ properly supported as in Proposition \ref{Qnegfourier} is an isomorphism $H^s_\Upgamma(\spinb^{-}(\Sigma_1))\,\rightarrow\,H^s_\Upgamma(\spinb^{-}(\Sigma_2))$ for all $s\in \R$ and unitary for $s=0$, hence $\Upgamma$-Fredholm with $\Index_\Upgamma(\tilde{Q})=0$.
\item[(b)] $\tilde{Q}_{\pm\pm} \in \FIO{0}_\Upgamma(\Sigma_1,\Sigma_2;\mathsf{C}'_{1\rightarrow 2};\Hom(\spinb^{-}(\Sigma_1),\spinb^{-}(\Sigma_2)))$; 
\item[(c)] $\tilde{Q}_{\pm\mp} \in \FIO{-1}_\Upgamma(\Sigma_1,\Sigma_2;\mathsf{C}'_{1\rightarrow 2};\Hom(\spinb^{-}(\Sigma_1),\spinb^{-}(\Sigma_2)))$; 
\item[(d)] $\tilde{Q}_{\pm\pm}\in \mathscr{F}_\Upgamma(L^2_\Upgamma(\spinb^{-}(\Sigma_1)),L^2_\Upgamma(\spinb^{-}(\Sigma_2)))$ with indices $\Index_\Upgamma(\tilde{Q}_{++})=-\Index_\Upgamma(\tilde{Q}_{--})$.
\end{itemize}
\end{prop}
\begin{proof}
The proof of all claims are the same up to notations. Assertion (d) follows with the same argument, but the negative chirality induces another commutation relation of parallel transport and another restricted Clifford multiplication has been used in the calculation of the principal symbols of $\hat{Q}_{\pm\mp}$ and their vanishing. From this one concludes analogously from the first equations in \cref{Qnegsys1}, that $\tilde{Q}^\ast_{\pm\pm}$ are initial parametrices for $\tilde{Q}_{\pm\pm}$ with error terms $\tilde{\mathcal{Q}}_{\mp\pm}:=-\tilde{Q}^\ast_{\mp\pm}\tilde{Q}_{\mp\pm}$ respectively. Left parametrices $\tilde{\mathcal{P}}_{\mp\pm}$ with $\Upgamma$-trace-class errors as in the proof of Proposition \ref{propQfred} can be constructed by replacing notationaly any $Q$ with $\tilde{Q}$:
\begin{equation} \label{Qnegdiaggammatraceclass}
\begin{split}
\left(\tilde{\mathcal{P}}_{+-} \tilde{Q}_{--}-\Iop{}\right) &\in S\ydo{-N}{}(\Hom(\spinb^{-}(\Sigma_{1}),\spinb^{-}(\Sigma_{1}))) \\ 
&\text{and} \\
\left(\tilde{\mathcal{P}}_{-+} \tilde{Q}_{++}-\Iop{}\right) &\in S\ydo{-N}{}(\Hom(\spinb^{-}(\Sigma_{1}),\spinb^{-}(\Sigma_{1}))) \quad .
\end{split}
\end{equation} 
Right parametrices can be constructed from the adjoint of the first two equations in \cref{Qnegsys1}. Choosing $N$ in both constructions sufficiently large shows the assertion.
\end{proof}

The $\Upgamma$-Fredholmness can be related with the $\Upgamma$-Fredholm operators $\tilde{Q}_{\pm\pm}$ and $\tilde{Q}^\ast_{\pm\pm}$. Replacing all chiralities, $D$ by $\tilde{D}=D_{-}$ and $Q(t_2,t_1)$ by $\tilde{Q}(t_2,t_1)$ everywhere leads to 
\begin{theo}\label{indexDaAPSneg}
Let $M$ be a temporal compact, globally hyperbolic spatial $\Upgamma$-manifold with spin structure, $\spinb^{\pm}(M)\rightarrow M$ the $\Upgamma$-spin bundles of positive/negative chirality, then the $\Upgamma$-invariant Dirac operators $\tilde{D}_{\mathrm{APS}}$ under APS and $\tilde{D}_{\mathrm{aAPS}}$ aAPS boundary conditions on the Cauchy boundary $\Upgamma$-hypersurfaces $\Sigma_1=\Sigma_{t_1}$ and $\Sigma_2=\Sigma_{t_2}$ with closed base are $\Upgamma$-Fredholm with indices
\begin{align*}
\Index_{\Upgamma}(\tilde{D}_{\mathrm{APS}})&=\Index_{\Upgamma}\left(\left[(P_{\intervalllo{-\infty}{0}{}}(t_1)\circ\rest{t_1})\oplus (P_{\intervallro{0}{\infty}{}}(t_2)\circ\rest{{t_2}})\right]\oplus \tilde{D}\right)=\Index_{\Upgamma}(\tilde{Q}_{++}(t_2,t_1)) \nonumber \\
\text{and} & \\
\Index_{\Upgamma}(\tilde{D}_{\mathrm{aAPS}})&=\Index_{\Upgamma}\left(\left[(P_{\intervallo{0}{\infty}{}}(t_1)\circ\rest{t_1})\oplus (P_{\intervallo{-\infty}{0}{}}(t_2)\circ\rest{{t_2}})\right]\oplus \tilde{D}\right)=\Index_{\Upgamma}(\tilde{Q}_{--}(t_2,t_1)) \quad.\nonumber
\end{align*}
\end{theo}
The proof is the same, but $\tilde{D}_{\mathrm{(a)APS}}$ now maps from $FE^0_{\Upgamma,\mathrm{(a)APS}}(M,\timef,\tilde{D})$ to $L^2_\Upgamma(\spinb^{+}(M))$ and the reversed boundary conditions lead to reversed spectral entries in the indices.\\
\\
This result and Theorem \ref{indexDaAPS} and finally concludes the Fredholmness in Theorem \ref{maintheoII} for $n$ odd. The $\Upgamma$-index is a consequence of Corollary \ref{helpingcorollary} and the fact, that 
\begin{equation*}
\left( \begin{matrix}
0 & \Iop{} \\
\Iop{} & 0
\end{matrix}\right)
\end{equation*}
is an invertible element in $\mathscr{L}_\Upgamma$:
\begin{eqnarray*}
\Index_\Upgamma(\Dirac\vert_{\mathrm{(a)APS}}) &=& \Index_\Upgamma\left(\left( \begin{matrix}
0 & \Iop{} \\
\Iop{} & 0
\end{matrix}\right) (D_{+}\vert_{\mathrm{(a)APS}}\oplus D_{-}\vert_{\mathrm{(a)APS}})\right)\\
&=&\Index_\Upgamma\left(\left( \begin{matrix}
0 & \Iop{} \\
\Iop{} & 0
\end{matrix}\right)\right) + \Index_\Upgamma\left(D_{+}\vert_{\mathrm{(a)APS}}\right) + \Index_\Upgamma\left(D_{-}\vert_{\mathrm{(a)APS}}\right) \\
&=& \Index_\Upgamma\left(D_{\mathrm{(a)APS}}\right) + \Index_\Upgamma\left(\tilde{D}_{\mathrm{(a)APS}}\right) \quad.
\end{eqnarray*}




\appendix




\addtocontents{toc}{\vspace{-3ex}}
\section{Fourier integral operators}\label{chap:app2}
This chapter gives a sufficient detailed overview about Fourier integral operators (in short FIO) on manifolds and some applications according to inital value problems. For more details, concerning this topic on its own, we refer to the standard literature \cite{hoermfio}, \cite{duistfio} and diverse chapters in \cite{hoerm1}, \cite{hoerm3} and primarily \cite{hoerm4}. 
\\
\\
For this appendix $X$ and $Y$ are manifolds with dimensions $\dim(X)=n_X(=n)$, $\dim(Y)=n_Y$ and $E\rightarrow X$ as well as $F\rightarrow Y$ are (complex) vector bundles. 
For $\alpha \in \R$ the \textit{space of} $\alpha$\textit{-densities} $\vert\Omega\vert^\alpha(X):=\vert\Omega\vert^\alpha(T^\ast X)$ is defined as density bundle over the cotangent bundle via $\vert\Omega\vert^\alpha(T^\ast X)=\vert \Omega^n(X)\vert^\alpha$ (i.e. power of $n$-volume forms). For $\alpha,\beta \in \R$ they satisfy 
\begin{equation*}
\vert\Omega\vert^\alpha(X)\otimes\vert\Omega\vert^\beta(X)\cong \vert\Omega\vert^{\alpha+\beta}\quad\text{and}\quad \vert\Omega\vert^0(X)\cong \R\quad,
\end{equation*}  
from which one concludes, that $\vert\Omega\vert^{-\alpha}(X)$ is the dual bundle to $\vert\Omega\vert^\alpha(X)$. Any dual object of a smooth section $C^\infty(X,E)$ can be interpreted as distributional section of the bundle $E\otimes\vert\Omega\vert$ by $\alpha+\beta$. Following the standard convention ($\alpha=1/2$) the density bundle is divided into two half-density bundles, $\hdense{X}:=\vert\Omega\vert^{\frac{1}{2}}(X)$, such that
\begin{eqnarray*}
C^{-\infty}(X,E\otimes\vert\Omega\vert^\frac{1}{2}(X))&:=&\left(C^\infty_\comp(X,E\otimes\hdense{X})\right)' \\
\text{and} \quad C_\comp^{-\infty}(X,E\otimes\hdense{X})&:=&\left(C^\infty(X,E\otimes\hdense{X})\right)'\quad.
\end{eqnarray*} 
Let $\mathsf{\Lambda}\subset T^\ast X$; since $(T^\ast X,\omega_X)$ is a symplectic manifold with non-degenerate and closed two-form $\omega_X$ the space $\mathsf{\Lambda}$ is a \textit{Lagrangian submanifold}, if $\omega_X$ vanishes on it and $\dim(\mathsf{\Lambda})=\frac{1}{2}\dim(T^\ast X)=n_X$. If it is a subset of $\dot{T}^\ast X$, it is said to be \textit{conic}, if it is a Lagrangian submanifold with respect to $(T^\ast X,\omega_X)$ and if 
\begin{equation*}
(x,\xi) \in \mathsf{\Lambda}\,\Rightarrow\,(0,\xi)\in T_{(x,\xi)}\mathsf{\Lambda} \quad.
\end{equation*}
Hörmander has proven, that every conic Lagrangian submanifold can be locally parametrized by a non-degenerate phase function, which are homogeneous of degree 1; see Theorem 3.1.3 and remarks afterwards in \cite{hoermfio}. 
These concepts enables to lift the notion of oscillatory integrals to manifolds via coordinate neighborhoods. In this way \textit{Lagrangian distributions of order m}, denoted by the space $I^m(X;\mathsf{\Lambda},E\otimes\hdense{X})$, are defined as those distributional sections of $C^{-\infty}(X,E\otimes\hdense{X})$, which are represented as sum of locally finite supported oscillatory integrals with non-degenerate phase functions. 
Invariance under coordinate transformation is related to invariance under the change of the phase function. For this the \textit{(Keller-) Maslov line bundle} $M_{\mathsf{\Lambda}}$ is introduced, which is a complex line bundle of the Grassmanian $\mathsf{Gr}(\dot{T}^\ast X)$. In this way the amplitudes and half-densities are invariant under a change of local coordinates up to a multiplicative factor (Maslov factor).
The principal symbol $\bm{\sigma}_{m}(K)$ of a Lagrangian distribution $K \in I^m(X;\mathsf{\Lambda},E\otimes\hdense{X})$ becomes a section of the bundle $M_{\mathsf{\Lambda}}\otimes\hdense{\mathsf{\Lambda}}\otimes\mathrm{pr}^\ast_{\mathsf{\Lambda}} E$, where $\mathrm{pr}^\ast_{\mathsf{\Lambda}} E$ denotes the lift of the vector bundle to $\mathsf{\Lambda}$ via the projection mapping $\mathrm{pr}_{\mathsf{\Lambda}}:(x,\xi)\,\mapsto\,x$. The appearing line bundle can always be trivialized, which is shown in Lemma 4.1.3 in \cite{duistfio} and in \cite{hoermfio}. 
\\
\\
Let $\mathsf{\Lambda} \subset \dot{T}^\ast(X\times Y)$ be a closed conic Lagrangian submanifold with respect to the symplectic form $\omega_X \oplus \omega_Y$ on $T^\ast(X\times Y)$. A function $K_\mathcal{F}\in I^m(X\times Y;\mathsf{\Lambda},(E\boxtimes F^\ast)\otimes\hdense{X\times Y})$ can be identified with a continuous map $\mathcal{F}$ according to the Schwartz kernel theorem; the boxed tensor product $E\boxtimes F$ denotes the \textit{external tensor product}:
\begin{equation*}
E\boxtimes F := \mathrm{pr}_1^\ast E \otimes \mathrm{pr}_2^\ast F = \Hom(\mathrm{pr}_1^\ast E,(\mathrm{pr}_2^\ast F)^\ast)
\end{equation*} 
where the pullback bundles are defined for $\mathrm{pr}_1:X\times Y \,\rightarrow\, X$ and $\mathrm{pr}_2:X\times Y \,\rightarrow\, Y$ and $(\mathrm{pr}_2^\ast F)^\ast$ is the dual bundle of the pullback bundle $(\mathrm{pr}_2^\ast F)$. 
The wave front set analysis for Lagrangian distributions relates the Lagrangian submanifold $\mathsf{\Lambda}$ to a \textit{homogeneous canonical relation} $\mathsf{C}$ from $\dot{T}^\ast X$ to $\dot{T}^\ast Y$. This is a closed conic Lagrangian submanifold in $\dot{T}^\ast(X\times Y)$ with respect to the symplectic form $\omega_X \oplus (-\omega_Y)$, contained in $\dot{T}^\ast X \times \dot{T}^\ast Y$:
\begin{equation*}
\mathsf{C}:=\SET{(x,\xi,y,\eta) \in \dot{T}^\ast X \times \dot{T}^\ast Y \,:\,(x,\xi,y,-\eta)\in\mathsf{\Lambda}}\quad.
\end{equation*}
In the common literature the corresponding Lagrangian submanifold $\mathsf{\Lambda}$ is denoted by $\mathsf{C}'$ in order to stress the homogeneous canonical relation.
\begin{defi}
Given two (complex) vector bundles $E\rightarrow X$ and $F\rightarrow Y$ of two smooth manifolds $X$ and $Y$ and a closed conic Lagrangian submanifold $\mathsf{\Lambda} \subset \dot{T}^\ast(X\times Y)$ with corresponding homogeneous canonical relation $\mathsf{C}$ from $\dot{T}^\ast X$ to $\dot{T}^\ast Y$; a \textit{Fourier integral operator of order m} from sections of $E$ to those of $F$ is an operator with kernel in $I^m(X\times Y;\mathsf{\Lambda},(E\boxtimes F^\ast)\otimes \hdense{X\times Y})$. The space of those operators will be denoted by $\FIO{m}(X,Y;\mathsf{C}';\Hom(E,F))$.
\end{defi}
The principal symbol of such an operator is a section of $M_{\mathsf{\Lambda}}\otimes\hdense{\mathsf{\Lambda}}\otimes\mathsf{j}^\ast_{\mathsf{\Lambda}}(E\boxtimes F^\ast)$, where $\mathsf{j}_{\Lambda}$ maps from $\mathsf{\Lambda}$ to $X\times Y$ via inclusion into $T^\ast(X\times Y)$ and projection on the bases. Before we list some properties, notions and concepts are needed to be introduced: the maps
\begin{eqnarray*}
\mathsf{r}&:&T^\ast X\rightarrow T^\ast X \quad , \,\,\text{such that}\,\,(x,\xi)\mapsto (x,-\xi) \\
&& \text{and} \\
\mathsf{s}&:&X\times Y \rightarrow Y \times X \quad , \,\,\text{such that}\,\,(x,y)\mapsto (y,x)
\end{eqnarray*}
are reflection in the cotangent bundle and interchanging the factors of the Cartesian product. $\mathsf{\Lambda}^{-1}:=\mathsf{r}^\ast\mathsf{s}^\ast(\mathsf{\Lambda})$ is the inverse of $\mathsf{\Lambda}$ and $\mathsf{C}^{-1}$ denotes the corresponding inverse canonical relation, where $\dot{T}^\ast Y$ and $\dot{T}^\ast X$ are interchanged. Given two homogeneous relations $\mathsf{C}_1$ from $\dot{T}^\ast Y$ to $\dot{T}^\ast Z$ and $\mathsf{C}_2$ from $\dot{T}^\ast X$ to $\dot{T}^\ast Y$, their composition is defined as
\begin{equation*}
\mathsf{C}_1\circ\mathsf{C}_2:=\SET{(x,\xi,z,\zeta)\in \dot{T}^\ast X \times \dot{T}^\ast Z\,\vert\,\exists\,(y,\eta)\in \dot{T}^\ast Y\,:\,(x,\xi,y,\eta)\in \mathsf{C}_2\quad\text{\&}\quad (y,\eta,z,\zeta)\in \mathsf{C}_1} \,\, .
\end{equation*} 
The composition is called \textit{transversal}, if $(\mathsf{C}_1\times\mathsf{C}_2)$ intersects $T^\ast X \times \mathrm{diag}(T^\ast Y)\times T^\ast Z$ transversally, i.e. 
there exists a $\widetilde{\mathsf{C}} \subset T^\ast X \times \mathrm{diag}(T^\ast Y)\times T^\ast Z$, such that
\begin{equation*}
T_p\widetilde{\mathsf{C}}=T_p(\mathsf{C}_1\times\mathsf{C}_2)+T_p(T^\ast X \times \mathrm{diag}(T^\ast Y)\times T^\ast Z)
\end{equation*}
for all points $p$ in the intersection. The composition is furthermore called \textit{proper}, if the projection $\widetilde{C}\rightarrow\dot{T}^\ast(X\times Z)$ is a proper map. 
For $n_X=n_Y$ a homogeneous canonical relation from $\dot{T}^\ast X$ to $\dot{T}^\ast Y$ will be called \textit{local canonical graph}, if 
the homogeneous canonical relation is locally the graph of a canonical transformation. In that case the homogeneous canonical relation becomes a symplectic manifold on its own right. 

\begin{lem}[see/c.f. \cite{hoerm4},\cite{hoermfio},\cite{hoermduistfio} and \cite{nazaetal}]\label{fioprop}
Given some (complex) vector bundles $E\rightarrow X$, $F,H \rightarrow Y$ and $G\rightarrow Z$ for three smooth manifolds $X$, $Y$ and $Z$ and homogeneous canonical relations $\mathsf{C}=\mathsf{C}_1$ from $\dot{T}^\ast X$ to $\dot{T}^\ast Y$ and $\mathsf{C}_2$ from $\dot{T}^\ast Y$ to $\dot{T}^\ast Z$; the following properties hold for all $m \in \R$.
\begin{itemize}
\item[(a)] (adjoint FIO) If $A \in \FIO{m}(X,Y;\mathsf{C}';\Hom(E,F))$, then 
\begin{equation*}
A^\ast \in \FIO{m}(Y,X;(\mathsf{C}^{-1})';\Hom(F^\ast,E^\ast))\quad ;
\end{equation*} 
\item[(b)] (composition) Given two operators $A_2 \in \FIO{m_2}(X,Y;\mathsf{C}_2;\Hom(E,F))$ and \\ $A_1 \in \FIO{m_1}(Y,Z;\mathsf{C}_1;\Hom(F,G))$; if at least one is properly supported and the composition $\mathsf{C}_1\circ \mathsf{C}_2$ is proper and transversal, then $A_1\circ A_2 \in \FIO{m_1+m_2}(X,Z;(\mathsf{C}_1\circ\mathsf{C}_2)';\Hom(E,G))$.
\item[(c)] (exact sequence) There is a symbol exact sequence
\begin{align*}
0\rightarrow \FIO{m-1}(X,Y;\mathsf{C}';\Hom(E,F))\hookrightarrow\FIO{m}&(X,Y;\mathsf{C}';\Hom(E,F)) \\
&\stackrel{\bm{\sigma}_m}{\longrightarrow} C^\infty(\mathsf{C},M_{\mathsf{C}'}\otimes\hdense{\mathsf{C}'}\otimes\mathrm{pr}^\ast_{\mathsf{C}'} E) \rightarrow 0
\end{align*}
for all $m \in \R$, hence the operator has at least one order less if and only if its principal symbol vanishes.
\item[(d)] (special case) If $A\in\ydo{m}{}(Y,\Hom(F,H))$, then $A\in\FIO{m}(Y, Y;N^\ast\mathrm{diag}(X);\Hom(F,H))$; 
\item[(e)] ($L^2$-regularity) If $\mathsf{C}$ is locally the graph of a locally canonical transformation from $\dot{T}^\ast X$ to $\dot{T}^\ast Y$ and $A \in \FIO{0}(X,Y;\mathsf{C}';\Hom(E,F))$, then $(A^\ast\circ A) \in \ydo{2m}{}(X,\End(E))$ and $A$ maps continuously $L^2_{\comp}(X,E\otimes\hdense{X})$ to $L^2_{\loc}(Y,F\otimes\hdense{Y})$; if 
\begin{equation*}
\sup_{(x,y)\in K}\absval{\bm{\sigma}_0(A)(x,\xi;y,\eta)} \quad \rightarrow \quad 0
\end{equation*}
for $\absval{(\xi,\eta)}\rightarrow \infty$ for all $K \Subset X \times Y$, then it maps as compact operator between $L^2(X,E\otimes\hdense{X})$ to $L^2(Y,F\otimes\hdense{Y})$. 
\item[(f)] (Sobolev regularity) If $\mathsf{C}$ is locally the graph of a locally canonical transformation from $\dot{T}^\ast X$ to $\dot{T}^\ast Y$ and $A \in \FIO{m}(X,Y;\mathsf{C}';\Hom(E,F))$, then $(A^\ast\circ A) \in \ydo{2m}{}(X,\End(E))$ and $A$ maps continuously $H^s_{\comp}(X,E\otimes\hdense{X})$ to $H^s_{\loc}(Y,F\otimes\hdense{Y})$ for all $s\in \R$. 
\end{itemize}  
\end{lem} 

Given a smooth submanifold $Y\subset X$ of a smooth manifold $X$ with codimension $k$ and inclusion $\inclus\,:\,Y\,\hookrightarrow\,X$. This gives rise to two operators\bnote{af1}: for $s\in\R$ with $s>k/2$ and local coordinates $(x',x'')$ on $X$, such that $Y=\SET{x''=0}$, one has
\begin{itemize}
\item[(1)] \textit{restriction operator}:
\begin{equation}\label{restop}
\begin{split}
\rest{Y}:=\inclus^\ast &: H^s_\loc(X)\quad\rightarrow\quad H^{s-k/2}_\loc(Y) \\
& u(x',x'') \quad\mapsto\quad u(x',0) 
\end{split}
\end{equation}
assigning each function its trace on the submanifold;
\item[(2)] \textit{corestriction operator}:
\begin{equation}\label{corestop}
\begin{split}
\inclus_\ast &: H^{-s+k/2}_\loc(Y)\quad\rightarrow\quad H^{-s}_\loc(X) \\
& v(x') \quad\mapsto\quad (v \otimes \updelta_{Y})(x',x'') 
\end{split}
\end{equation}
is the adjoint operator, which takes functions on the submanifold and lifts them to distributions on $X$ localized on $Y$ with a $\updelta$-Distribution.
\end{itemize}

These operators can be extended to act on vector-valued sections: let $E$ be a vector bundle over $X$ and $F$ a vector bundle over $Y$. The restriction operator maps sections of $E$ to sections on the restricted vector bundle $E\vert_Y$; the corestriction operator maps sections of $F$ to a vector bundle $\tilde{F}$ over $X$, such that $\tilde{F}\vert_Y=F$. Both operators are two special Fourier integral operators:
\begin{lem}\label{restcorestfio}
Consider a smooth submanifold $Y\subset X$ of a smooth manifold $X$ with codimension $k$ and inclusion map $\inclus\,:\,Y\,\hookrightarrow\,X$ as well as vector bundles $E\rightarrow X$ and $F\rightarrow Y$, then 
\begin{equation*}
\inclus^\ast \in \FIO{k/4}(X,Y;\mathsf{C}'(\inclus^\ast);\Hom(E,E\vert_Y))\quad\text{and}\quad \inclus_\ast \in \FIO{k/4}(Y,X;\mathsf{C}'(\inclus_\ast);\Hom(F,\tilde{F}))
\end{equation*}
with $\tilde{F}\vert_Y=F$ and the Lagrangian submanifolds are defined by the canonical relations
\begin{align}
\mathsf{C}(\inclus_\ast)&:=N^\ast \mathrm{graph}(\inclus):=\SET{(x,\xi,y,\eta)\in \dot{T}^\ast(X\times Y)\,\vert\, x=\inclus(y) \quad\text{and}\quad \xi\in(\differ \inclus\vert_x)^{\dagger}(\SET{\eta})} \nonumber \\ 
\text{and} & \label{relrest} \\
\mathsf{C}(\inclus^\ast)&:=
\SET{(y,\eta,x,\xi)\in \dot{T}^\ast(Y\times X)\,\vert\, (x,\xi) \in \dot{T}^\ast X \,:\, \rest{Y}^\ast(x,\xi)=(y,\eta)}\quad. \nonumber
\end{align}
Moreover both operators are properly supported.
\end{lem}
\begin{proof}
In \cite{duistfio} and \cite{nazaetal} it is shown, that the integral kernels of both operators are Lagrangian distributions of order $k/2$ with claimed canonical relations. This implies, that both operators are FIOs as stated. Since both map local Sobolev sections to local Sobolev sections, their adjoints map compactly supported Sobolev sections to compactly supported Sobolev sections with negative orders by duality, see Theorem 7.7 in \cite{shubin11}. This shows in particular, that both operators are properly supported.
\end{proof}

The rest of this appendix is dedicated to applications with respect to initial value problems: consider the Cauchy Problem 
\begin{equation}\label{cauchygen}
Pu=0\quad \text{and} \quad \rest{\Sigma}Q_ju=Q_ju\vert_{\Sigma}=g_j , 
\end{equation}
where $P \in \ydo{m}{}(M),\, Q_j \in \ydo{m_j}{}(M)$ and $g_j \in C^{-\infty}_\comp(\Sigma)$, $j\in \SET{0,...,m-1}$, are initial values on the initial hypersurface $\Sigma$ of a manifold $M$ and the $Q_j$ can be related to operators, which generates the intial values from solutions $u$. The aim is to express $u$ with solution operators, acting on the initial values $g_j$ in \cref{cauchygen}. In fact these solution operators can be constructed with Lagrangian distributions as integral kernels, i.e. they are represented as FIOs. In order to do so, one needs to impose several conditions on $P$ and $Q_j$ - see \cite{duistfio} for details: a pseudo-differential operator $P$, with bicharacteristics from its principal symbol being transversal to the initial hypersurface and $\dim(\kernel{\bm{\sigma}_m(P)(x_0,\xi_0)})=\mu$ for all $(x_0,\xi_0)\in \dot{T}^\ast(\Sigma)$, is called \textit{strictly hyperbolic} of multiplicity $\mu$ with respect to $\Sigma$. Thus $\mu$ counts the number of solutions of
\begin{equation}\label{princecond}
\bm{\sigma}_m(P)(x_0,\xi)=0\quad \text{and}\quad \xi\vert_{T_{x_0}\Sigma}=\xi_0 \quad.
\end{equation}
The transversality of the bicharacteristics is a rephrasing of the condition, that zeroes of the principal symbol at $(x_0,\xi)$ are simple and thus non-vanishing on the orthogonal complement of $T^\ast_{x_0}\Sigma$ or in other words:
the initial hypersurface is noncharacteristic and the number of simple zeroes is finite. Further assumptions has to be made in order to have well defined compositions of the solution operators with the pseudo-differential operators on the hypersurface: the relation for the solution operators
\begin{equation}\label{canrelationfio}
\mathsf{C}:=\SET{(y,\eta,x_0,\xi_0) \in \dot{T}^\ast M \times \dot{T}^\ast \Sigma\,\vert\, (x_0,\xi)\rightsquigarrow(y,\eta)\,:\, \cref{princecond}\,\,\text{holds}}
\end{equation}
guarantees a well defined composition with all $Q_j$, where $(x_0,\xi)\rightsquigarrow(y,\eta)$ denotes, that $(y,\eta)$ is connected with a bicharacteristic through $(x_0,\xi)$. It is an embedded submanifold of $\dot{T}^\ast (M\times\Sigma)$, if in addition the following conditions hold:
\begin{itemize}
\item[(1)] (transversality) every bicharacteristic curve of $P$ intersects $\Sigma$ at most once;
\item[(2)] (properness) for all $K \Subset M$ exists a $K_0 \Subset \Sigma$, such that each bicharacteristic curve, starting in $K$, hits $\Sigma$ in $K_0$;
\item[(3a)] no bicharacteristic curve starting on $\Sigma$ stays in a compact region in $M$;
\item[(3b)] (pseudo convexity) for all $K_0 \Subset \Sigma$ and $K \Subset M$ exists a $K' \Subset M$, s.t. a segment of the bicharacteristic curve, connecting one point in $\Sigma$ and one in $M$, lies inside $K'$.
\end{itemize}
The last imposed condition is related to the initial value operators $Q_j$:
\begin{itemize}
\item[(4)] the principal symbols of $Q_j$ are non-singular for any $(x_0,\xi^{(j)}(x_0,\xi_0)) \in \dot{T}^\ast \Sigma$, s.t. \cref{princecond} holds for any solution $\xi^{(j)}$, $j \in \SET{1,..., \mu}$.
\end{itemize}
Theorem 5.1.2 in \cite{duistfio} claims, that for strictly hyperbolic pseudo-differential operators there exist solution operators $\mathcal{G}_k \in \FIO{-m_k -1/4}(\Sigma,M;\mathsf{C}')$ with canonical relation in \cref{canrelationfio}, such that it maps $C^{-\infty}_\comp(\Sigma) \rightarrow C^{-\infty}(M)$ continuously, conditions (1) up to (4) are satisfied and a solution $u$ of \cref{cauchygen} can be expressed as $u=\sum_{j=1}^\mu \mathcal{G}_j g_j$. Lemma 5.1.3 and Lemma 5.1.4 imply Theorem 5.1.6 in the same reference, which states a similar result for strictly hyperbolic differential operators and initial value operators $Q_j=(\nabla_{\partial_t})^j$. From this result one observes, that $\supp{\mathcal{G}_j} \subset \SET{(p,x)\in M \times \Sigma_0\,\vert\, x \in \mathscr{D}(p)\cap \Sigma_0}$ and $\mathcal{G}_j \in \FIO{-j-1/4}(\Sigma,M;\mathsf{C}',\Hom(E\vert_{\Sigma_0},E))$, where $\supp{\mathcal{G}_j}$ refers to the support of the corresponding Schwartz kernels. $\mathscr{D}(p)$ denotes the domain of dependence, containing all points in $M$, which can be reached from $p \in M$ by all curves, which tangent vectors are the tangent vectors of bicharacteristics through $p$, pointing downward along the lower cone in $T_pM$. The result can be extended to any differential operator with scalar valued principal symbol or to any (system of) differential operators of real principal type.\\
\\
A special situation arises, if $M$ is globally hyperbolic with Lorentzian metric $\met$, $E\rightarrow M$ vector bundle, $\Sigma_0$ a complete Cauchy hypersurface and $P\in \Diff{2}{}(M,\End(E))$ is \textit{normally hyperbolic}, i.e. its principal symbol is determined by the metric in such a way, that 
\begin{equation*}
\bm{\sigma}_2(P)(x,\xi)=\pm\met_x(\xi^\sharp,\xi^\sharp) \id{E} \quad.
\end{equation*} 
The vanishing of the principal symbol corresponds to the vanishing of $\met_x(\xi^\sharp,\xi^\sharp)$ at each point $x \in M$, which is non-trivially fullfilled for $\xi^\sharp$ being a lightlike vector at $x$. As the bicharacteristic strips are determined by the Hamilton equations of the Hamilton function $1/2\met_x(\xi^\sharp,\xi^\sharp)=1/2\met_x^{\ast}(\xi,\xi)$, they can be reduced to the (co-) geodesic equations for null curves in $T^\ast M$. Thus the bicharacteristic curves (projection of the bicharacteristic strip on $M$) are given by lightlike geodesics. 
The domain of dependence corresponds to the past causal domain $\Jlight{-}(p)$ at $p$. Because $\Jlight{-}(p)\subset \Jlight{}(p)$ it itself spatially compact for any point in $M$ the intersection with any Cauchy hypersurface is compact, wherefore $\Jlight{-}(p)\cap \Sigma_0$ is compact for every $p \in M$. The global hyperbolicity of $M$ is equivalent to $M$ being causal and strongly causal, see Theorem \ref{theo2-2}. 
Proposition 4.3 and Proposition 4.4 in \cite{radz} imply, that normally hyperbolic operators are strictly hyperbolic and thus (3a) and pseudo convexity (3b) are satisfied. Global hyperbolicity implies transversality (1) according to Theorem \ref{theo21-1}, since it induces the existence of a Cauchy temporal function, which level sets are again Cauchy hypersurfaces, but with the additional property, that any inextendable causal curve crosses any $\Sigma_t$ once. Thus the lightlike geodesic curves intersect the initial slice $\Sigma_0$ at most once. Properness follows from the following argument as in \cite{radz}: take any $K \Subset M$, then $\Jlight{+}(K)\cap\Jlight{-}(K)$ is compact. This implies, that $K_0:= (\Jlight{+}(K)\cap\Jlight{-}(K))\cap \Sigma_0$ is compact, as $\Sigma_0$ is closed by completeness. Any lightlike geodesic starting in $K$ will hit $\Sigma_0$ in $K_0$. In summary all conditions from (1) to (3) on the bicharacteristics are proven. The canonical relation in \cref{canrelationfio} for the solution operators reduces to
\begin{equation}\label{canrelationnormhyp}
\mathcal{C}:=\SET{(x,\xi,y,\eta) \in \dot{T}^\ast M \times \dot{T}^\ast \Sigma_0\,\vert\, (x,\xi)\sim(y,\eta)}\quad ,
\end{equation}
where $(x,\xi)\sim(y,\eta)$ means, that for a lightlike covector $\zeta \in T^\ast_y M$ the points $(x_0,\xi)$ and $(y,\zeta)$ are one the same lightlike (co-) geodesic strip 
and $\rest{\Sigma_0}^\ast \zeta = \eta$. \\
\\
After all this details Theorem A.1 in \cite{BaerStroh} is recovered with all needed details in addition:
\newpage 
\begin{theo}\label{cauchynormhyphom}
Let $E\rightarrow M$ be a vector bundle of a globally hyperbolic Lorentzian manifold $M$ and $\Sigma_0$ a complete Cauchy hypersurface; the Cauchy problem
\begin{equation*}
Pu  =0\quad \text{in} \quad M \quad, \quad \rest{\Sigma_0}(\nabla_{\partial_t})^ju =g_j \quad\text{for}\quad j \in \SET{0,1}, 
\end{equation*}
for a normally hyperbolic operator $P\in \Diff{2}{}(M,\End(E))$ has a unique solution $u$ for every initial value $g_j \in C^\infty(\Sigma_0)$, such that 
\begin{equation*}
u =\sum_{j=0}^m \mathcal{G}_j g_j \quad,
\end{equation*}
where the solution operators $\mathcal{G}_j$ are continuous mappings from $C^\infty(\Sigma_0,E\vert_{\Sigma_0})$ to $C^\infty(M,E)$ with the properties
\begin{itemize}
\item[(1)] $\supp{\mathcal{G}_j} \subset \SET{(p,x)\in M \times \Sigma_0\,\vert\, x \in \Jlight{-}(p)\cap \Sigma_0}$ and
\item[(2)] $\mathcal{G}_j \in \FIO{-j-1/4}(\Sigma,M;\mathcal{C}';\Hom(E\vert_{\Sigma_0},E))$ with $\mathcal{C}$ as in \cref{canrelationnormhyp}.
\end{itemize} 
\end{theo}




\addtocontents{toc}{\vspace{-3ex}}
{\footnotesize
\bibliography{literature}}			

\noindent
{\footnotesize \textsc{Mathematisches Institut, Universit\"at Oldenburg, 26129 Oldenburg, Germany}\\
\emph{Email address:} \texttt{orville.damaschke@uni-oldenburg.de}



\end{document}